\documentclass{amsart}

\usepackage{amssymb}
\usepackage{amsmath}
\usepackage[unicode,linkcolor=blue,hyperindex]{hyperref}
\usepackage{url,breakurl}
\usepackage{array,multirow}
\usepackage{multicol}
\usepackage{graphicx}
\usepackage{enumerate}
\usepackage{cleveref}
\usepackage{tikz}

\usepackage{geometry}

\numberwithin{equation}{section}
\numberwithin{figure}{section}
\numberwithin{table}{section}

{\theoremstyle{remark} \newtheorem{remark}{Remark}[section]}
\newtheorem{definition}{Definition}[section]
\newtheorem{proposition}{Proposition}[section]

\newtheorem{theorem}{Theorem}[section]
\newtheorem{corollary}{Corollary}[section]
\newtheorem{lemma}{Lemma}[section]
{\theoremstyle{remark} \newtheorem{example}{Example}[section]}

\newtheorem{algorithm}{Algorithm}[section]
\usepackage{algpseudocode,algorithm,algorithmicx}

\algrenewcommand\algorithmicrequire{\textbf{Input:}}
\algrenewcommand\algorithmicensure{\textbf{Output:}}

\newcommand{\Wzop}{ \mbox{ \raisebox{7.2pt} {\tiny$\circ$} \kern-10.7pt} {W}^1_p }
\newcommand{\Wzoinfty}{\mbox{ \raisebox{7.2pt} {\tiny$\circ$} \kern-10.7pt} {W}^1_\infty }
\newcommand{\Laps}{(-\Delta)^s}

\newcommand{\REFINE}{\textsf{REFINE}}
\newcommand{\GREEDY}{\textsf{GREEDY}}

\def\calM{\mathcal{M}}
\def\calT{\mathcal{T}}
\def\calF{\mathcal{F}}

\def\calC{\mathcal{C}}
\def\calP{\mathcal{P}}

\newcommand{\Th}{{\mathcal {T}_h}}
\newcommand{\x}{\texttt{x}}

\newcommand{\R}{{\mathbb{R}}}
\newcommand{\pp}{\partial}
\newcommand{\Rd}{{\mathbb{R}^d}}
\newcommand{\vn}{\textbf{n}}
\newcommand{\tHs}{\widetilde H^s(\Omega)}
\newcommand{\eps}{\varepsilon}
\newcommand{\phii}{\varphi}
\newcommand{\vA}{\mathbf{A}}
\newcommand{\vB}{\mathbf{B}}
\newcommand{\vK}{\mathbf{K}}
\newcommand{\vU}{\mathbf{U}}
\newcommand{\vF}{\mathbf{F}}
\newcommand{\vW}{\mathbf{W}}
\newcommand{\sob}{\mbox{Sob\,}}
\newcommand{\iii}[1]{{\left\vert\kern-0.25ex\left\vert\kern-0.25ex\left\vert #1 
    \right\vert\kern-0.25ex\right\vert\kern-0.25ex\right\vert}}
\newcommand{\scalardelta}[1]{{( \kern-0.25ex ( #1 )\kern-0.25ex )}}

\usepackage{color}

\newcommand{\supp}{\textrm{supp~}}
\newcommand{\dist}{\textrm{dist}}
\newcommand{\wt}{\widetilde}
{\par\begin{samepage}%
\begin{tabbing}\ttfamily%
 \hspace*{5mm}\=\hspace{3ex}\=\hspace{3ex}\=\hspace{3ex}\=\hspace{3ex}%
\=\hspace{3ex}\=\hspace{3ex}\=\hspace{3ex}\=\hspace{3ex}\kill}%
{\end{tabbing}\end{samepage}}
\begin{document}
%%%%%%%%%%%%%%%%%%%%%%%%%%%%%%%%%%%%%%%%%%%%%%%%%%%%%%%%%%%%%%%%%%%%%%%%%%%%%%%%%%%%%%%%
\title[Fractional Elliptic Problems: Regularity and Approximation]{Fractional Elliptic Problems on Lipschitz Domains: \\ Regularity and Approximation}

\author[J.P.~Borthagaray]{Juan Pablo~Borthagaray}
\address[J.P.~Borthagaray]{Departamento de Matem\'atica y Estad\'istica del Litoral, Universidad de la Rep\'ublica, Salto, Uruguay, and Centro de Matem\'atica, Facultad de Ciencias, Universidad de la Rep\'ublica, Montevideo, Uruguay}
\email{jpb@cmat.edu.uy}
\thanks{JPB has been supported in part by Fondo Vaz Ferreira grant 2019-068.}

\author[W.~Li]{Wenbo~Li}
\address[W.~Li]{Department of Mathematics, University of Tennessee, Knoxville, TN 37996, USA.}
\email{wli50@utk.edu}
\thanks{WL has been supported in part by NSF grant DMS-2111228.}

\author[R.H.~Nochetto]{Ricardo H.~Nochetto}
\address[R.H.~Nochetto]{Department of Mathematics and Institute for
Physical Science and Technology, University of Maryland, College
Park, MD 20742, USA}
\email{rhn@math.umd.edu}
\thanks{RHN has been supported in part by NSF grant DMS-1908267.}
 
\begin{abstract}
This survey hinges on the interplay between regularity and approximation for linear and
  quasi-linear fractional elliptic problems on Lipschitz domains. For the linear Dirichlet integral Laplacian,
  after briefly recalling H\"older regularity and applications, 
  we discuss novel optimal shift theorems in Besov spaces and their Sobolev
  counterparts. These results extend to problems with finite horizon and are instrumental for the
  subsequent error analysis. Moreover,
  we dwell on extensions of Besov regularity to the fractional $p$-Laplacian, and review
  the regularity of fractional minimal graphs and stickiness. We
  discretize these problems using continuous piecewise linear finite elements and derive global
  and local error estimates for linear problems, thereby improving some existing error
  estimates for both quasi-uniform and
  graded meshes. We also present a BPX preconditioner which turns out to be robust with respect to
  both the fractional order and the number of levels. We conclude with the discretization of fractional quasi-linear
  problems and their error analysis. We illustrate the theory with several illuminating numerical experiments.
\end{abstract}
 
\maketitle

%---------------------------------------------
%---------------------------------------------
\section{Introduction and motivation}\label{sec:introduction}
%---------------------------------------------
%---------------------------------------------

Let $s \in (0,1)$ and $u \colon \Rd \to \R$ be a sufficiently smooth function. The {\it integral fractional Laplacian} of order $s$ of $u$, which we will denote by $(-\Delta)^s$, is given by
\begin{equation} \label{eq:def_Laps}
\Laps v (x) := C_{d,s} \, \mbox{p.v.} \int_\Rd \frac{v(x)-v(y)}{|x-y|^{d+2s}} \; dy , \quad C_{d,s} := \frac{2^{2s} s \Gamma(s+\frac{d}{2})}{\pi^{d/2} \Gamma(1-s)} .
\end{equation}
This is a nonlocal operator: evaluation of $\Laps v$ at some point $x$ involves a weighted and regularized average of the values of $v$ over the whole space $\Rd$. Equivalently, the fractional Laplacian is the pseudodifferential operator with symbol $|\xi|^{2s}$: for sufficiently smooth $v \colon \Rd \to \R$, it holds that
\[
\calF(\Laps v) (\xi) = |\xi|^{2s} \calF v (\xi) \quad \forall \xi \in \Rd.
\]
This characterization makes apparent the fact that $\Laps$ approaches the classical Laplacian (resp. identity operator) as $s \to 1$ (resp. $s \to 0$); we point out that the asymptotic behavior of the constant in \eqref{eq:def_Laps} $C_{d,s} \simeq s (1-s)$ is crucial for this to hold.

This work is a survey of theoretical and computational aspects of boundary value problems involving the fractional Laplacian \eqref{eq:def_Laps} and related fractional-order operators on bounded domains. Our main goal is to emphasize the interplay between regularity and approximation with continuous piecewise linear finite element methods (FEMs). The presence of algebraic boundary layers in the solution of both linear and nonlinear fractional elliptic PDEs --and even of discontinuities in minimal graph problems (stickiness)-- limits the convergence rates achievable by the FEMs discussed below. We shall not dwell on the {\em spectral} fractional Laplacian; we refer to the surveys \cite{BBNOS18,Lischke_et_al18, DElia20} for a comparison between such an operator and the integral fractional Laplacian \eqref{eq:def_Laps} and a review of other discretization techniques.

After a thorough discussion on regularity and approximation of linear equations, we turn to the regularity and numerical treatment of certain quasilinear problems. Concretely, we focus on nonlocal minimal graph problems and on the Dirichlet problem for the $(p,s)$-Laplacian. Further discussion on these and related nonlinear, nonlocal equations can be found in \cite{BucurValdinoci16,Vazquez17, BoLiNo19}.

%---------------------------------------------
\subsection{\bf Examples}
%---------------------------------------------
%
We present three key examples that constitute the basis for our discussion below.

%---------------------------------------------
\subsubsection*{\bf Random walk with long jumps} 
%---------------------------------------------
As a first motivation for the operator \eqref{eq:def_Laps}, we follow \cite{BucurValdinoci16}. Given $s \in (0,1)$, we consider the probability mass function on the nonzero natural numbers 
\[
\mbox{P}(k) := \frac{c_s}{|k|^{1+2s}}, \qquad c_s := \left( \sum_{k=1}^\infty \frac1{|k|^{1+2s}} \right)^{-1}.
\]
We consider a particle moving in the space $\Rd$ according to a time- and space-discrete process. Namely, every $\tau >0$ units of time (here, $\tau$ is a time step), the particle draws a direction $\vn \in \mathbb{S}^{d-1}$ according to a uniform probability distribution and a positive integer number $k$ according to $\mbox{P}$. Then, the particle takes a step $k h \vn$ (here, $h>0$ is a length scale): if at time $t$ the particle is at a point $x\in\Rd$, then at time $t + \tau$ it jumps to $x +k h \vn\in\Rd$; the step is discrete but $x$ is continuous.

Let $u(x,t)$ be the probability density of the particle to be at the location $x\in\Rd$ at time $t>0$. At time $t+\tau$, such a probability equals the sum of the probabilities of finding the particle somewhere else, say at $y = x-kh\vn \in\Rd$, times the probability of jumping from $y$ to $x$:
\[
u(x,t+\tau) = \frac{c_s}{\omega_{d-1}} \sum_{k = 1}^\infty \int_{\mathbb{S}^{d-1}} \frac{u(x-kh\vn, t)}{|k|^{1+2s}} \, d\sigma(\vn) ,
\]
where $\omega_{d-1} = |\mathbb{S}^{d-1}|$.
At this point, it is important that $\tau$ and $h$ have the correct scaling, namely $\tau = h^{2s}$. Subtracting $u(x,t)$ on both sides of the identity above and dividing by $\tau = h^{2s}$, we obtain
\[
\frac{u(x,t+\tau)-u(x,t)}{\tau} = \frac{c_s h }{\omega_{d-1}} \sum_{k = 1}^\infty \int_{\mathbb{S}^{d-1}} \frac{u(x-kh\vn, t)-u(x,t)}{|h k|^{1+2s}} \, d\sigma(\vn).
\]
Since the right-hand side is a Riemann sum corresponding to an integral in polar coordinates over $\Rd$, taking the limit $\tau \to 0$ (or, equivalently, $h \to 0$) we formally obtain
\[
u_t(x,t) = \frac{c_s}{\omega_{d-1}}\int_{\Rd} \frac{u(y, t)-u(x,t)}{|y-x|^{d+2s}} \, dy = C \Laps u (x,t) .
\]
We observe that a {\it fractional heat equation} arises as a formal limit of the random walk with jumps.

There are extensive reports in the literature of natural phenomena being successfully modeled by processes of this type. For instance, in biology, a ``hit-and-run'' hunting strategy consists in the following \cite{Humphries, Ramos, Sims}: a predator moves according to the random walk we discussed above, searches prey in its surroundings and then goes on. For non-destructive foraging, that is, assuming the prey is distributed in patches and being only temporarily depleted, reference \cite{viswanathan} shows that, whenever no a priori information about the surroundings is available, the value $s=1/2$ delivers an optimal searching strategy. 

In contrast, for destructive foraging (namely, in the case the target resource is depleted once found), the optimal search pattern corresponds to the limit $s\to 0$. This is consistent with data gathered from in-situ observations; reference \cite{Sims} compares the behavior of diverse marine vertebrates, and shows that the best fitting for the conduct of such species corresponds to values of $s$ between $0.35$ and $0.7$.

From a mathematical viewpoint, \cite{Kao} studies a model of two competing species that have the same population dynamics but different dispersal strategies: one species moves according to a classical random walk, while the other adopts a nonlocal dispersal strategy. 
Supported both by a local stability analysis and numerical simulations, the authors conjecture that nonlocal dispersal is always preferred over random dispersal.

In nature, it may be unrealistic to allow arbitrarily long jumps. Instead, there may well be a maximal jump length $\delta$. For a given function $\phi \colon [0,\infty) \to [0,\infty)$ supported in the unit interval $[0,1]$, we can consider a {\it linear fractional diffusion operator with finite horizon $\delta$},
\begin{equation} \label{eq:def-finite-horizon}
\mathcal A_{\delta,s} v (x) := \int_\Rd \phi\left( \frac{|x-y|}{\delta} \right) \frac{(v(x)-v(y))}{|x-y|^{d+2s}} \, d y.
\end{equation}
This operator localizes the interactions built in \eqref{eq:def_Laps} to a ball of radius $\delta$ centered at $x\in\Omega$. Formally, the operator above recovers (up to a constant) the fractional Laplacian in the limit $\delta \to \infty$. Operators of this type become {\em local} in either limit $s \to 1$ or $\delta \to 0$. The behavior of operators of this type in the limit $\delta \to 0$ (with a suitable scaling with respect to $\delta$) is explored in \cite{TianDu14,Du19}, for example.

%---------------------------------------------
\subsubsection*{\bf Fractional perimeters and fractional minimal graphs}
%---------------------------------------------
Suppose we want to measure the perimeter of a region $E$ by using an image of it. In particular, let us assume the figure $\Omega$ is composed by square pixels with length $h$, and the region is a tilted square with sides at $45^\circ$ with respect to the orientation of the pixels. Then, independently of the image resolution (namely, of the pixel sizes), the approximation of the perimeter of $E$ by the perimeter of $E_h$ always produces an error by a factor of $\sqrt{2}$, see Figure \ref{fig:tilted_square}.

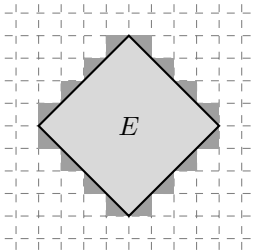
\begin{figure}[htbp] \label{fig:tilted_square}
\begin{tikzpicture}[scale=0.6]
\draw[step=0.5, gray, thin, dashed] (-2.75,-2.75) grid (2.75,2.75);
\draw[fill=gray!75, draw=none] (0,1.5) rectangle (0.5, 2);
\draw[fill=gray!75, draw=none] (0.5,1) rectangle (1, 1.5);
\draw[fill=gray!75, draw=none] (1,0.5) rectangle (1.5, 1);
\draw[fill=gray!75, draw=none] (1.5,0) rectangle (2, 0.5);

\draw[fill=gray!75, draw=none] (0,1.5) rectangle (-0.5, 2);
\draw[fill=gray!75, draw=none] (-0.5,1) rectangle (-1, 1.5);
\draw[fill=gray!75, draw=none] (-1,0.5) rectangle (-1.5, 1);
\draw[fill=gray!75, draw=none] (-1.5,0) rectangle (-2, 0.5);

\draw[fill=gray!75, draw=none] (0,-1.5) rectangle (-0.5, -2);
\draw[fill=gray!75, draw=none] (-0.5,-1) rectangle (-1, -1.5);
\draw[fill=gray!75, draw=none] (-1,-0.5) rectangle (-1.5, -1);
\draw[fill=gray!75, draw=none] (-1.5,0) rectangle (-2, -0.5);

\draw[fill=gray!75, draw=none] (0,-1.5) rectangle (0.5, -2);
\draw[fill=gray!75, draw=none] (0.5,-1) rectangle (1, -1.5);
\draw[fill=gray!75, draw=none] (1,-0.5) rectangle (1.5, -1);
\draw[fill=gray!75, draw=none] (1.5,0) rectangle (2, -0.5);

\draw[thick,rotate around={45:(0,0)}, fill=gray!30] (-2^.5,-2^.5) rectangle (2^.5, 2^.5);
\node at (0, 0) {$E$};
\end{tikzpicture}
\caption{Discrepancy between the perimeter of a tilted square $E$ and its pixel representation. Independently of the grid size $h$, the perimeter of the dark gray approximation $E_h$ is off by a factor of $\sqrt{2}$.}
\end{figure}

As an alternative, one can consider the fractional $s$-perimeter of $E$ in $\Omega$ (see Definition \ref{def:s-perimeter} below),
\begin{equation} \label{eq:s-perimeter}
P_s(E; \Omega) := | \chi_E |_{\widetilde W^{2s}_1(\Omega)}.
\end{equation}
Estimation of the $s$-perimeter of $E$ in $\Omega$ by using $E_h$ has an error of the order of $h^{1-2s}$. Moreover, because fractional $s$-perimeters recover the classical perimeter as $s \to 1/2$ (after a suitable normalization), they can provide a consistent mean to estimate the perimeter of $E$.

A classical result \cite{ModicaMortola77} states that classical minimal graphs arise as $\Gamma$-limits of certain Ginzburg-Landau energies. More precisely, given $\Omega \subset \Rd$ a bounded domain with Lipschitz boundary and $\eps >0$, we consider the double-well potential $W(t) = \frac12 (1-t^2)^2$ and the energy 
\[
E_{1,\eps}[v; \Omega] = \frac{\eps^2}2 |v|_{H^1(\Omega)}^2 + \int_\Omega W(v) .
\]
Then, for every sequence $\{u_\eps\}$ of minimizers of the rescaled energy $\eps^{-1} E_{1,\eps}[\cdot; \Omega]$ with uniformly bounded energies, there exists a subsequence that converges in $L^1(\Omega)$ to $\chi_E - \chi_{E^c}$, where $E$ is a set of minimal perimeter in $\Omega$.

Instead, when one considers the following fractional energy
\[
E_{s,\eps}[v; \Omega] = \frac{\eps^{2s}}{2} \iint_{Q_\Omega} \frac{|v(x)-v(y)|^2}{|x-y|^{d+2s}} \, dy dx  + \int_\Omega W(v) , \quad s \in (0,1),
\]
defined in $Q_\Omega:=(\Rd \times \Rd) \setminus (\Omega^c \times \Omega^c)$ with $\Omega^c := \Rd \setminus \overline\Omega$, and the $\Gamma$-limit of the corresponding rescaled energy $\gamma_\eps E_{s,\eps}[\cdot; \Omega]$ with scaling parameter
\[
\gamma_\eps = \begin{cases}
\eps^{-2s} & \mbox{ if } s < 1/2 , \\
\eps^{-1} | \log \eps|^{-1} & \mbox { if } s = 1/2 , \\
\eps^{-1} & \mbox { if } s > 1/2,
\end{cases}
\]
then a remarkable phenomenon arises \cite{SaVa12Gamma}. Indeed, in the limit with $\eps \to  0$ one also obtains convergence in $L^1(\Omega)$ to $\chi_E - \chi_{E^c}$; the set $E \subset \Omega$ has minimal perimeter if $s \ge 1/2$, but it is a minimizer of the $s$-perimeter $P_s(\cdot; \Omega)$ (cf. \eqref{eq:s-perimeter}) in case $s < 1/2$.

Let us now consider fractional minimal sets over cylinders of the form $\Omega \times \R \subset \R^{d+1}$ that, outside the cylinder correspond to the subgraph of a given bounded function $g$. In that case, the solution of the Plateau problem of finding a set $E$ that minimizes $P_s(\cdot ; \Omega \times \R)$ among all sets that coincide with the subgraph of $g$ in $\Omega^c \times \R$ is in turn the subgraph of a function, say $u$. After some technical considerations \cite{Lombardini-thesis}, one can show that such a function $u$ can be found by minimizing the energy
\begin{equation*} \label{eq:minimal-graph-energy}
I_s[v] = \iint_{Q_\Omega} F_s \left( \frac{v(x)-v(y)}{|x-y|} \right) \frac{1}{|x-y|^{d+2s-1}} \, dy dx,
\end{equation*}
where $F_s \colon \R \to \R$ is a suitable convex and nonnegative function; see \eqref{E:def_Fs} below. This energy is a fractional-order analogue of the classical graph area funcional
\[
I_1[v] = \int_\Omega \sqrt{1 + |\nabla v|^2} .
\]
The Euler-Lagrange equation corresponding to a minimizer of $I_s$ turns out to be quasilinear and elliptic, although not {\em uniformly} elliptic.

%---------------------------------------------
\subsubsection*{\bf Quasilinear operators}
%---------------------------------------------
The third class of fractional operators is given by energies of the form
\begin{equation*} \label{eq:def-energy-quasilinear}
\calF (u) = \iint_{Q_\Omega} G \left( x, y, \frac{u(x)-u(y)}{|x-y|^s} \right) \frac1{|x-y|^{d}} \, dy dx - \int_\Omega f u,
\end{equation*}
where the function $G \colon \Rd \times \Rd \times \R \to \R$ is required to satisfy suitable assumptions described in detail in \cite{BoLiNo22}. For the sake of clarity, in this paper we shall restrict our attention to the case $G(x,y,\rho) = \frac{C_{d,s,p}}{2p}|\rho|^p$ with $1<p<\infty$ and $C_{d,s,p}$ a suitable normalization constant; see \eqref{eq:Cdsp} below. The resulting energy is $\frac{1}{p} |v|_{\widetilde{W}^s_p(\Omega)}^p$ and gives rise to the {\it fractional $(p,s)$-Laplacian operator}:
\[
(-\Delta)^s_p u (x) := C_{d,s,p} \int_\Rd \frac{|u(x)-u(y)|^{p-2} (u(x)-u(y))}{|x-y|^{d+sp}} \, d y,
\] 
This generalizes \eqref{eq:def_Laps} to $1<p<\infty$. We can rewrite the operator above as
\[
	(-\Delta)^s_p u (x) = C_{d,s,p} \int_{\Rd} \left(\frac{|u(x)-u(y)|}{|x-y|^s}\right)^{p-2} \frac{(u(x)-u(y))}{|x-y|^{d+2s}} \, dy.
\]
which suggests that, heuristically, one can understand the fractional $(p,s)$-Laplacian as a weighted fractional Laplacian of order $s$, with a weight $ \left( \frac{|u(x)-u(y)|}{|x-y|^s} \right)^{p-2}$. This is analogous to the local case, for which the $p$-Laplacian $(-\Delta)_p u = \mbox{div}(|\nabla u|^{p-2} \nabla u)$ can be regarded as a Laplacian with weight $|\nabla u|^{p-2}$. The Dirichlet problem for the local $p$-Laplacian arises in a number of models of physical processes, including non-Newtonian fluids \cite{Atkinson:74} and turbulent flows in porous media \cite{Diaz:94}.

%---------------------------------------------
%---------------------------------------------
\section{Linear problems}\label{sec:linear}
%---------------------------------------------
%---------------------------------------------

This section deals with the homogeneous Dirichlet problem for the fractional Laplacian \eqref{eq:def_Laps}. We shall make use of fractional-order Sobolev spaces and Besov spaces. We employ the notation from \cite{BoNo21} and refer to that work for elementary properties of these spaces.
 
Given $f \in H^{-s}(\Omega)$, we set $\Omega^c :=\Rd\setminus\overline\Omega$ and look for a function $u\in\tHs$, the space of functions in $H^s(\Rd)$ that vanish in $\Omega^c$, such that
\begin{equation}\label{eq:Dirichlet}
\left\lbrace\begin{array}{rl}
\Laps u = f &\mbox{ in }\Omega, \\
u = 0 &\mbox{ in }\Omega^c.
\end{array} \right.
\end{equation}
To formulate \eqref{eq:Dirichlet} weakly, we introduce the following inner product in $\tHs$,
\begin{equation} \label{eq:defofinnerprod}
  (v,w)_s := \frac{C_{d,s}}2 \iint_{Q_\Omega} \frac{(v(x)-v(y))(w(x) - w(y))}{|x-y|^{d+2s}} \; dx \; dy, \quad \|v\|_{\tHs} := (v,v)_s^{1/2},
\end{equation}
and recall the notation $Q_\Omega = (\Rd\times\Rd) \setminus(\Omega^c\times\Omega^c)$.
We denote by $\langle \cdot , \cdot \rangle$ the duality pairing between $\tHs$ and its dual $H^{-s}(\Omega)$.
The constant $C_{d,s}$, defined in \eqref{eq:def_Laps}, enforces consistency between \eqref{eq:defofinnerprod} and the definition of fractional Laplacian via Fourier transform
\[
(v,w)_s = \langle (-\Delta)^s v, w \rangle = \int_{\Rd} |\xi|^{2s} \calF[v](\xi) \overline{\calF[w]}(\xi) d\xi.
\]
The weak formulation of \eqref{eq:Dirichlet} reads: find $u \in \tHs$ such that
\begin{equation} \label{eq:weak_linear}
(u, v)_s = \langle f, v \rangle \quad \forall v \in \tHs.
\end{equation}
Moreover, the following Poincar\'e inequality is valid with a constant $C(\Omega)$ uniform in $s\in(0,1)$
\begin{equation}\label{eq:Poincare}
\| v \|_{L^2(\Omega)} \le C(\Omega) \|v\|_{\tHs} \quad \forall v \in \tHs,
\end{equation}
and implies that $(\cdot,\cdot)_s$ is indeed an inner product.
By the Lax-Milgram Theorem, one immediately deduces that for all $f \in H^{-s}(\Omega)$ there exists a unique $u \in \tHs$ satisfying \eqref{eq:weak_linear} and that the solution map is continuous,
\begin{equation}\label{eq:solution-map}
\| u \|_{\tHs} \le \| f\|_{H^{-s}(\Omega)}.
\end{equation}

%---------------------------------------------
\subsection{Regularity}
%---------------------------------------------
We now review some results regarding regularity of solutions to problem \eqref{eq:weak_linear}. It is natural to expect that, if the right hand-side $f$ is smoother than $H^{-s}(\Omega)$, then some additional regularity is inherited by the weak solution $u$, which a priori is only known to belong to $\tHs$. 

By using the definition of $(-\Delta)^s$ as a pseudodifferential operator of order $2s$ and the characterization of fractional Sobolev spaces as Bessel potential spaces for $p=2$, it is clear that if $u \in L^2(\Rd)$ is such that $\Laps u \in H^t(\Rd)$, then $u \in H^{t+2s}(\Rd)$. A much less trivial question is whether this result can be used to derive interior regularity estimates for problems posed on a bounded domain. In this regard, we point out to interesting estimates in \cite{BiWaZu17,Cozzi17, Faustmann:20}.

Let us next discuss regularity of solutions to \eqref{eq:Dirichlet} up to the boundary of the domain $\Omega$. For problems posed on one-dimensional domains, one can take advantage of explicit expression of $\dist(x,\pp\Omega)$ to obtain sharp regularity results for linear operators. Indeed, for the fractional Laplacian \eqref{eq:def_Laps}, a certain kind of Jacobi polynomials, the {\em Gegenbauer polynomials}, play a crucial role in the analysis  of \eqref{eq:Dirichlet} \cite{Mao:16, ABBM18, Ervin:18}. On the interval $[-1,1]$, these polynomials are $L^2$-orthogonal with respect to the weight $\omega(x) = (1-x^2)^s$. Special polynomials are also a crucial tool in the analysis of problems with advection and/or reaction terms \cite{Hao:20,Ervin:21} and are instrumental in the development of spectral methods for these problems in 1d.

For problems based on domains in $\Rd$ with $d \ge 1$, techniques based on Fourier analysis, such as the ones employed by Vi\v{s}ik and \`Eskin \cite{Eskin, VishikEskin} or Grubb \cite{Grubb15}, allow for a full characterization of mapping properties of the integral fractional Laplacian of functions supported in $\Omega$. However, such arguments typically require the domain $\Omega$ to be $C^\infty$; the recent work by Abels and Grubb \cite{AbelsGrubb} introduces a method to handle nonsmooth coordinate changes that leads to regularity results for domains with $C^{1+\beta}$ boundary with $\beta > 2s$. The maximal regularity reads \cite{Eskin, VishikEskin, Grubb15, AbelsGrubb}
\begin{equation}\label{eq:max-reg}
  u \in H^{s+ \frac12 -\eps} (\Omega),
\end{equation}    
for any $0<\eps< s+ \frac12$ regardless of the regularity of $f$.
The following exact solution $u$ of \eqref{eq:Dirichlet} in $\Omega = B(0,1) \subset \Rd$ with $f \equiv 1$ exhibits this extreme behavior:
\begin{equation} \label{eq:getoor}
u(x) = C_s  ( 1- |x|^2)^s_+,
\end{equation}
where $t_+ =\max\{t,0\}$ and $C_s=\frac{\Gamma(\frac{d}{2})}{2^{2s}\Gamma(\frac{d+2s}{2})\Gamma(1+s)}$; in fact $u\notin H^{s+ \frac12} (\Omega)$, see \Cref{ex:solution-profile}. Therefore, it stems from this example that the boundary behavior of the solution $u$ to \eqref{eq:Dirichlet} is expected to be
\begin{equation} \label{eq:boundary_behavior}
u (x) \approx d(x, \pp\Omega)^s,
\end{equation}
for every $x$ close to $\pp\Omega$ irrespective of the regularity of $\pp\Omega$. This is in striking contrast with the boundary behavior of the classical Laplacian that inherits the regularity of $\pp\Omega$.

%------------------------------------------------------------------------------------
\subsubsection{\bf H\"older regularity}\label{sec:Holder-regularity}
%------------------------------------------------------------------------------------
%
Using potential theory tools, references \cite{AbatangeloRosOton,RosOtonSerra} studied problem \eqref{eq:Dirichlet} and related problems for integral operators with translation-invariant kernels. These references obtain a fine characterization of boundary H\"older regularity of solutions. More precisely, \cite[Proposition 1.1]{RosOtonSerra} proves that the solution $u$ satisfies
\begin{equation} \label{eq:Holder-regularity}
\| u \|_{C^s(\Rd)} \le C(\Omega) \| f \|_{L^\infty(\Omega)} .
\end{equation}
provided $\Omega$ is a Lipschitz domain that satisfies an {\it exterior ball condition} property. This estimate is also a consequence of \cite[Theorem 1.4]{AbatangeloRosOton} in case $\Omega$ is of class $C^{1,\beta}$ and $f\in C^{\beta-s}(\overline{\Omega})$ for $s<\beta<1$. Estimate \eqref{eq:Holder-regularity} does not seem to be valid uniformly in $s$ for polygonal domains in $2d$ with reentrant corners \cite{Gimperlein:19,MelenkSchwab_2021}. Once the $s$-H\"older continuity \eqref{eq:Holder-regularity} of $u$ is established, \cite{AbatangeloRosOton,RosOtonSerra} derive higher-order estimates in case $f$ possesses certain additional regularity. In fact, if
\[
\delta(x,y) := \min \left\{ \dist(x, \pp \Omega), \dist(y, \pp \Omega) \right\}
\quad\forall \, x,y\in\overline{\Omega},
\]
then a typical weighted estimate from \cite{AbatangeloRosOton,RosOtonSerra} reads
\begin{equation}\label{eq:holder}
\sup_{x,y\in\overline{\Omega}} \left\{\delta(x,y)^{\beta+s} \frac{|u(x)-u(y)|}{|x-y|^{\beta+2s}} \right\} \le C(f,u),
\end{equation}
where $\beta>0$ and $C(f,u)$ depends of suitable regularity of $f$ and \eqref{eq:Holder-regularity}. 

Estimates such as \eqref{eq:holder} are not immediately useful in the analysis of finite element methods, which rely on Sobolev estimates.
To capture the boundary behavior of $u$ in terms of Sobolev norms, \cite{AcBo17} converts estimates such as \eqref{eq:Holder-regularity} into (weighted) Sobolev estimates for $u$ in terms of H\"older norms of $f$. To make this precise, given $\ell = k + s$, with $k \in \mathbb{N}$ and $s \in (0,1)$, and $\kappa \ge 0$, we define the weighted norm
\[
\| v \|_{H^{\ell}_\kappa (\Omega)}^2 := \| v \|_{H^k (\Omega)}^2 + \sum_{|\beta| = k } 
\iint_{\Omega\times\Omega} \frac{|D^\beta v(x)-D^\beta v(y)|^2}{|x-y|^{d+2s}} \, \delta(x,y)^{2\kappa} \, dy \, dx
\]
and the associated weighted space
\begin{equation} \label{eq:weighted_sobolev}
H^\ell_\kappa (\Omega) :=  \left\{ v \in H^\ell(\Omega) \colon \| v \|_{H^\ell_\kappa (\Omega)} < \infty \right\} .
\end{equation}

The regularity estimate in the weighted Sobolev scale \eqref{eq:weighted_sobolev} for $d=2$ reads as follows \cite{AcBo17}.

\begin{theorem}[weighted Sobolev estimate] \label{T:weighted_regularity}
Let $\Omega$ be a bounded Lipschitz domain in $\R^2$, let the solution $u$ of \eqref{eq:weak_linear} satisfy \eqref{eq:Holder-regularity} and $f \in C^{1-s}(\overline\Omega)$. Then, for all $\eps>0$ sufficiently small, $u \in \widetilde H^{1+s-2\eps}_{1/2-\eps}(\Omega)$ satisfies the estimate
\[
\|u\|_{\widetilde{H}^{1+s-2\eps}_{1/2-\eps}(\Omega)} \le \frac{C(\Omega,s)}{\eps} \|f\|_{C^{1-s}(\overline{\Omega})}.
\]
\end{theorem}
Compared with \eqref{eq:max-reg}, we observe a gain of about $1/2$-order differentiability at the expense of about $1/2$-order weight.

We explore next an alternative to $L^2$-based weighted estimates that consists of reducing the integrability index \cite{BoNo21Constructive}. We consider $\Omega=(0,\infty)$ to be the half line in $1d$, and interpret the behavior \eqref{eq:boundary_behavior} as $u(x)=x_+^s$. We then wonder under what conditions this function belongs to a Sobolev space with differentiability index $r$ and integrability index $p$. For that purpose, let us compute (Riemann-Liouville) derivatives of order $r>0$ of the function $v : \R \to \R$, $v(x) =  x_+^s$:
\begin{equation} \label{eq:RL-derivative}
\partial^r v (x) \simeq x_+^{s-r}, \quad r > 0.
\end{equation}
We realize that $\partial^r v$ is $L^p$-integrable near $x=0$ if and only if $p(s-r) > - 1$, namely, if $r<s+\frac1p$. This heuristic discussion illustrates the natural interplay between the differentiability order $r$ and integrability index $p$ for membership of solutions to \eqref{eq:Dirichlet} in the class $W^r_p$, at least for dimension $d=1$. It turns out that the restriction $r<s+\frac1p$ is needed irrespective of $d$ \cite[Theorem 3.7]{BoNo21Constructive}. 

To figure out the optimal choice of indices $r,p$ compatible with nonlinear approximation theory for $d=2$, we inspect the DeVore diagram; see Figure \ref{fig:DeVore}. Recall the definition of Sobolev number $\sob (W^r_p):=r-\frac{d}{p}$ and the Sobolev line corresponding to the nonlinear approximation scale of $H^s$,
\[
\left\{ \left(\frac{1}{p}, r\right) \colon \sob(W^r_p) = \sob (H^s) \right\} = \left\{ \left(\frac{1}{p}, r\right) \colon r = s - 1 + \frac{2}{p} \right\}.
\]
In order to have a compact embedding $W^r_p(\Omega) \subset H^s(\Omega)$, we require that $\sob(W^r_p) > \sob (H^s)$ or equivalently that $(1/p,r)$ lies above the Sobolev line, as well as $r>s$. 
In addition, this line intersects the regularity line $r=s+\frac1p$ at $(1,1+s)$, which is not an admissible pair
(see Figure \ref{fig:DeVore}).
\begin{figure}[htbp]
\begin{center}
  \includegraphics[width=0.65\linewidth]{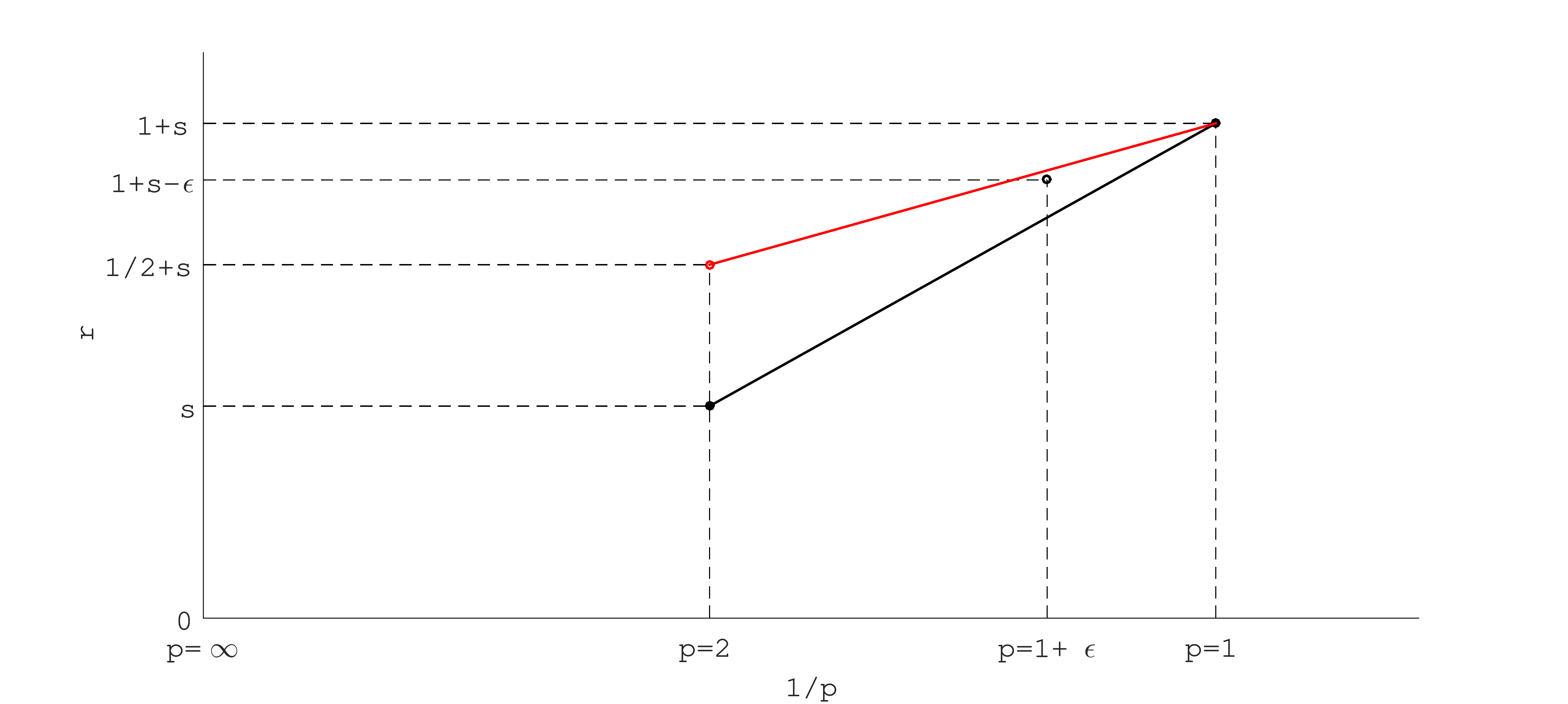}
\end{center}
\vspace{-0.5cm}
\caption{DeVore diagram for $d = 2$. Solutions to \eqref{eq:Dirichlet} are of class $W^r_p$, where $(1/p,r)$ must be below the regularity line $ \{ (1/p, r) \colon r = s + \frac{1}{p} \}$ (red). The Sobolev line $\{ (1/p, r) \colon r = s - 1 + \frac{2}{p} \}$ (black) connects the spaces $W^r_p$ and $H^s$ and corresponds to the nonlinear approximation scale of $H^s$. Both lines intersect at $(1,1+s)$.}
\label{fig:DeVore}
\end{figure}
Letting $p=1+\eps$ and $r=s+1-\eps$ with $\eps > 0$ arbitrarily small, we have
\[
\sob(W^{s+1-\eps}_{1+\eps}) =  s - 1 + \eps  \left(\frac{1-\eps}{1+\eps} \right) > \sob(H^s),
\]
while $\frac{1}{p}=\frac{1}{1+\eps} > 1 - \eps=r-s$ is also satisfied. This is thus an optimal choice of parameters for $d=2$. The H\"older regularity of \cite{AbatangeloRosOton,RosOtonSerra} leads to the following Sobolev estimate \cite[Corollary 3.8]{BoNo21Constructive}.

\begin{theorem}[differentiability vs integrability] \label{T:higher-regularity-Lp}
Let $\Omega$ be a bounded Lipschitz domain in $\R^2$, let the solution $u$ of \eqref{eq:weak_linear} satisfy \eqref{eq:Holder-regularity}, and $f\in C^{1-s}(\overline \Omega)$. Then, for $\eps$ sufficiently small, the function $u \in \widetilde{W}^{s+1-\eps}_{1+\eps}(\Omega)$ satisfies the a priori bound
\begin{equation}\label{eq:higher-regularity}
\|u\|_{\widetilde W^{s+1-\eps}_{1+\eps}(\Omega)} \leq \frac{C(\Omega, s)}{\eps^3} \| f \|_{C^{1-s}(\overline \Omega)}.
\end{equation}
\end{theorem}

We recall that \eqref{eq:Holder-regularity} hinges on the exterior ball condition \cite{AbatangeloRosOton,RosOtonSerra}. In finite element applications, where $\Omega$ is usually a polygon for $d=2$, such a condition implies $\Omega$ must be convex, which is too restrictive in practice. We next discuss regularity estimates that solely require $\Omega$ to be Lipschitz.

%---------------------------------------------------------------------------------%
\subsubsection{\bf Besov regularity}\label{sec:Besov-regularity} 
%---------------------------------------------------------------------------------%
We follow \cite{BoNo21,BoLiNo22}, which are in turn inspired in \cite{Savare98} for second order linear elliptic problems. This is a functional analysis approach to regularity and starts with the observation that the solution $u$ of \eqref{eq:Dirichlet} is the minimizer of the quadratic functional $v\mapsto \calF_2(v) - \calF_1(v)$, where
\begin{equation*}\label{eq:quadratic}
\calF_1(v) := \int_\Omega fv, \quad
\calF_2(v) := \frac12 \|v\|_{\widetilde H^s(\Omega)}^2 \quad\forall \, v \in \wt{H}^s(\Omega),
\end{equation*}
whence it satisfies the stationarity condition
\begin{equation}\label{eq:stationarity}
\frac12 \|u-v\|_{\widetilde H^s(\Omega)}^2 = \big( \calF_2(v) - \calF_2(u) \big) - \big( \calF_1(v) - \calF_1(u) \big)
  \quad\forall \, v\in \wt{H}^s(\Omega).
\end{equation}
If we now think of $v$ as a suitable translation of $u$, then \eqref{eq:stationarity} reveals that further regularity of $u$ beyond $\wt{H}^s(\Omega)$ could be inferred from regularity of the functionals $\calF_1,\calF_2$. However, to carry out this program we face two important difficulties: first, we need to localize the translations because global translations are not admissible for bounded domains; second, instead of all possible directions. we need to deal with a convex cone of directions dictated by the Lipschitz regularity of $\partial\Omega$. We briefly describe both issues now and refer to \cite{BoNo21,Savare98} for details.

We start with the definition of Besov spaces by real interpolation. Given a pair of compatible Banach spaces $(X_0,X_1)$, $u \in X_0 + X_1$, and $t>0$, we set the $K$-functional \cite{mclean2000strongly}
\begin{equation}\label{eq:K-functional}
K(t, u) := \inf \left\{ \big(\| u_0 \|_{X_0}^2 + t^2 \| u_1 \|_{X_1}^2 \big)^{1/2} \colon u = u_0 + u_1, \ u_0 \in X_0, \ u_1 \in X_1 \right\}. 
\end{equation}
For $\theta \in (0,1)$ and $q \in [1, \infty]$, let us define interpolation spaces
\[
\big[X_0, X_1\big]_{\theta, q} := \{ u \in X_0 + X_1 \colon \| u \|_{(X_0, X_1)_{\theta, q}} < \infty \},
\]
where
\begin{equation}\label{eq:Besov-norm}
\| u \|_{[X_0, X_1]_{\theta, q}} =
\begin{cases}
\Big[ q \theta (1-\theta) \int_0^\infty t^{-(1+\theta q)} |K(t,u)|^q \, dt \Big]^{1/q} & \mbox{if } 1 \le q < \infty, \\
\sup_{t > 0} \ t^{-\theta} |K(t,u)|  & \mbox{if } q = \infty.
\end{cases}
\end{equation}
The normalization factor $q \theta (1-\theta)$ in the norm \eqref{eq:Besov-norm} guarantees the correct scalings in the limits $\theta \to 0$, $\theta \to 1$ and $q \to \infty$ for norm continuity; see \cite[Appendix B]{mclean2000strongly}.
In our setting, we let $X_0:=L^2(\Omega)$, $X_1:=H^2(\Omega)$, $\sigma\in(0,2)$ and $q\in [1,\infty]$ to obtain
\[
B^\sigma_{2,q} (\Omega) := \big[L^2(\Omega), H^2(\Omega)\big]_{\sigma/2,q},\quad
\dot{B}^\sigma_{2,q} (\Omega) := \{v\in B^\sigma_{2,q} (\Omega): \supp v \subset\overline{\Omega}\}.
\]
In particular, we have that $B^\sigma_{2,2} (\Omega) = H^\sigma(\Omega)$ for all $\sigma \in (0,2)$.
If $\sigma \in (0,1)$, then we can also set $\dot{B}^\sigma_{2,q} (\Omega) := [L^2(\Omega), H^1_0(\Omega)]_{\sigma,q}$ and point out that the spaces coincide but the corresponding norms exhibit a different scaling as $\sigma\to1$; we will quantify this discrepancy in \eqref{eq:Marchaud} below. 
Besov spaces can also be characterized in terms of first and second difference operators $\delta_1(h)$ and $\delta_2(h)$:
\[
\delta_1(h) v(x) := v(x+h) - v(x),
\qquad
\delta_2(h) v(x) := v(x+h) + v(x-h) - 2v(x),
\]
for $v\in L^2(\Omega)$ and $x \in \Omega_{|h|} := \{y\in\Omega: d(y,\partial\Omega)>|h|\}$. In fact, if $D=D_\rho(0)$ is the ball in $\R^d$ of radius $\rho$ centered at the origin and $\sigma\in(0,2)$, we define the seminorms for $q\in [1,\infty)$
\begin{equation}\label{eq:q<infty}
|v|_{B^\sigma_{2,q}(\Omega)} := \left( q \sigma (2-\sigma) \int_{D} 
\frac{\| \delta_2(h) v \|^q_{L^2(\Omega_{|h|})}}{|h|^{d+ q \sigma}} dh,
\right)^{1/q} ,
\end{equation}
and $q = \infty$
\begin{equation}\label{eq:q=infty}
|v|_{B^\sigma_{2,\infty}(\Omega)} := \sup_{h \in D} 
\frac{\| \delta_2(h) v \|_{L^2(\Omega_{|h|})}}{|h|^{\sigma}},
\end{equation}
and observe that $\|v\|_{L^2(\Omega)} + |v|_{B^\sigma_{2,q}(\Omega)}$ is equivalent to the norm induced by interpolation and is robust with respect to $\sigma,q$ and $\rho$ \cite[Theorem 7.47]{AdamsFournier:2003}. It turns out that $q=\infty$ is the most significant index for us, whence we focus on it and state the {\it Marchaud inequality} for $\sigma\in(0,1)$ (cf. \cite{Ditzian88})
\begin{equation}\label{eq:Marchaud}
  \sup_{h\in D} \frac{\| \delta_1(h) v \|_{L^2(\Omega_{|h|})}}{|h|^{\sigma}} \lesssim
  \|v\|_{L^2(\Omega)} + 
  \frac{1}{\sqrt{1-\sigma}} \sup_{h\in D} \frac{\| \delta_2(h) v \|_{L^2(\Omega_{|h|})}}{|h|^{\sigma}}
  \qquad\forall v\in \dot{B}^\sigma_{2,\infty}(\Omega),
\end{equation}
that quantifies the precise blow-up on first differences relative to second differences as $\sigma\to1$. This behavior is sharp as we discuss next with an explicit example.

\begin{example}[regularity of explicit solution] \label{ex:solution-profile}
Our goal now is to justify that $q=\infty$ is the most adequate choice in the present setting. To this end, we examine first and second differences for the 1d-function $v(x) = x_+^s$ for $x \in [0,1]$, $v(x) = 0$ for $x \in (-1,0)$ and $v(x) = 1$ for $x \in (1,2)$, which possesses a profile at $x=0$ consistent with \eqref{eq:getoor} and \eqref{eq:boundary_behavior}. For $h\in(0,\rho)$, we decompose $\| \delta_1(h) v \|_{L^2(\Omega_h)}^2$ into integrals over $(-h,0)$, $(0,h)$, $(h,1-h)$ and $(1-h,1)$ which are all of order $h^{1+2s}$ for $s\in(0,1/2)$. To estimate the most delicate integral over the interval $(h,1-h)$, we use a dyadic partition $(2^k h , 2^{k+1} h)$ with $k \le K < \frac{|\log h|}{\log 2}$
\[\begin{split}
\int_h^{1-h} \big[ (x+h)^s - x^s \big]^2 dx & \approx \sum_{k=0}^K \int_{2^k h}^{2^{k+1} h} x^{2s} \big[ (1 + 2^{-k})^s - 1 \big]^2 dx \\
& \approx s^2 h^{1+2s} \sum_{k=0}^K 2^{k(2s-1)} \approx (1-2s)^{-1} h^{1+2s}.
\end{split}\]
This implies
\begin{equation} \label{eq:blow-up-q}
\int_0^\rho \frac{\| \delta_1(h) v \|_{L^2(\Omega_{h})}^q}{h^{1+q(s+1/2)}} dh \approx \int_0^\rho \frac{dh}{h} = \infty,
\end{equation}
for any $q \in [1,\infty)$ whereas for $q = \infty$
\[
\sup_{h\in (0,\rho)} \frac{\| \delta_1(h) v \|_{L^2(\Omega_h)}}{h^{s+1/2}} \approx \frac{1}{\sqrt{1-2s}},
\]
which shows a blow-up as $s\to\frac12$.
Moreover, if $s=1/2$, the calculation above becomes
\[
\|\delta_1(h) v\|_{L^2(\Omega_h)}^2 \approx
\sum_{k=0}^K \int_{2^k h}^{2^{k+1} h} x \big[ (1 + 2^{-k})^{\frac12} - 1 \big]^2  dx
\approx h^2 \sum_{k=1}^K 1 = h^2 K \approx h^2 |\log h|.
\]
This shows that first differences are not adequate for $s=\frac12$. In contrast, utilizing second differences yields a critical term of the form
\[
\int_h^{1-h} \big[ (x+h)^s + (x-h)^s - 2 x^s \big]^2 dx \approx h^{1+2s}
\quad\Rightarrow\quad
\sup_{h\in(0,\rho)} \frac{\|\delta_2(h) v\|_{L^2(\Omega_h)}}{h^{s+1/2}} \approx 1
\]
for all $s\in(0,1/2]$ with a uniform constant as $s\to\frac12$. This derivation reveals that the constant $(1-\sigma)^{-\frac12}$ in \eqref{eq:Marchaud} is sharp. Moreover,
if $s\in (1/2,1)$, then $v \in H^1(\Omega)$, whence computing the first weak
derivative $v'(x) = s x_+^{s-1}$ and repeating the previous calculation yields
\[
\|\delta_1(h) v' \|_{L^2(\Omega_h)} \approx \frac{h^{s-1/2}}{\sqrt{2s-1}},
\qquad
\|\delta_2(h) v' \|_{L^2(\Omega_h)} \approx h^{s-\frac12}.
\]
In view of our definition \eqref{eq:q=infty} involving second differences, we deduce that $v(x) = x_+^s$ satisfies
\begin{equation*}\label{eq:optimal-regularity}
  v \in B^{s+\frac12}_{2,\infty} (\Omega)
  \quad \forall s\in (0,1),
\end{equation*}
and motivates the use of the Besov space $\dot{B}^{s+\frac12}_{2,\infty} (\Omega)$ in our regularity theory for Lipschitz domains. In contrast, \eqref{eq:blow-up-q} implies that $v \notin B^{s+\frac12}_{2,q} (\Omega)$ for any $1\le q < \infty$, and in particular $v \notin H^{s+\frac12}(\Omega)$.
\end{example}

In light of Example \ref{ex:solution-profile}, we next discuss the key steps leading to the {\it optimal shift property}
\begin{equation*} \label{eq:shift_property}
\| u \|_{\dot B^{s+1/2}_{2,\infty}(\Omega)} \le C(\Omega,d) \|f\|_{B^{-s+1/2}_{2,1}(\Omega)}.
\end{equation*}

To guarantee that local translations of $u$ are admissible test functions, we need to restrict the set of {\it admissible directions} $h$ to a convex cone related to the local Lipschitz structure of $\partial\Omega$: there exist $\rho>0$, $\theta \in (0,\pi]$ and a map $\vn:\Omega\to\mathbb{S}^{d-1}$ such that for all $x\in\Omega$ the cone $\calC_\rho(\vn(x),\theta)$ with height $\rho$, aperture $\theta$, apex $0$ and axis $\vn(x)$ gives admissible outward vectors in the sense that
\begin{equation}\label{eq:admissible-cone}
h \in \calC_\rho(\vn(x),\theta)
\qquad\Rightarrow\qquad
D_{3\rho}(x)\cap\Omega^c + th \subset \Omega^c \quad \forall t\in [0,1].
\end{equation}
Such cone $D =  \calC_\rho(\vn(x),\theta)$ has the property that generates $\R^d$ in the sense that for all $h\in D_\rho(0)$, there exist $\{h_j\}_{j=1}^d \subset D \cup (-D)$ and a constant $c>0$ only depending on $\theta$ such that
\[
h = \sum_{j=1}^d h_j, \qquad
\sum_{j=1}^d |h_j| \le c |h|.
\]
It thus follows that restricting the ball $D_\rho(0)$ to the cone $D$ in the definitions of Besov seminorms \eqref{eq:q<infty} and \eqref{eq:q=infty} yields equivalent seminorms \cite[Proposition 2.2]{BoNo21}.

Given $x_0\in\Omega$, let $\phi\in C_0^\infty(D_{2\rho}(x_0))$ be a cut-off function satisfying $0\le\phi(x)\le 1$ for all $x\in D_{2\rho}(x_0)$ and $\phi(x)=1$ for all $x\in D_{\rho}(x_0)$. Given an admissible direction $h\in D=\calC_\rho(\vn(x_0),\theta)$, we define the {\it localized translation operator} $T_h:\wt{H}^s(\Omega) \to \wt{H}^s(\Omega)$ to be
\begin{equation}\label{eq:translation}
T_h v(x) := v \big(x+h\phi(x)\big) \quad\forall \, x\in\Omega.
\end{equation}
The operator $T_h$ translates $v$ along the direction $h\in D$ and coincides with the identity in $D_{2\rho}(x_0)^c$. Moreover, if $v\in\wt{H}^s(\Omega)$ then $T_hv\in \wt{H}^s(\Omega)$ is an admissible test function to insert in \eqref{eq:stationarity} and probe the behavior of the functionals $\calF_1$ and $\calF_2$. Combining a reiteration property of Besov seminorms (cf. \cite[Theorem 7.21]{AdamsFournier:2003},\cite[Proposition 2.1]{BoNo21}) with \eqref{eq:stationarity} yields for $\sigma \in (0,2)$,
\begin{align*}
  |u|^2_{B^{s+\sigma/2}_{2,\infty} (D_\rho(x_0))} & \lesssim \sup_{h\in D} \frac{|\delta_1(h) u|^2_{H^s(D_\rho(x_0))}}{|h|^\sigma}
  \\
  & \lesssim \sup_{h\in D} \frac{\|T_h u - u\|^2_{\wt{H}^s(\Omega)}}{|h|^\sigma}
  \\
  & \lesssim \sup_{h\in D} \frac{|\calF_1(T_h u) - \calF_1(u) |}{|h|^\sigma}
  + \sup_{h\in D} \frac{|\calF_2(T_h u) - \calF_2(u) |}{|h|^\sigma},
\end{align*}
and shows that bounding the right hand side induces further local regularity of $u$.
The following estimates for $\calF_1$ and $\calF_2$ are derived in \cite{BoLiNo22}: for $\sigma \in (0,1]$ and $t \in (-1,1+\sigma)$, we have
\begin{gather} \label{eq:regularity-F1}
  \sup_{h\in D}
  \frac{|\calF_1(T_hv) - \calF_1(v)|}{|h|^\sigma}
  \lesssim \|f\|_{B_{2,1}^{t}(D_{2\rho}(x_0) \cap\Omega)} |v|_{B^{\sigma-t}_{2,\infty}(D_{2\rho}(x_0))}
  \qquad\forall v \in \dot{B}^{\sigma-t}_{2,\infty}(D_{2\rho}(x_0));
\\
 \label{eq:regularity-F2}
  \sup_{h\in D}
  \frac{|\calF_2(T_hv) - \calF_2(v)|}{|h|^\sigma}
  \lesssim \int\int_{Q_{D_{2\rho}(x_0)}} \frac{|v(x)-v(y)|^2}{|x-y|^{d+2s}} dydx  \qquad \forall v \in \wt{H}^s(\Omega).
   \end{gather}
These estimates improve upon \cite[Propositions 3.1 and 3.2]{BoNo21} due to the special structure of the translation operator $T_h$ in \eqref{eq:translation}. We illustrate this point with the first estimate, which is the trickiest and most insightful. We first observe that the estimate
\begin{equation}\label{eq:translation-Besov}
  \| v - T_h v \|_{B^{r}_{2,q}(D_{2\rho} (x_0))} \lesssim
  |h|^\sigma \|v\|_{B^{r+\sigma}_{2,q}(D_{2\rho} (x_0))}
  \quad\forall \, v\in B^{r+\sigma}_{2,q}(D_{2\rho} (x_0))
\end{equation}
is valid for $r \in (-1,1)$ with constants independent of $\sigma\in[0,1]$ and $q\in[1,\infty]$, with possible blow up as $r\to\pm1$; we will come back to this below. Applying \eqref{eq:translation-Besov} for $q=\infty$ readily implies
\[
\big| \calF_1(T_h v) - \calF_1 (v) \big| \lesssim
|h|^\sigma \|f\|_{B^{-r}_{2,1}(\Omega\cap D_{2\rho}(x_0))} \|v\|_{B^{r+\sigma}_{2,\infty}(D_{2\rho} (x_0))}.
\]
We next rewrite $T_hv$ in \eqref{eq:translation} as $T_hv=v\circ S_h$ with $S_h=I+h\phi$, whence
$\int_\Omega f T_h v = \int_{S_h(\Omega)} f \circ S_h^{-1} v |J|$ with $J= \det\nabla S_h^{-1}$.
Exploiting this along with the fact that $S_h$ is close to the identity, and using \eqref{eq:translation-Besov}
with $q=1$, yields
\[
\big| \calF_1(T_h v) - \calF_1 (v) \big| \lesssim
|h|^\sigma \|f\|_{B^{r+\sigma}_{2,1}(\Omega\cap D_{2\rho}(x_0))} \|v\|_{B^{-r}_{2,\infty}(D_{2\rho} (x_0))}.
\]
Since the map $(v,f)\mapsto  \calF_1(T_h v) - \calF_1 (v)$ is bilinear, operator interpolation theory gives the asserted
estimate \eqref{eq:regularity-F1}. Consequently, by this estimate and \eqref{eq:regularity-F2}, we arrive at the local estimate for a generic point $x_0 \in \Omega$,
\[
|u|_{B^{s+\sigma/2}_{2,\infty}(D_\rho(x_0))}^2 \lesssim
\|f\|_{B_{2,1}^{t}(D_{2\rho}(x_0) \cap\Omega)} |u|_{B^{\sigma-t}_{2,\infty}(D_{2\rho}(x_0))}
+ \iint_{Q_{D_{2\rho}(x_0)}} \frac{|u(x)-u(y)|^2}{|x-y|^{d+2s}} dydx.
\]

To apply this estimate, we consider a finite covering of $\{x\in\Rd: \dist(x,\Omega)<\rho/2\}$ with balls $D_j=D(x_j,\rho)$ centered at $x_j\in\Omega$ and radius $\rho$, $j\le J$, and recall the {\it norm localization property} 
\begin{equation}\label{eq:localization}
|v|^2_{\dot{B}^r_{2,q}(\Omega)} \approx \sum_{j=1}^J |v|^2_{B^r_{2,q}(D_j)}
\quad \forall \, r > 0, \, q\in[1,\infty].
\end{equation}
This, in conjunction with the previous estimate and \eqref{eq:solution-map},  gives the {\it fundamental recursion formula}
\begin{equation}\label{eq:regularity-recursion}
  \|u\|_{\dot{B}_{2,\infty}^{s+\sigma/2}(\Omega)}^2 \lesssim \|f\|_{B^{t}_{2,1}(\Omega)} \|u\|_{\dot{B}_{2,\infty}^{\sigma-t}(\Omega)}
  + \|f\|_{H^{-s}(\Omega)}^2
  \qquad \forall \, \sigma\in(0,1], \, t \in (-1,1+\sigma).
\end{equation}

We finally discuss two direct consequences of \eqref{eq:regularity-recursion}. The next result is proven in \cite{BoLiNo22}.

\begin{theorem}[optimal shift property]\label{T:Besov_regularity}
Let $\Omega$ be a bounded Lipschitz domain and $f \in B^{-s+1/2}_{2,1}(\Omega)$. Then, the
solution $u$ to \eqref{eq:Dirichlet} belongs to the Besov space $\dot B^{s+1/2}_{2,\infty}(\Omega)$ and satisfies
\begin{equation} \label{eq:Besov_regularity}
\| u \|_{\dot B^{s+1/2}_{2,\infty}(\Omega)} \le C(\Omega,d) \|f\|_{B^{-s+1/2}_{2,1}(\Omega)},
\end{equation}
with a constant $C(\Omega,d)$ uniform with respect to $s\in(0,1)$.
\end{theorem}
\begin{proof}
We proceed by iteration of \eqref{eq:regularity-recursion}. We fix $t=\frac12-s$ and note that, according to \eqref{eq:solution-map},
\begin{equation}\label{eq:bound-Hs}
\|u\|_{\dot{B}_{2,\infty}^s(\Omega)} \lesssim \|u\|_{\wt{H}^s(\Omega)} \le \|f\|_{H^{-s}(\Omega)} \lesssim \|f\|_{B^{-s+1/2}_{2,1}(\Omega)}
\ \ \Rightarrow\ \
\|u\|_{\dot{B}_{2,\infty}^s(\Omega)} \le \Lambda_0 \|f\|_{B^{-s+1/2}_{2,1}(\Omega)},
\end{equation}
where $\Lambda_0$ depends on $\Omega,d$ and $s$. We rewrite \eqref{eq:regularity-recursion} as
\begin{equation}\label{eq:iteration}
\|u\|_{\dot{B}_{2,\infty}^{\sigma_{k+1}}(\Omega)}^2 \le \Big( C_1 \|f\|_{B^{-s+1/2}_{2,1}(\Omega)} + C_2 \|u\|_{\dot{B}_{2,\infty}^{\sigma_k}(\Omega)}   \Big) \|f\|_{B^{-s+1/2}_{2,1}(\Omega)},
\end{equation}
with constants $C_1,C_2>0$ depending on $\Omega, d$ and $s$, and parameters
\[
\sigma_{k+1} := s+\sigma/2, \quad
\sigma_k := \sigma - t = \sigma - \frac12 + s,
\]
for $k\ge 0$.
We now set $\sigma_0=s$ and prove by induction that $\sigma_k = s + \frac12 \big( 1-\frac{1}{2^k} \big)$ and 
\begin{equation}\label{eq:induction}
\|u\|_{\dot{B}_{2,\infty}^{\sigma_k}(\Omega)} \le \Lambda_k \|f\|_{B^{-s+1/2}_{2,1}(\Omega)},
\end{equation}
with uniformly bounded constants $\Lambda_k \le M := \max\big\{\Lambda_0,\big( 2C_1 + C_2^2 \big)^{\frac12} \big\}$. We first observe that $\sigma = 1 - \frac{1}{2^{k+1}} \in (0,1]$ and $-1 < t = \frac12-s < 1+\sigma$ are within the range of validity of \eqref{eq:regularity-recursion}.
For $k=0$, formula \eqref{eq:induction} is true in view of \eqref{eq:bound-Hs} and $\sigma_0=s$, whereas for $k\ge0$ the expression for $\sigma_{k+1}$ is easy to verify and \eqref{eq:iteration} gives
\[
\|u\|_{\dot{B}_{2,\infty}^{\sigma_{k+1}}(\Omega)}^2 \le \big(C_1 + C_2\Lambda_k \big) \|f\|_{B^{-s+1/2}_{2,1}(\Omega)}^2
= \Lambda_{k+1}^2 \|f\|_{B^{-s+1/2}_{2,1}(\Omega)}^2.
\]
It remains to show that $\Lambda_k \le M$ for all $k \ge 0$.  This is clearly true for $k=0$,  and for $k \ge 0$ we see that if $\Lambda_k \le M$, then
\[
\Lambda_{k+1}^2 = C_1 + C_2 \Lambda_k \le C_1 + C_2 M \le \frac12 \big( 2C_1+C_2^2 + M^2  \big) \le M^2.
\]
We finally replace $\Lambda_k$ by $M$ and take $k\to\infty$ in \eqref{eq:induction} to obtain \eqref{eq:Besov_regularity}.
\end{proof}

\begin{remark}[intermediate estimate]\label{R:intermediate}
Interpolating between \eqref{eq:Besov_regularity} and \eqref{eq:bound-Hs} yields
\begin{equation*} \label{eq:shift-r}
\| u \|_{\dot B^{s+r/2}_{2,\infty}(\Omega)} \le C(\Omega,d,s) \|f\|_{B^{-s+r/2}_{2,1}(\Omega)}
\end{equation*}
for all $r \in (0,1)$, with a constant independent of $r$. This estimate will be useful in Proposition \ref{prop:negativeSob}.
\end{remark}

We now state the second consequence of \eqref{eq:regularity-recursion}, which was originally derived in \cite{BoNo21}.

\begin{theorem}[regularity with $L^2$-data]\label{T:L2-data}
Let $\Omega$ be a bounded Lipschitz domain and $f \in L^2(\Omega)$. If $s\ne\frac12$, then the
solution $u$ to \eqref{eq:Dirichlet} belongs to the Besov space $\dot B^{s+\alpha}_{2,\infty}(\Omega)$
with $\alpha=\min\{s,1/2\}$ and satisfies
\begin{equation} \label{eq:Besov_regularityL2}
\| u \|_{\dot B^{s+\alpha}_{2,\infty}(\Omega)} \le \frac{C(\Omega,d)}{\sqrt{|1-2s|}} \|f\|_{L^2(\Omega)}.
\end{equation}
Instead, if $s=\frac12$, then for any $\eps>0$ sufficiently small there holds
\begin{equation} \label{eq:Besov_regularityL2_s=1/2}
\| u \|_{\dot B^{1-\eps}_{2,\infty}(\Omega)} \le \frac{C(\Omega,d)}{\sqrt{\eps}} \|f\|_{L^2(\Omega)}.
\end{equation}
\end{theorem}
\begin{proof}
We proceed in two steps. We consider first the case $s\le\frac12$. To account for the Marchaud inequality \eqref{eq:Marchaud}, we let $\sigma\in(0,1)$ and modify \eqref{eq:translation-Besov} as follows:
\[  
\| v - T_h v \|_{L^2(D_{2\rho}(x_0))} \lesssim \frac{|h|^\sigma}{\sqrt{1-\sigma}} \|u\|_{B^\sigma_{2,\infty}(D_{2\rho}(x_0))}.
\]
Correspondingly, we change the estimate of the functional $\calF_1$ to read
\[
\big| \calF_1(T_hv) - \calF_1(v) \big| \lesssim  \frac{|h|^\sigma}{\sqrt{1-\sigma}}
\|f\|_{L^2(D_{2\rho}(x_0))} \|u\|_{B^\sigma_{2,\infty}(D_{2\rho}(x_0))}.
\]
The bound \eqref{eq:regularity-recursion} changes accordingly and becomes
\[
\| u \|_{\dot{B}^{s+\sigma/2}_{2,\infty}(\Omega)}^2 \lesssim \frac{1}{\sqrt{1-\sigma}}
\|f\|_{L^2(\Omega)} \|u\|_{B^\sigma_{2,\infty}(\Omega)} + \|f\|_{L^2(\Omega)}^2,
\]
We now set $\sigma_0=s$ and convert this expression into the recursion formula
\[
\| u \|_{\dot{B}^{\sigma_{k+1}}_{2,\infty}(\Omega)}^2 \le \Big( \frac{C_1}{\sqrt{1-\sigma_k}}
\|u\|_{B^{\sigma_k}_{2,\infty}(\Omega)} + C_2 \|f\|_{L^2(\Omega)} \Big) \|f\|_{L^2(\Omega)},
\]
with $\sigma_{k+1} = s + \frac{\sigma_k}{2}$. We again proceed by induction to show that
$\sigma_k = 2s\big(1-2^{-(k+1)}\big)$, $\Lambda_k\le M \big(1-2s\big)^{-1/2}$ with $M$ depending on $\Omega$
and $d$, and
\begin{equation}\label{eq:iterationL2}
\| u \|_{\dot{B}^{\sigma_k}_{2,\infty}(\Omega)} \le \Lambda_k \|f\|_{L^2(\Omega)}.
\end{equation}
This is valid for $k=0$, and for $k\ge0$ can be derived along the lines of the proof of Theorem \ref{T:Besov_regularity}. For $s<\frac12$, we realize that $\sigma_k\to 2s$ and \eqref{eq:Besov_regularityL2} follows immediately. For $s=\frac12$ and $0<\eps<\frac12$, let $k_\eps\in\mathbb{N}$ satisfy $2^{-k_\eps} < \eps \le 2^{-(k_\eps-1)}$. Using \eqref{eq:iterationL2} for this $k_\eps$ yields \eqref{eq:Besov_regularityL2_s=1/2}.

It remains to tackle the case $s>\frac12$. In view of\eqref{eq:Besov_regularity}, we only need to worry about the dependence in $s$ of the constant in the inequality $\|f\|_{B^{-s+1/2}_{2,1}(\Omega)} \lesssim \|f\|_{L^2(\Omega)}$ as $s \to 1/2$. We aim to prove that such a constant scales as $(s-1/2)^{-1/2}$. By duality, it suffices to show
\[
\|v\|_{L^2(\Omega)} \lesssim \frac{1}{\sqrt{2s-1}} \|v\|_{B^{s-1/2}_{2,\infty}(\Omega)}
\quad\forall \, v\in B^{s-1/2}_{2,\infty}(\Omega).
\]
This follows from a careful manipulation of the $K$-functional \eqref{eq:K-functional}. To do this, we regard $L^2(\Omega)$ and $B^{s-1/2}_{2,\infty}(\Omega)$ as interpolation spaces between $X_0=H^{-1}(\Omega)$ and $X_1=H^1_0(\Omega)$ with 
\[
L^2(\Omega) = \big[ H^{-1}(\Omega),H^1_0(\Omega) \big]_{1/2,2},
\quad
B^{s-1/2}_{2,\infty}(\Omega) = \big[ H^{-1}(\Omega),H^1_0(\Omega) \big]_{\theta,\infty},
\]
where $\theta = \frac12 \big(s+\frac12\big)$. This choice of spaces $(X_0,X_1)$ guarantees that $\theta$ is uniformly far from $0,1$ and the norms in \eqref{eq:Besov-norm} are robust. Using \eqref{eq:Besov-norm} for $q=2$ we deduce that for any $N\ge1$ to be found
\[
\|v\|_{L^2(\Omega)}^2 \lesssim \int_0^N t^{-2} \big| K(t,v) \big|^2 dt + \int_N^\infty t^{-2} \big| K(t,v) \big|^2 dt.
\]
Moreover, exploiting again \eqref{eq:Besov-norm} but now for $q=\infty$ yields
\[
\int_0^N t^{-2} \big| K(t,v) \big|^2 dt \le \sup_{t>0} \Big( t^{-2\theta} \big| K(t,v) \big|^2 \Big)
\int_0^N t^{2\theta-2}dt = \frac{N^{s-\frac12}}{s-\frac12} \|v\|_{B^{s-1/2}_{2,\infty}(\Omega)}.
\]
On the other hand, since $\Omega$ is Lipschitz, combining a duality argument with the standard Poincar\'e inequality in $H^1_0(\Omega)$ gives $\big| K(t,v) \big| \le \|v\|_{H^{-1}(\Omega)} \le C(\Omega) \|v\|_{L^2(\Omega)}$ and
\[
\int_N^\infty t^{-2} \big| K(t,v) \big|^2 dt \le C(\Omega)^2 \|v\|_{L^2(\Omega)}^2 \int_N^\infty t^{-2} dt
= \frac{C(\Omega)^2}{N} \|v\|_{L^2(\Omega)}^2.
\]
Finally, choosing $N=2C(\Omega)^2$ leads to the desired estimate and concludes the proof.
\end{proof}

We are now in a position to compare the estimate \eqref{eq:Besov_regularity} for $s=\frac12$ and \eqref{eq:Besov_regularityL2_s=1/2}. We see again the special role played by the Besov space $B^0_{2,1}(\Omega)$ which is a strict subspace of $L^2(\Omega)$.

We next convert the Besov estimates in Theorems \ref{T:Besov_regularity} and \ref{T:L2-data} into Sobolev estimates. Such estimates can be compared with those in the literature and, more importantly, used to improve the existing the error analyses in Sections \ref{S:global-error-estimates} and \ref{S:local-error-estimates}.

\begin{corollary}[Sobolev estimates]\label{C:Sobolev-estimates}
Let $\Omega$ be a bounded Lipschitz domain, $\alpha=\min\{s,\frac12\}$, and $u$ be the solution to \eqref{eq:Dirichlet}. For any
$\eps>0$ sufficiently small, the following estimates are valid
\begin{align}\label{eq:regularity-SobBesov}
\| u \|_{\wt{H}^{s+1/2-\eps}(\Omega)} &\le \frac{C(\Omega,d)}{\sqrt\eps} \|f\|_{{B}^{-s+1/2}_{2,1}(\Omega)}
\quad s\in (0,1),
 \\
 \label{eq:regularity-SobL2}
 \| u \|_{\wt{H}^{s+\alpha-\eps}(\Omega)} &\le \frac{C(\Omega,d)}{\sqrt{\eps|1-2s|}} \|f\|_{L^2(\Omega)}\quad
 s\ne\frac12,
 \\
 \label{eq:regularity-SobL2-s=1/2}
 \| u \|_{\wt{H}^{1-\eps}(\Omega)} &\le \frac{C(\Omega,d)}{\eps} \|f\|_{L^2(\Omega)}\quad
 s = \frac12.
\end{align}
\end{corollary}
\begin{proof}
The three estimates follow by combining respectively  \eqref{eq:Besov_regularity}, \eqref{eq:Besov_regularityL2}, or \eqref{eq:Besov_regularityL2_s=1/2}, with the embedding $\dot{B}^\sigma_{2,\infty}(\Omega) \subset \wt{H}^{\sigma-\eps}(\Omega)$ for any $\sigma\in (0,2)$ and the $\eps$-dependent estimate for $\eps>0$ small
\begin{equation}\label{eq:Besov-Sobolev-emb}
  \|v\|_{\wt{H}^{\sigma-\eps}(\Omega)} \lesssim \frac{1}{\sqrt{\eps}} \|v\|_{B^\sigma_{2,\infty}(\Omega)}
  \quad\forall \, v\in B^\sigma_{2,\infty}(\Omega).
\end{equation}
To derive \eqref{eq:Besov-Sobolev-emb}, we proceed as in the proof of Theorem \ref{T:L2-data} and consider 
the $\mbox{K}$-functional \eqref{eq:K-functional} for the pair of spaces $X_0=L^2(\Omega)$ and $X_1=H^2(\Omega)$
and $\theta = \frac{\sigma-\eps}{2}$. We see that
\[
\|v\|_{H^{\sigma-\eps}(\Omega)}^2 \lesssim \int_0^\infty t^{-1-\sigma+\eps} \big| K(t,v) \big|^2 dt
\le \frac{1}{\eps} \sup_{t>0} \Big( t^{-\sigma} \big| K(t,v) \big|^2 \Big) + \frac{1}{\sigma-\eps} \|v\|_{L^2(\Omega)}^2
\lesssim \frac{1}{\eps} \|v\|_{B^\sigma_{2,\infty}(\Omega)}^2
\]
is the asserted estimate.
\end{proof}

The estimates \eqref{eq:regularity-SobBesov} and \eqref{eq:regularity-SobL2-s=1/2} are consistent with the maximal regularity \eqref{eq:max-reg} derived in \cite{Eskin, VishikEskin, Grubb15, AbelsGrubb} under assumptions on $\partial\Omega$ stronger than Lipschitz continuity.

%---------------------------------------------------------------------------------%
\subsubsection{\bf Linear problems with finite horizon.}
%---------------------------------------------------------------------------------%
%
We now extend the preceding results to operators of the form \eqref{eq:def-finite-horizon}. Indeed, if $\phi$ is bounded then we immediately have that the energy norm
\[
\iii{v} := \left(\frac{C_{d,s}}{2} \iint_{\Rd\times\Rd} \phi\left( \frac{|x-y|}{\delta} \right) \frac{|v(x)-v(y)|^2}{|x-y|^{d+2s}} \, d y dx \right)^{\frac12}
\]
with $C_{d,s} \approx s(1-s)$ defined in \eqref{eq:def_Laps} satisfies
\begin{equation} \label{eq:FH-upper-bound}
\iii{v} \le \| \phi \|_{L^\infty(\Rd)}^{\frac12} \|v\|_{\tHs} \quad \forall v \in \tHs,
\end{equation}
for all $s\in(0,1)$.
Additionally, if $\phi \ge \phi_0 > 0$ on an interval $[0,r]$ for some $r\in(0,1)$, then the localization estimate (cf. for example \cite[Lemma 7]{DyKa13})
\begin{equation*} \label{eq:FH-chain-bound}
  \iint_{D_R(0) \times D_R(0)} \frac{|v(x)-v(y)|^2}{|x-y|^{d+2s}} \, d y dx \le \left(\frac{3R}{r \delta} \right)^{2(1-s)} \iint_{D_R(0) \times D_R(0)} \frac{|v(x)-v(y)|^2}{|x-y|^{d+2s}} \, \chi_{ \{|x-y| \le r \delta\} } \, d y dx,
\end{equation*}
valid for all $R > r \delta > 0$ and $s \in (0,1)$, implies
\begin{equation*}
  \iint_{D_R(0) \times D_R(0)} \frac{|v(x)-v(y)|^2}{|x-y|^{d+2s}} \, d y dx
  \le \frac{2(3R)^{2(1-s)}}{C_{d,s}(r \delta)^{2(1-s)} \phi_0 }\iii{v}^2
\end{equation*}
regardless of whether or not the support of $\phi$ is compact.
Given any $v \in \tHs$, we use this bound with $R>0$ sufficiently large so that $\Omega \subset D_R(0)$ and $\mbox{dist}(\Omega, \partial D_R(0)) \ge \frac{R}2$, exploit the fact that $|x-y| \ge  \frac{R}2$ for all $x \in \Omega$ and $y \in D_R(0)^c$ and integrate in polar coordinates to get
\[ \begin{split}
\|v\|_{\tHs}^2 & = \frac{C_{d,s}}{2}\iint_{D_R(0) \times D_R(0)} \frac{|v(x)-v(y)|^2}{|x-y|^{d+2s}} \, d y dx
    + C_{d,s} \iint_{D_R(0) \times D_R(0)^c} \frac{|v(x)|^2}{|x-y|^{d+2s}} \, d y dx
  \\
& \le \frac{(3R)^{2(1-s)}}{(r \delta)^{2(1-s)} \phi_0 } \iii{v}^2 +  \frac{2^{2s-1} C_{d,s} \omega_{d-1} }{s \, R^{2s}} \| v \|_{L^2(\Omega)}^2,
\end{split} \]
where $\omega_{d-1}=|\mathbb{S}^{d-1}|$ denotes the $(d-1)$-dimensional measure of the unit sphere $\mathbb{S}^{d-1}=\partial D_1(0)$ in $\Rd$. We next resort to the Poincar\'e inequality \eqref{eq:Poincare}
and fix $R>0$ in such a way that
\[
\frac{2^{2s-1} C_{d,s} \omega_{d-1} }{s \, R^{2s}} \| v \|_{L^2(\Omega)}^2 \le \frac12 \|v\|_{\tHs}^2 \quad \forall v \in \tHs,
\]
to  obtain a constant $C(d,\Omega,r, \phi_0, \delta)$ uniform in $s\in(0,1)$ such that
\begin{equation} \label{eq:FH-lower-bound}
\|v\|_{\tHs} \le C(d,\Omega,r, \phi_0, \delta) \, \iii{v}  \quad \forall v \in \tHs. 
\end{equation}

By combining \eqref{eq:FH-upper-bound} and \eqref{eq:FH-lower-bound}, we deduce that the energy norm $\iii{\cdot}$ is equivalent to the $\|\cdot\|_{\tHs}$-norm. Consequently, the Dirichlet problem for the operator $\mathcal A_{\delta,s}$ defined in \eqref{eq:def-finite-horizon} but scaled by $C_{d,s}$ is well-posed in $\tHs$ uniformly in $s$: if $f \in H^{-s}(\Omega)$, then there exists a unique $u \in \tHs$ so that
\begin{equation}\label{eq:weak-delta}
\scalardelta{u,v}_{\delta,s} = \langle \mathcal A_{\delta,s} u , v \rangle = \langle f , v \rangle \quad \forall v \in \tHs,
\end{equation}
and it satisfies $\| u \|_{\tHs} \lesssim \|f\|_{H^{-s}(\Omega)} $; hereafter $\scalardelta{\cdot,\cdot}_{\delta,s}$ is the scalar product associated with $\iii{\cdot}$.

Another consequence of the equivalence between the energy norm $\iii{\cdot}$ and the $\tHs$ norm is that we can adapt the technique employed in the proof of \Cref{T:Besov_regularity} (optimal shift property) to this finite-horizon problem. We refer to \cite{BoLiNo22} for further details.

\begin{corollary}[optimal shift property for finite-horizon operators]\label{C:regularity-finite-horizon}
Let  $\phi \colon [0,\infty) \to [0,\infty)$ be a bounded function of class $C^\beta$ on a neighborhood of the origin, for some $\beta \in (0,1]$, and satisfy $\phi(0) > 0$.
Let $\Omega$ be a bounded Lipschitz domain, $s \in (0,1)$, and $f \in B^{-s+\beta/2}_{2,1}(\Omega)$. Then, the unique function $u \in \tHs$ that solves \eqref{eq:weak-delta} belongs to the Besov space $\dot B^{s+\beta/2}_{2,\infty}(\Omega)$ and
\begin{equation} \label{eq:Besov_regularity_delta}
\| u \|_{\dot B^{s+\beta/2}_{2,\infty}(\Omega)} \le C(\Omega,d) \|f\|_{B^{-s+\beta/2}_{2,1}(\Omega)}.
\end{equation}
\end{corollary}
\begin{proof}
  It suffices to realize that the proof of \Cref{T:Besov_regularity} hinges on the action of the localized translation operator $T_h$ on generic balls $D_{2\rho}(x_0)$ of radius $2\rho$ and center $x_0$ via \eqref{eq:regularity-F1} and \eqref{eq:regularity-F2}; the only change is that \eqref{eq:regularity-F2} is now valid for $\sigma\in(0,\beta]$ due the presence of $\phi$. Consequently, if $\rho$ is sufficiently small so that $\phi\in C^\beta([0,4\rho])$ and $\phi\ge\phi_0>0$ on $[0,4\rho]$, then we end up with the following variant of \eqref{eq:iteration} regardless of the regularity of $\phi$ outside $[0,4\rho]$ and the size of its support:
\[
\|u\|_{\dot{B}_{2,\infty}^{\sigma_{k+1}}(\Omega)}^2 \le \Big( C_1 \|f\|_{B^t_{2,1}(\Omega)} + C_2 \|u\|_{\dot{B}_{2,\infty}^{\sigma_k}(\Omega)}   \Big) \|f\|_{B^t_{2,1}(\Omega)},
\]
with
\[
\sigma_{k+1} := s+\sigma/2, \quad
\sigma_k := \sigma - t, \quad
t := \frac{\beta}{2} - s.
\]
If $\sigma_0=s$, then we can prove by induction that $\sigma_k = s + \frac{\beta}{2} \big(1-\frac{1}{2^k} \big)$ and
$\sigma = \beta \big( 1 - \frac{1}{2^{k+1}} \big) \in (0,\beta]$. The rest of the proof follows that of \Cref{T:Besov_regularity}. 
\end{proof}

\begin{remark}[Sobolev estimates for finite-horizon operators]\label{R:Sobolev-delta}
In the same fashion as in \Cref{C:Sobolev-estimates} (Sobolev estimates), but with $\alpha:=\{s,\frac{\beta}{2}\}$, we can convert \eqref{eq:Besov_regularity_delta} into the Sobolev regularity estimates for all $\eps>0$ sufficiently small
\begin{equation}\label{eq:BesovSobolev-delta}
\| u \|_{\widetilde H^{s+\beta/2-\eps}(\Omega)} \le \frac{C(\Omega,d)}{\sqrt{\eps}} \|f\|_{B^{-s+\beta/2}_{2,1}(\Omega)},
\end{equation}
and
\begin{equation}\label{eq:SobolevL2-delta}
\| u \|_{\widetilde H^{s+\alpha-\eps}(\Omega)} \le \frac{C(\Omega,d,s)}{\eps^\kappa} \|f\|_{L^2(\Omega)};
\end{equation}
here $\kappa=\frac12$ for $s\ne\frac12$ and $\kappa=1$ for $s=\frac12$ and $C(\Omega,d,s)\approx|1-2s|^{-\frac12}$ for $s\ne\frac12$ provided $\beta=1$, whereas $\kappa=\frac12$ and the constant does not blow up in $s$ provided $\beta<1$. The latter is a consequence of $\sigma\le\beta<1$ in the proof of \Cref{T:L2-data} due to the functional $\calF_2$. These estimates are instrumental for the finite element error analysis of linear problems with finite horizon.
\end{remark}

\begin{remark}[assumptions on the diffusivity]
We stress that \eqref{eq:Besov_regularity_delta}, \eqref{eq:BesovSobolev-delta} and \eqref{eq:SobolevL2-delta} do not require global H\"older-$\beta$ regularity of the diffusivity function $\phi$ but just in a neighborhood of the origin. Thus, this result applies to the case of {\it truncated Laplacians} (that correspond to $\phi(\rho) = \chi_{[0,1]}(\rho)$) and thereby extends the estimates from \cite{Burkovska:19} to Lipschitz domains. To the best of our knowledge, these are the first Besov or Sobolev regularity estimates for finite-horizon operators with non-constant diffusivity $\phi$, whose support can even extend beyond $[0,1]$ provided $\phi$ is globally bounded, and valid on Lipschitz domains.
\end{remark}
  
%---------------------------------------------
\subsection{Direct finite element discretization}  \label{sec:FE}
%---------------------------------------------

We next consider a direct finite element discretizations of \eqref{eq:weak_linear} by using piecewise linear continuous functions.
Given $h_0 > 0$, for $h \in (0, h_0]$, we let $\mathcal{T}_h$ denote a triangulation of $\Omega$, i.e., $\calT = \{T\}$ is a conforming partition of $\Omega$ into simplices $T$ of diameter $h_T$. 
We assume the family $\{\calT \}$ to be shape-regular, namely,
\[
\sigma := \sup_{\calT} \max_{T \in \calT} \frac{h_T}{\rho_T} <\infty,
\]
where $\rho_T $ is the diameter of the largest ball contained in $T$ \cite{BrennerScott08}. As usual, the parameter $h$ denotes the mesh size, $h = \max_{T \in \calT} h_T$; moreover, we take elements to be closed sets.

Let $\mathcal{N}_h$ be the set of interior vertices of $\calT$, $N$ be its cardinality , and $\{ \varphi_i \}_{i=1}^N$ be the standard piecewise linear Lagrangian basis, with $\phii_i$ associated to the node $\x_i \in \mathcal{N}_h$ and star $S_i=\supp\phii_i$. The finite element space is the set of continuous piecewise linear functions over $\calT_h$,
\begin{equation} \label{eq:FE_space}
\mathbb{V}_h :=  \left\{ v \in C_0(\Omega) \colon v = \sum_{i=1}^N v_i \varphi_i \right\}.
\end{equation}
It is clear that $\mathbb{V}_h \subset \tHs$ for all $s \in (0,1)$ and therefore we have a conforming discretization.

With the notation described above, the discrete weak formulation reads: find $u_h \in \mathbb{V}_h$ such that
\begin{equation} \label{eq:weak_linear_discrete}
(u_h, v_h)_s = \langle f, v_h \rangle \quad \forall v_h \in \mathbb{V}_h.
\end{equation}
We now briefly describe some practical aspects of the implementation and solution process of \eqref{eq:weak_linear_discrete}. In view of \eqref{eq:FE_space}, the linear system associated with \eqref{eq:weak_linear_discrete} can be expressed as $\vA \vU = \vF$, where $\vU = (u_i)_{i=1}^N$ and $u_h = \sum_{i=1}^N u_i \varphi_i$, and the entries of the stiffness matrix $\vA=(a_{ij})_{ij=1}^N$ and right-hand side vector $\vF=(f_i)_{i=1}^N$ read
\[
a_{ij} = (\phii_i, \phii_j)_s, \quad f_i = \langle f, \phii_i \rangle.
\]

%----------------------------------------------------------------------------------
\subsubsection*{\bf Computation of the stiffness matrix}
%----------------------------------------------------------------------------------
There are two issues in dealing with the entries $\vA_{ij}$ ($i,j = 1 , \ldots , N$). The first one is that the integration domain in the bilinear form $(\phii_i, \phii_j)_s$ is unbounded if $S_i \cap S_j \neq \emptyset$. Namely, one has to compute integrals of the form
\[
\int_{S_i \cap S_j} \phii_i(x) \phii_j(x) \int_{S_i^c} \frac{dx}{|x-y|^{d+2s}} dy dx.
\]  
In the experiments we display below, we have used auxiliary exterior domains and Dirichlet data truncation as proposed in \cite{ABB}. We point out that --at least for homogeneous problems-- an efficient alternative to compute the integrals over $\Omega\times\Omega^c$ is to transform them into integrals over $\Omega\times\pp\Omega$ by means of the Divergence Theorem \cite{AiGl17}.

The second issue is that the singular (non-integrable) kernel $|x-y|^{-d-2s}$ offers significant difficulties when computing stiffness matrix entries corresponding to nodal basis functions with supports close to each other. Enforcing a degree of cancellation compatible with the accuracy of \eqref{eq:weak_linear_discrete} requires suitable quadrature rules to compute such entries. Reference \cite{ABB} adopts techniques from the boundary element method \cite{Chernov:11, SauterSchwab}.

%----------------------------------------------------------------------------------
\subsubsection*{\bf Matrix compression}
%----------------------------------------------------------------------------------
The finite element spaces \eqref{eq:FE_space} give rise to {\em full} stiffness matrices regardless of the value of $s \in (0,1)$. Indeed, if $i,j$ are such that $S_i \cap S_j = \emptyset$, then
\[
a_{ij} =  - \frac{C_{d,s}}2 \iint_{S_i \times S_j} \frac{\phii_i(x) \; \phii_j(y)}{|x-y|^{d+2s}} \; dy \; dx < 0.
\]
Thus, in a naive implementation most of the matrix assembly time is devoted to the computation of elements $a_{ij}$ for $\x_i$ and $\x_j$ far away from one another. However, these matrix elements should be significantly smaller than the ones that involve neighboring nodes. Therefore, in an efficient implementation of the finite element method, these far field contributions can be replaced by computationally cheaper low-rank blocks. The cluster paneling method from the boundary element literature has been considered in \cite{AiGl17, zhao2017adaptive}, resulting in a data-sparse representation with $\mathcal{O}(N \log^\alpha N)$ complexity for some $\alpha \ge 0$; see also \cite{Bauer:20}. We finally remark that in \cite{karkulik2018mathcal} it is shown that the inverse of $\vA$ can be represented using the same block structure as employed to compress $\vA$.

%----------------------------------------------------------------------------------
\subsubsection*{\bf Preconditioning}
%----------------------------------------------------------------------------------
The use of matrix factorization techniques to solve a dense matrix equation $\vA \vU = \vF$ has complexity $\mathcal{O}(N^3)$. A common alternative is the use of the conjugate gradient method, for which the number of iterations needed for a fixed tolerance scales like $\sqrt{\kappa(\vA)}$, where $\kappa(\vA)$ is the condition number of $\vA$. Reference \cite{AiMcTr99} shows that
\begin{equation} \label{eq:conditioning}
\kappa(\vA) = \mathcal{O}\left( N^{2s/d} \left( \frac{h_{max}}{h_{min}} \right)^{d-2s} \right), 
\end{equation}
where $h_{max}$ and $h_{min}$ denote the maximum and minimum element sizes, respectively. Therefore, on quasi-uniform meshes we have $\kappa(\vA) = \mathcal{O}(h^{-2s})$, where $h \simeq N^{-1/d}$ is a mean element diameter.
As we discuss in Section \ref{sec:graded}, the use of graded meshes yields a higher convergence rate with respect to the number of degrees of freedom. However, as \eqref{eq:conditioning} illustrates, such an improvement comes at the expense of poorer conditioning; a simple diagonal rescaling restores the same condition number as for uniform meshes \cite{AiMcTr99}.

There have been some recent progress in the development of preconditioners for fractional elliptic problems. Reference \cite{AiGl17} mentions the use of multigrid preconditioners, while \cite{gimperlein2019optimal} analyzes an operator preconditioner based on an explicit representation of the Green's function for the integral fractional Laplacian on a ball (Boggio's formula). 

Additive Schwarz preconditioners of BPX-type have been considered in \cite{faustmann2021stability,fractionalbpx}. Reference \cite{faustmann2021stability} develops a local multilevel diagonal preconditioner, while \cite{fractionalbpx} introduces a scaling for coarse spaces that yields condition numbers uniformly bounded with respect to both the number of levels and the order $s \in (0,1)$. We briefly discuss the latter next.

Since the $s$-Laplacian \eqref{eq:def_Laps} tends to the classical Laplacian as $s\to1$ and to the identity as $s\to0$, due to the suitable $s$-depending scaling built into $C_{d,s}$, the challenge is to design a preconditioner $\vB$ such that $\kappa (\vB\vA) \le C$ uniformly in $s$ as well as in the number of levels $J$. The standard BPX preconditioners cannot do this because $\vA$ is not a good approximation of the identity. The key idea of \cite{fractionalbpx} is an $s$-dependent scaling that captures the transition to the identity as $s\to0$. Let $V=\mathbb{V}_h$ and consider the space decomposition $\{V_j\}_{j=0}^J$ of $V$ with nested spaces $V_j$, $V=\sum_{j=1}^J V_j$, and the $L^2$-projection operators $Q_j:V\to V_j$. The following {\it $s$-uniform decomposition} is valid for all $\gamma,\wt{\gamma} \in (0,1)$, $s>0$, and for every $v\in V$:
\begin{equation}\label{eq:s-uniform}
\sum_{j=0}^J \gamma^{-2sj}\|(Q_j - Q_{j-1})v\|_0^2 =
\inf_{\substack{v_j \in V_j \\
\sum_{j=0}^J v_j = v}} \bigg( \gamma^{-2sJ}\|v_J\|_0^2 +  
\sum_{j = 0}^{J-1} \frac{\gamma^{-2sj}}{1 - \wt\gamma^{2s}}\|v_j\|_0^2
\bigg).
\end{equation}
The critical role of the weight $1 - \wt\gamma^{2s}$ is documented in Table \ref{T:comparison} for a simple example $\Omega = (0,1)^2$ and $f=1$ with $s=10^{-1}, 10^{-2}$: the number of iterations is not uniform for $\wt\gamma=0$!
\begin{table}[!ht]
\centering
\begin{tabular}{|c||c|c||c|c|}
  \hline
  \multicolumn{5}{|c|}{Uniform grids} \\ \hline
 \multirow{2}{*}{DOFs} 
  & \multicolumn{2}{c||}{$s=10^{-1}$}
  & \multicolumn{2}{c|}{$s=10^{-2}$} \\ \cline{2-5}
  &  $\tilde{\gamma} = 0$ & $\tilde{\gamma} = \frac12$  &
  $\tilde{\gamma} = 0$  & $\tilde{\gamma} = \frac12$ \\ \hline
 225 & 14 & 10 & 16 & 10 \\ 
961 & 17 & 10 & 18 & 10 \\ 
3969 & 19 & 10 & 21 & 10 \\ 
16129 & 20 & 10 & 23 & 9 \\ 
\hline  
\end{tabular}
\qquad
\begin{tabular}{|c||c|c||c|c|}
  \hline
  \multicolumn{5}{|c|}{Graded bisection grids} \\ \hline
 \multirow{2}{*}{DOFs} 
  & \multicolumn{2}{c||}{$s=10^{-1}$}
  & \multicolumn{2}{c|}{$s=10^{-2}$} \\ \cline{2-5}
  &  $\tilde{\gamma} = 0$ & $\tilde{\gamma} = \frac12$  &
  $\tilde{\gamma} = 0$  & $\tilde{\gamma} = \frac12$ \\ \hline
161 & 13 & 10 & 15 & 11 \\ 
853 & 17 & 12 & 19 & 13 \\ 
2265 & 20 & 12 & 22 & 14 \\ 
9397 & 22 & 13 & 25 & 14 \\ 
\hline  
\end{tabular}
\medskip
\caption{Number of iterations needed when using a PCG method with BPX
  preconditioner without ($\wt{\gamma} = 0$) and with
  ($\wt{\gamma} = \frac12$) a correction factor. We display results
  on a family of uniformly refined grids (left panel), and on a
  sequence of suitably graded bisection grids (right panel).}
\label{T:comparison}
\end{table}

Consider first the case of {\it quasi-uniform meshes} $\calT_j$ with meshsize $h_j$ and corresponding spaces $V_j$. If $I_j:V_j \to V$ is the injection operator, so that $Q_j=I_j^T$, then the additive preconditioner $B$ in operator form reads
\begin{equation}\label{eq:BPX-uniform}
B := I_{J} h_{J}^{2s} Q_{J} + (1-\wt{\gamma}^{s})\sum_{j=0}^{J - 1} I_j
h_j^{2s} Q_j.
\end{equation}
The performance of $B$, showing robustness with respect to $J$ and $s$, is displayed in Table \ref{T:BPX-uniform}.
\begin{table}[!ht]
\centering
\begin{tabular}{|m{0.3cm}|m{0.52cm}|m{1.0cm}||m{0.75cm}|
m{0.75cm}||m{0.75cm}|m{0.75cm}||m{0.75cm}|m{0.75cm}|}
  \hline
  \multirow{2}{*}{$J$} & \multirow{2}{*}{$h_{J}$} &
  \multirow{2}{*}{DOFs} 
  & \multicolumn{2}{c||}{$s=0.9$}
  & \multicolumn{2}{c||}{$s=0.5$} 
  & \multicolumn{2}{c|}{$s=0.1$} \\ \cline{4-9}
  &&  & CG  & PCG  & CG & PCG  & CG  & PCG 
  \\ \hline
  1 & $2^{-1}$ & 9   & 4 & 4 
                     & 4 & 4 
                     & 4 & 4 \\ \hline
  2 & $2^{-2}$ & 49  & 12 & 12 
                     & 8 & 8 
                     & 8 & 9 \\ \hline
  3 & $2^{-3}$ & 225  & 25 & 16 
                     & 11 & 10 
                     & 8 & 10 \\ \hline
  4 & $2^{-4}$ & 961  & 46 & 19 
                     & 17 & 11 
                     & 8 & 10 \\ \hline
  5 & $2^{-5}$ & 3969 & 84 & 21 
                     & 24 & 12 
                     & 8 & 10 \\ \hline
  6 & $2^{-6}$ & 16129 & 157 & 22 
                     & 32 & 13
                     & 8 & 10 \\ \hline
\end{tabular}
\medskip
\caption{Quasi-uniform meshes: Iterations of CG, and PCG with BPX
  preconditioner \eqref{eq:BPX-uniform}, parameter $\tilde{\gamma} = 0.5$, and fractional order $s=0.9, 0.5, 0.1$.}
\label{T:BPX-uniform}
\end{table}
\vskip-0.3cm
\begin{table}[!htbp]
\centering
\begin{tabular}{|m{0.32cm}|m{1.cm}||m{0.75cm}|m{0.75cm}
||m{0.75cm}|m{0.75cm}||m{0.75cm}|m{0.75cm}|}
  \hline
  \multirow{2}{*}{$\bar{J}$} & \multirow{2}{*}{DOFs} 
  & \multicolumn{2}{c||}{$s=0.9$}
  & \multicolumn{2}{c||}{$s=0.5$} 
  & \multicolumn{2}{c|}{$s=0.1$}\\ \cline{3-8}
  & & CG  & PCG  & CG & PCG & CG & PCG \\ \hline
  7&  61 & 10 & {10} & 10 & {7} 
  & 13 & {8}\\ \hline
  8 &  153 &15 &{13} & 15 & {9} 
  & 21 & {10}\\ \hline
  9 &  161 & 15 &{14} & 15 & {9} 
  & 21 & {10}\\ \hline
  10 &  369 & 20 &{16} & 20 & {11} 
  & 34 & {11}\\ \hline
  11 & 405  & 21 &{16} & 19 & {11} 
  & 31 & {12}\\ \hline
  12 & 853 & 26 &{18} & 26 & {12} 
  & 48 & {12}\\ \hline
  13 & 973 & 30 &{19} & 26 & {12} 
  & 47 & {12}\\ \hline
  14 & 1921 & 34 &{20} & 33 & {13} 
  & 72 & {12}\\ \hline
  15 & 2265 & 40 &{21} & 32 & {13} 
  & 65 & {12}\\ \hline
  16 & 4269 & 46 &{22} & 39 & {14} 
  & 97 & {13}\\ \hline
   17 & 5157 & 55 &{22} & 40 & {14} 
  & 92 & {12}\\ \hline
   18 & 9397 & 64 &{24} & 48 & {14} 
  & 135 & {13}\\ \hline
\end{tabular}
\medskip
\caption{Graded bisection meshes: Iterations of CG and PCG with BPX preconditioner
  \eqref{eq:BPX-graded}, parameter $\tilde{\gamma} = \sqrt{2}/2$, and fractional orders $s=0.9, 0.5, 0.1$.}
\label{T:PBX-graded} 
\end{table}

\vskip-0.1cm
Consider now {\it graded bisection meshes} $\calT_j=\calT_{j-1} + b_j = \calT_0 + \{b_1,\cdots,b_j\}$ obtained from an initial mesh $\calT_0$ by successive compatible bisections $b_j$. Associated with each $b_j$ there is a triplet of nodes, namely the bisection node $x_j$ and the parents nodes of $x_j$ at the end of the bisection edge containing $x_j$; $h_j$ is the local meshsize. The local space $V_j$ is the span of the three hat functions corresponding to this triplet on the mesh $\calT_j$. Let $\calP$ be the set of interior nodes of the finest graded mesh $\calT_J$ and let $V_p$ be the one dimensional space spanned by the hat function $\phi_p\in V$ associated with $p\in\calP$ and local meshsize $h_p$. Given the space decomposition
\[
V = \sum_{p\in\calP} V_p + \sum_{j=1}^J V_j,
\]
the additive preconditioner $B$ has a similar structure to \eqref{eq:BPX-uniform} and reads
\begin{equation}\label{eq:BPX-graded}
B = \sum_{p\in \calP} I_p h_p^{2s} Q_p + (1-\wt{\gamma}^s) \sum_{j=0}^J I_j h_j^{2s} Q_j.
\end{equation}
Table \ref{T:PBX-graded} documents the robust performance of $B$ with respect to $J$ and $s$.

The robust performance of $B$ is supported by theory. The following result is shown in \cite{fractionalbpx}. The proof for quasi-uniform meshes $\calT_j$ hinges on \eqref{eq:s-uniform} as well as $s$-uniform interpolation and local inverse estimates. However, for graded meshes there is no global notion of scale and the subspaces $V_j$ are local and non-nested, so \eqref{eq:s-uniform} cannot be used directly. The proof relies on the geometric structure of bisection grids.

\begin{theorem}[uniform preconditioning]\label{T:BPX}
The BPX preconditioner of \eqref{eq:BPX-uniform} and \eqref{eq:BPX-graded} satisfies $\kappa(\vB \vA) \le C$, where the constant $C$ is independent of the number of levels $J$ and fractional order $s$.
\end{theorem}

%---------------------------------------------------------------------------
\subsection{Global error estimates}\label{S:global-error-estimates}
%---------------------------------------------------------------------------
In this section, we review global error estimates for the finite element discretization \eqref{eq:weak_linear_discrete} in $\wt{H}^s(\Omega)$, $L^2(\Omega)$ and negative-order Sobolev norms.

The {\it energy norm} in $\wt{H}^s(\Omega)$ satisfies the best approximation property,
\begin{equation} \label{eq:best_approximation_linear}
\|u - u_h \|_{\tHs} = \min_{v_h \in \mathbb{V}_h} \|u - v_h \|_{\tHs},
\end{equation}
because $u_h$ is the projection of $u$ onto $\mathbb{V}_h$ with respect to such a norm. Therefore, we must account for nonlocality of the fractional norm in $\wt{H}^s(\Omega)$ as well as the regularity of $u$. We present error estimates for quasi-uniform meshes and suitably graded meshes that compensate for the singular boundary behavior of $u$. The rates of convergence improve upon previous results because they hinge on the new regularity estimates of Theorem \eqref{T:Besov_regularity} (optimal shift property).

%---------------------------------------------------------------------------
\subsubsection{\bf Localization and interpolation estimates}
%---------------------------------------------------------------------------
A fundamental tool in the error analysis is the use of interpolation estimates. In order to obtain estimates valid for arbitrary meshes, a typical approach is to derive them elementwise (or patchwise), thereby giving rise to local interpolation estimates. Since the energy norm is {\em nonlocal,} a localization procedure is thus required.

We start by definining the star (or patch) of a subset $A \subset \Omega$ by
\[
  S_A^1 := \bigcup \big\{ T \in \calT \colon T \cap A \neq \emptyset \big\}.
\]
Given $T \in \calT$, the star $S_T^1$ of $T$ is the first ring of $T$ and the star $S_T^2$ of $S_T^1$
is the second ring of $T$. References \cite{Faermann2, Faermann} derive the following localized estimate:
\begin{equation} \label{eq:localization-F}
  |v|_{H^s(\Omega)}^2 \leq \sum_{T \in \calT} \left[ \int_T \int_{S_T^1} \frac{|v (x) - v (y)|^2}{|x-y |^{d+2s}} \; dy \; dx + \frac{C(d,\sigma)}{s h_T^{2s}} \| v \|^2_{L^2(T)} \right] \quad \forall v \in H^s(\Omega).
\end{equation}
Therefore, to estimate an $H^s(\Omega)$-seminorm, it suffices to compute local contributions on patches of the form $\{T \times S_T^1 \}_{T \in \calT }$ and elementwise $L^2$-contributions. The energy norm in our problem is not exactly the $H^s(\Omega)$ one but the $\tHs$ one instead. The latter involves integration over the space $\Rd$ even if the functions are supported in $\Omega$, and thus we need to make some slight modifications on \eqref{eq:localization-F}. Following \cite{BoNo21Constructive}, given $T \in \calT$ we denote its barycenter by $x_T$ and consider a ball $B_T=B(x_T,Ch_T)$ with center $x_T$ and radius $Ch_T$, where $C=C(\sigma)$ is a shape regularity dependent constant such that $S_T^1\subset B_T$. Then, we define the {\it extended} stars
\[
\widetilde{S}_T^1 :=
\left\lbrace \begin{array}{rl}
  S_T^1   & \textrm{if } T \cap \partial\Omega = \emptyset,
  \\
  B_T    & \textrm{otherwise,}
\end{array}\right.
\]
and the extended second ring $\widetilde{S}_T^2$,
\[
\widetilde{S}_T^2 := \bigcup \big\{ \widetilde{S}_{T'}^1 \colon T' \in \calT, \mathring{T}' \cap S_T^1 \neq \emptyset \big\}.
\]
With these modifications, \cite{BoNo21Constructive} proves the following estimate:
\begin{equation} \label{eq:localization-F-tilde}
  \|v\|_{\tHs}^2 \leq \sum_{T \in \calT} \left[ \int_T \int_{\widetilde S_T^1} \frac{|v (x) - v (y)|^2}{|x-y |^{d+2s}} \; dy \; dx + \frac{C(d,\sigma)}{s h_T^{2s}} \| v \|^2_{L^2(T)} \right] \quad \forall v \in \tHs.
\end{equation}

Let us now mention some standard local quasi-interpolation estimates, involving a suitable quasi-interpolation operator that we shall denote by $\Pi$. Examples of possible choices for $\Pi$ include the Scott-Zhang and the Cl\'ement interpolation operators \cite{BrennerScott08,CiarletJr}. One can prove the following local estimates (see, for example, \cite{AcBo17,BoNoSa18,CiarletJr}) in either standard or weighted Sobolev spaces.

\begin{lemma}[local interpolation error estimates]\label{L:interpolation-error}
Let $T \in \calT$, $s \in (0,1)$, $r \in (s, 2]$, and $\Pi$ be a suitable local quasi-interpolation operator. If $v \in W^r_p (\widetilde{S}_T^2)$, then
\begin{equation} \label{eq:interpolation}
  \int_T \int_{\widetilde S_T^1} \frac{|(v-\Pi v) (x) - (v-\Pi v) (y)|^2}{|x-y|^{d+2s}} \, d y \, d x \le C \, h_T^{2\l}
  |v|_{W^r_p(\widetilde{S}_T^2)}^2,
\end{equation}
where $\l = \sob(W^r_p) - \sob(H^s) = r -s - d \big(\frac{1}{p} - \frac{1}{2} \big) > 0$
and $C = C(\Omega,d,s,\sigma,r)$.

Moreover, considering the weighted Sobolev scale \eqref{eq:weighted_sobolev}, it holds that for all $v \in H^r_\gamma (\widetilde S_{T}^2)$,
\begin{equation} \label{eq:weighted_interpolation}
 \int_T \int_{\widetilde S_T^1} \frac{|(v-\Pi v) (x) - (v-\Pi v) (y)|^2}{|x-y|^{d+2s}} \, d y \, d x \le C
   h_T^{2(r-s-\gamma)} |v|_{H^r_\gamma(\widetilde S_{T}^2)}^2.
\end{equation}
\end{lemma}

An important feature of either \eqref{eq:localization-F} and \eqref{eq:localization-F-tilde} is that, when applied to an interpolation error --that typically has zero mean over elements--, the scaled $L^2$-terms can be converted into $H^s$ terms by means of a Poincar\'e inequality. This means that, when measuring interpolation errors, the $H^s(\Omega)$-seminorm effectively localizes as the sum of patchwise $H^s$-seminorms.

We finally point out that \eqref{eq:localization-F-tilde} avoids the need to use the Hardy inequality. This is particularly important in the case $s=1/2$, because such a step gives rise to an additional logarithmic factor
\cite{BoLeNo21,BoNoSa18}.

%----------------------------------------------------------------------------------
\subsubsection{\bf Quasi-uniform meshes}\label{S:quasi-uniform}
%----------------------------------------------------------------------------------
Combining the best approximation property \eqref{eq:best_approximation_linear} with regularity, localization and local interpolation estimates, one can derive the following quasi-optimal a priori estimates in the energy norm. The first estimate is new and the second one improves upon \cite{BoLeNo21} the power of the logarithmic factor for all $s\in(0,1)$.

\begin{theorem}[energy error estimates] \label{T:conv_linear_uniform}
Let $\Omega \subset \R^d$ be a bounded Lipschitz domain, $0<s<1$, and let $u$ denote the solution of \eqref{eq:weak_linear} and $u_h \in \mathbb{V}_h$ the solution of the discrete problem \eqref{eq:weak_linear_discrete}, computed over a quasi-uniform mesh $\calT_h$, where $h = \max_{T \in \calT} h_T$. If $f \in B^{1/2-s}_{2,1}(\Omega)$, then there exists a constant $C=C(\Omega,d,\sigma)$ such that for all $s\in(0,1)$
\begin{equation} \label{eq:conv_Hs-Besov}
\|u - u_h \|_{\tHs}  \le C \,  h^{\frac12} |\log h|^{\frac12} \, \|f\|_{B^{1/2-s}_{2,1}(\Omega)}.
\end{equation}
If instead $f \in L^2(\Omega)$ and $\alpha = \min \{s, \frac{1}{2} \}$, then for $\kappa=\frac12$ if $s\ne\frac12$
and $\kappa=1$ if $s=\frac12$ we have
\begin{equation} \label{eq:conv_Hs-L2}
\|u - u_h \|_{\tHs}  \le C \,  h^\alpha |\log h|^\kappa \, \|f\|_{L^2(\Omega)}.
\end{equation}
The constant $C=C(\Omega,d,s,\sigma)$ blows up as $s\to\frac12$ as $|1-2s|^{-\frac12}$.
\end{theorem}
\begin{proof}
Employing \eqref{eq:best_approximation_linear} in conjunction with the interpolation estimates \eqref{eq:localization-F-tilde} and \eqref{eq:interpolation} and the Sobolev regularity estimate \eqref{eq:regularity-SobBesov} yields
\[
\|u - u_h \|_{\tHs}  \le C h^{\frac12-\eps}  \|u\|_{\wt{H}^{s+\frac12-\eps}(\Omega)}
\le C h^{\frac12} \frac{h^{-\eps}}{\eps^{\frac12}} \|f\|_{B^{1/2-s}_{2,1}(\Omega)}.
\]
Taking $\eps=|\log h|^{-1}$ gives the desired estimate \eqref{eq:conv_Hs-Besov}. Replacing \eqref{eq:regularity-SobBesov} by \eqref{eq:regularity-SobL2} if $s\ne\frac12$ or by \eqref{eq:regularity-SobL2-s=1/2} if $s=\frac12$ leads to \eqref{eq:conv_Hs-L2}, and completes the proof.
\end{proof}

With error estimates in $\widetilde{H}^s(\Omega)$ at hand, one can perform an Aubin-Nitsche duality argument to derive convergence rates in weaker norms. To see this, let $g\in L^2(\Omega)$ or $g\in H^{s-\frac12}(\Omega)$ in case $s\in(0,\frac12)$. Let $u_g \in \wt{H}^s(\Omega)$ solve $(v,u_g)_s = \langle v, g\rangle$ for all $v\in\wt{H}^s(\Omega)$. We then have
\begin{equation} \label{eq:Aubin-Nitsche}
\langle u-u_h,g \rangle = (u-u_h,u_g)_s = (u-u_h,u_g-\Pi u_g)_s \le \|u-u_h\|_{\wt{H}^s(\Omega)}
\|u_g-\Pi u_g\|_{\wt{H}^s(\Omega)}.
\end{equation}

\begin{proposition}[$L^2$-error estimate]\label{prop:Aubin-Nitsche}
Let $\Omega$ be a bounded Lipschitz domain, $s\in(0,1)$, and $\alpha = \min \{s, \frac{1}{2} \}$. If $f \in B^{1/2-s}_{2,1}(\Omega)$, then for $\kappa$ defined in Theorem \ref{T:conv_linear_uniform} we have
\begin{equation*}\label{eq:L2Besov}
\|u - u_h\|_{L^2(\Omega)} \le C h^{\frac12+\alpha} | \log h |^{\frac12+\kappa} \, \| f \|_{B^{1/2-s}_{2,1}(\Omega)}.
\end{equation*}
If $f \in L^2(\Omega)$ instead, then we have
\begin{equation*} \label{eq:L2L2}
\|u - u_h\|_{L^2(\Omega)} \le C h^{2\alpha} | \log h |^{2\kappa} \, \| f \|_{L^2(\Omega)}.
\end{equation*}
\end{proposition}
\begin{proof}
Choose $g\in L^2(\Omega)$ in \eqref{eq:Aubin-Nitsche} and use either \eqref{eq:regularity-SobL2} or \eqref{eq:regularity-SobL2-s=1/2} for $u_g$ to get
\[
\|u_g-\Pi u_g\|_{\wt{H}^s(\Omega)} \lesssim h^\alpha |\log h|^\kappa \|g\|_{L^2(\Omega)}.
\]
Combining this estimate with either \eqref{eq:conv_Hs-Besov} or \eqref{eq:conv_Hs-L2} gives the asserted error bounds.
\end{proof}

For classical second order problems, the pivot space is $H^1_0(\Omega)$ and the pick-up regularity for piecewise linear elements is just $1$. Therefore, the accessible regularity is $H^2(\Omega)$ and by duality the lowest order space is just $L^2(\Omega)$. Getting higher-order estimates in negative Sobolev spaces entails increasing the polynomial degree beyond $1$. For the fractional Laplacian the situation is different: the pivot space is $\wt{H}^s(\Omega)$ and the pick-up regularity is $\frac12$. Therefore, the accessible regularity is $H^{s+\frac12-\eps}(\Omega)$ and the lowest order space accessible by duality is $H^{s-\frac12}(\Omega)$, which turns out to be a negative Sobolev space if $s<\frac12$. The following convergence rates improve upon Proposition \ref{prop:Aubin-Nitsche}.

\begin{proposition}[error estimate in negative Sobolev spaces]\label{prop:negativeSob}
Let $\Omega$ be a bounded Lipschitz domain and $s\in(0,\frac12)$. If $f \in B^{1/2-s}_{2,1}(\Omega)$, then we have
\begin{equation*}
\|u - u_h\|_{H^{s-\frac12}(\Omega)} \le C h | \log h |^{\frac32} \, \| f \|_{B^{1/2-s}_{2,1}(\Omega)}.
\end{equation*}
If $f \in L^2(\Omega)$ instead, then we have
\begin{equation*}
\|u - u_h\|_{{H}^{s-\frac12}(\Omega)} \le C h^{\frac12+s} | \log h |^{\frac32} \, \| f \|_{L^2(\Omega)}.
\end{equation*}
\end{proposition}
\begin{proof}
We invoke the Sobolev estimate
\[
\|v\|_{\dot{B}^{-s+1/2-\eps}_{2,1}(\Omega)} \lesssim \eps^{-\frac12} \|v\|_{\wt{H}^{-s+1/2}(\Omega)}\quad\forall \, v\in \wt{H}^{-s+1/2}(\Omega),
\]
which follows by duality and an argument with the K-functional and spaces $X_0=H^{-1}(\Omega)$ and $X_1=H^1_0(\Omega)$ similar to the second part of the proof of Theorem \ref{T:L2-data} and the proof of Corollary \ref{C:Sobolev-estimates}.

We take $g \in \wt{H}^{-s+1/2}(\Omega)$ in \eqref{eq:Aubin-Nitsche} and combine \eqref{eq:regularity-SobBesov} with the preceding estimate to arrive at
\begin{equation*}
\|u_g-\Pi u_g\|_{\wt{H}^s(\Omega)} \lesssim h^{\frac12-2\eps} |u_g|_{\wt{H}^{s+\frac12-2\eps}(\Omega)}
\lesssim \frac{h^{\frac12-2\eps}}{\eps^{\frac12}}  \| g \|_{\dot{B}^{-s+1/2-\eps}_{2,1}(\Omega)}
\lesssim \frac{h^{\frac12-2\eps}}{\eps} \| g \|_{\wt{H}^{-s+1/2}(\Omega)}.
\end{equation*}  
This, in conjunction with either \eqref{eq:conv_Hs-Besov} or \eqref{eq:conv_Hs-L2} and $\eps=|\log h|^{-1}$, gives the desired bounds.
\end{proof}

On the other hand, by combining inverse inequalities and interpolation estimates, one can also derive convergence estimates on {\em higher-order} seminorms such as the $H^1(\Omega)$-seminorm \cite{BoCi19}. We remark that the restriction $s>1/2$ below is due to the fact that under such a condition one can guarantee that the solution $u$ is actually in $H^1(\Omega)$. The condition on $f$ is weaker than in \cite{BoCi19} because here we are exploiting Theorem \ref{T:Besov_regularity} (optimal shift property).

\begin{proposition}[$H^1$-error estimate]\label{prop:H1-convergence}
Let $\Omega$ be a bounded Lipschitz domain, $f \in B^{-s+1/2}_{2,1}(\Omega)$, and $s \in (1/2,1)$. If $h$ is sufficiently small, then we have
\begin{equation*}
\|u - u_h\|_{H^1(\Omega)} \le C h^{s-1/2} | \log h |^{1/2} \| f \|_{B^{-s+1/2}_{2,1}(\Omega)}.
\end{equation*}
\end{proposition}
\begin{proof}
Let $\eps \in (0,s-1/2)$. Combining the interpolation error estimates \eqref{eq:localization-F-tilde} and \eqref{eq:interpolation}  with the Sobolev regularity estimate \eqref{eq:regularity-SobBesov}, we obtain
\[
\begin{aligned}
\|u-\Pi u\|_{H^1(\Omega)} \lesssim \frac{h^{s-1/2-\eps}}{\sqrt{\eps}} \|f\|_{B^{-s+1/2}_{2,1}(\Omega)}.
\end{aligned}
\]
It remains to bound $\|\Pi u-u_h\|_{H^1(\Omega)}$. By using the standard inverse inequality 
\[
\| v_h \|_{H^1(\Omega)} \le C h^{s-1} \| v_h \|_{\wt{H}^s(\Omega)} \quad \forall v_h \in \mathbb{V}_h,
\]
and the triangle inequality, it follows 
\[
\|\Pi u-u_h\|_{H^1(\Omega)} \le C h^{s-1}\left( \| \Pi u-u\|_{\widetilde H^s(\Omega)} + \|u-u_h\|_{\widetilde H^s(\Omega)} \right) .
\]
Finally, we resort again to \eqref{eq:localization-F-tilde}, \eqref{eq:interpolation}, and the energy error estimate \eqref{eq:conv_Hs-Besov}, to deduce
\[
\|u-u_h\|_{H^1(\Omega)} \lesssim \frac{h^{s-1/2-\eps}}{\sqrt{\eps}} \|f\|_{B^{-s+1/2}_{2,1}(\Omega)}, \quad \forall \eps \in (0, s-1/2).
\]
The proof is concluded upon setting $\eps = | \log h |^{-1}$ in the estimate above.
\end{proof}

%-----------------------------------------------------------------------------
\subsubsection{\bf Graded meshes} \label{sec:graded}
%-----------------------------------------------------------------------------
%
Prior convergence rates on quasi-uniform meshes suffer from boundary pollution.
A natural remedy to improve upon them is to consider graded meshes adapted to the boundary behavior of the solution to problem \eqref{eq:Dirichlet}, and to exploit the regularity estimates from \Cref{T:weighted_regularity} or \Cref{T:higher-regularity-Lp} for their design and analysis. This is the objective of this section.

The construction of graded meshes \`a-la-Grisvard hinges on weighted Sobolev estimates \cite{Grisvard}. We let the parameter $h$ represent a local meshsize in the interior of $\Omega$, and assume that the family of meshes $\{\calT\}$ is shape-regular and admits a parameter $\mu\ge1$ such that for every $T \in \calT$,
\begin{equation} \label{eq:H}
 h_T \leq C(\sigma) \left\lbrace
 \begin{array}{rl}
   h^\mu & \mbox{if } T \cap \partial \Omega \neq \emptyset, \\
   h \, \dist(T,\pp \Omega)^{(\mu-1)/\mu} & \mbox{if } T \cap \partial \Omega = \emptyset,
 \end{array} \right.
\end{equation}
for $d\ge1$. This yields mesh cardinality (see \cite{Babuska:79, BoNoSa18})
\begin{equation} \label{eq:dofs}
\#\calT \approx \dim \mathbb{V}_h \approx
\left\lbrace
 \begin{array}{rl}
    h^{-d} & \mbox{if } \mu < \frac{d}{d-1}, \\
    h^{-d} | \log h | & \mbox{if } \mu = \frac{d}{d-1}, \\
    h^{(1-d)\mu} & \mbox{if } \mu > \frac{d}{d-1},
  \end{array} \right.
\end{equation}
for $d>1$; for $d=1$ the relation $\dim \mathbb{V}_h \approx h^{-1}$ is valid regardless of the value of $\mu$.
Although \eqref{eq:H} provides a sufficient grading condition, it does not guarantee the existence of such meshes, especially for complicated geometries. We will discuss a constructive approach below (cf. \Cref{alg:greedy}).

If $\mu \le \frac{d}{d-1}$, then the interior mesh size $h$ and $\#\calT$ satisfy the optimal relation $h \simeq \#\calT^{-1/d}$ (up to logarithmic factors if $\mu = \frac{d}{d-1}$). To derive optimal convergence rates for meshes satisfying \eqref{eq:H}, one needs to tune the parameter $\mu$; the optimal choice of $\mu$ depends on the dimension; we refer to \cite{BoLeNo21} for details. In two dimensions, the optimal choice is $\mu = 2$. We combine \eqref{eq:localization-F-tilde} with either \eqref{eq:weighted_interpolation} or \eqref{eq:interpolation}, depending on whether $S_T^2$ intersects $\partial\Omega$ or not, and Theorem \ref{T:weighted_regularity} (weighted Sobolev estimate), to derive the following result.

\begin{theorem}[energy error estimates on graded meshes] \label{T:conv_linear_graded}
Let $\Omega \subset \R^2$ be a bounded Lipschitz domain satisfying an exterior ball condition, and $u \in \tHs$ denote the solution to \eqref{eq:weak_linear} and $u_h \in \mathbb{V}_h$ denote the solution of the discrete problem \eqref{eq:weak_linear_discrete}, computed over a mesh $\calT$ satisfying \eqref{eq:H} with $\mu = 2$. If $f \in C^{1-s}(\overline{\Omega})$, then we have
\begin{equation} \label{eq:conv_Hs_graded}
\|u - u_h \|_{\tHs}  \le C(\Omega,s,\sigma)  
h |\log h| \|f\|_{C^{1-s}(\overline{\Omega})}.
\end{equation}
Equivalently, in terms of mesh cardinality $\#\calT$, the estimate above reads
\begin{equation*} \label{eq:conv_Hs_graded_N}
\|u - u_h \|_{\tHs}  \le C(\Omega,s,\sigma)  
(\#\calT)^{-\frac{1}{2}} |\log \#\calT|^{\frac{3}{2}} \|f\|_{C^{1-s}(\overline{\Omega})}.
\end{equation*}
\end{theorem}

Since the practical implementation of meshes satisfying \eqref{eq:H} might be problematic, \cite{BoNo21Constructive} proposes
a constructive algorithm based on the {\it bisection method}, which works as follows. Given a mesh $\calT$, we assume every simplex $T\in\calT$ has an edge $e(T)$ marked for refinement. To subdivide $T$ into two children  $T_1,T_2$ such that $T = T_1 \cup T_2$, one connects the midpoint of $e(T)$ with the vertices of $T$ that do not lie in $e(T)$. If every simplex sharing $e(T)$ has $e(T)$ marked for refinement, then the patch is {\em compatible} and the refinement does not propagate beyond it. Otherwise, at least one element in the patch has an edge other than $e(T)$ marked for refinement, and the refinement procedure must go outside the patch to maintain conformity (namely, we have a nonlocal step). Therefore, two natural questions arise:
\begin{enumerate}[$\bullet$]
\item {\it Completion}:
How many elements other than $T$ must be refined to keep the mesh conforming?  
  
\item {\it Termination}:   
Does this procedure terminate?
\end{enumerate}

To guarantee termination, a special labeling of the initial mesh $\calT_0$ is required
(a suitable choice of the edge $e(T)$ for each element $T\in\calT_0$).
Completion is rather tricky to assess and was done by
P. Binev, W. Dahmen and R. DeVore for $d=2$ \cite{BinevDahmenDeVore:2004}
and R. Stevenson for $d>2$ \cite{Stevenson:2008}; we refer to the
surveys \cite{NochettoSiebertVeeser:2009,NochettoVeeser:2012} for a rather complete
discussion.

Given the $j$-th refinement $\calT_j$ of $\calT_0$ and a subset $\calM_j\subset\calT_j$ of elements marked
for bisection, 
\[
\calT_{j+1} = \REFINE (\calT_j,\calM_j)
\]
is a procedure that
creates the smallest conforming refinement $\calT_{j+1}$ of $\calT_j$ upon bisecting all elements of
$\calM_j$ at least once and perhaps additional elements to keep conformity. We point out that
it is simple to construct counterexamples to the estimate
\[
\#\calT_{j+1} - \#\calT_j \le \Lambda \; \#\calM_j
\]
where $\Lambda$ is a universal constant independent of $j$; see \cite[Section 1.3]{NochettoVeeser:2012}.
However, this can be repaired upon considering the cumulative effect of a sequence of conforming
bisection meshes $\{\calT_j\}_{j=0}^k$ for any $k$. In fact, the following crucial estimate is valid \cite{BinevDahmenDeVore:2004,Stevenson:2008} (see also \cite{NochettoSiebertVeeser:2009,NochettoVeeser:2012})
\begin{equation}\label{eq:BDD}
\#\calT_k - \#\calT_0 \le \Lambda \; \sum_{j=0}^{k-1} \#\calM_j
\end{equation}

We propose a greedy algorithm to choose simplices for refinement \cite{BoNo21Constructive}. For $d=2$, the basic idea is to equidistribute the local $H^s$-interpolation errors, for which one assumes access to the quantities
\begin{equation} \label{eq:error-estimator}
E_T(u) = C h_T^t R_T(u), \qquad R_T(u) = |u|_{W^{s+1-\eps}_{1+\eps}(\widetilde S^2_T)}  ,
\end{equation}
with $C$ being a constant depending on the mesh shape-regularity and $t = 2 - \eps - \frac{2}{1+\eps} >0$. We point out that such regularity was stated in Theorem \ref{T:higher-regularity-Lp} (differentiability vs integrability).
Given a tolerance $\delta>0$ and a conforming mesh $\calT_0$ with suitable labeling, the following Algorithm \ref{alg:greedy} finds
a conforming refinement $\calT$ of $\calT_0$ by bisection
such that $E_T(u)\le\delta$ for all $T\in\calT$. 

%let $\calT=\calT_0$ and
%
%\medskip
%\begin{algotab}
%  \> $\GREEDY (\calT,\delta)$\\
%  \>  while $\calM :=\{T\in\calT : \,E_T(u) > \delta\}\ne\emptyset$\\
%  \> \> $\calT := \REFINE(\calT,\calM)$\\
%  \> end while\\
%  \> return($\calT$)
%\end{algotab}

\begin{algorithm}
	\caption{$\GREEDY (\calT_0,\delta)$}
	\label{alg:greedy}
	\begin{algorithmic}%[1]
	\State Let $\calT=\calT_0$.
		\While{ $\calM :=\{T\in\calT : \,E_T(u) > \delta\}\ne\emptyset$ }
		\State $\calT := \REFINE(\calT,\calM)$
		\EndWhile	
\State \Return($\calT$)
	\end{algorithmic}
\end{algorithm}

\begin{theorem}[quasi-optimal error estimate on bisection meshes]\label{T:greedy}
Let $\Omega \subset \R^2$ be a polygonal domain. If $u\in \widetilde W^{s+1-\eps}_{1+\eps}(\Omega)$ satisfies
\eqref{eq:higher-regularity} with $\beta=1-s$,  then
$\GREEDY$ terminates in finite steps and the resulting isotropic mesh $\calT$ satisfies
\begin{equation}\label{eq:error-est}
  \| u - u_h \|_{\tHs} \le C(\Omega,s) \big(\#\calT\big)^{-\frac{1}{2}}
  |\log \#\calT|^{3} \|f\|_{C^{1-s}(\overline{\Omega})}.
\end{equation}
\end{theorem}
\begin{proof}
We sketch the main steps and refer to \cite[Theorem 4.5]{BoNo21Constructive} for details. Finite termination is guaranteed by $t>0$. To facilitate counting, we split the set of all marked elements $\calM = \cup_{j=0}^{k-1} \calM_j$ into the disjoint sets $\calP_j$ of elements $T\in\calM$ with size $h_T := |T|^{1/d}$ satisfying
\[
2^{-(j+1)} <|T| \le 2^{-j}
\quad\Rightarrow\quad
2^{-(j+1)/d} < h_T \le 2^{-j/d}.
\]
Exploiting the definition of $\calP_j$ and that $E_T(u) > \delta$ for all $T\in\calP_j$ we deduce the key properties
\[
\#\calP_j \le |\Omega| 2^j,
\qquad
\#\calP_j \le C \big( \delta^{-1} R(u) \big)^{1+\eps} 2^{-jt(1+\eps)/d},
\qquad
R(u) = \|u\|_{\wt{W}^{s+1-\eps}_{1+\eps}(\Omega)}.
\]
These complementary bounds allow us to split $\sum_j\#\calP_j$ into $j\le j_0$ and $j>j_0$, where $j_0$ is the smallest index for which the second term is smaller than the first one. This simple trick minimizes the counting and yields $\#\calM = \sum_j\#\calP_j \lesssim \big( \delta^{-1} R(u)  \big)^{1+\eps}$. This in conjunction with \eqref{eq:BDD} gives
\[
\#\calT - \calT_0 \lesssim \big( \delta^{-1} R(u)  \big)^{1+\eps}
\quad\Rightarrow\quad
\delta \lesssim \big( \#\calT \big)^{-\frac{1}{1+\eps}} R(u),
\]
provided $\#\calT \ge 2 \#\calT_0$. Upon termination of $\GREEDY$ we have $E_T(u)\le\delta$ for all $T\in\calT$, whence
\[
\|u-\Pi_\calT u\|_{\wt{H}^s(\Omega)}^2 \lesssim \delta^2 \#\calT \lesssim \big( \#\calT  \big)^{-1+\frac{2\eps}{1+\eps}} R(u)^2
\]
with $\Pi_\calT$ a local quasi-interpolant.
In view of \eqref{eq:higher-regularity}, we see that $R(u) \lesssim \eps^{-3} \|f\|_{C^{1-s}(\Omega)}$ and
\[ 
\|u-\Pi_\calT u\|_{\wt{H}^s(\Omega)} \lesssim \big(\#\calT \big)^{-\frac12}
\frac{\big(\#\calT \big)^{\frac{\eps}{1+\eps}}}{\eps^3}
\|f\|_{C^{1-s}(\Omega)}.
\]
The desired estimate \eqref{eq:error-est} follows upon choosing $\eps=|\log \#\calT|^{-1}$ and applying \eqref{eq:best_approximation_linear}.
\end{proof}

\begin{remark}[practical estimator] \label{rmk:practical-estimator}
For computtaional purposes, the error estimator \eqref{eq:error-estimator} is not practical. One can replace it with the geometric quantity 
\begin{equation*} 
\mathcal{E}_T(u) = |T| \dist(x_T,\pp\Omega)^{-1} ;
\end{equation*}
we refer to \cite[Section 5]{BoNo21Constructive} for details. The subordinate $\GREEDY$ algorithm to the surrogate estimator $\mathcal{E}_T(u)$ exhibits similar convergence rates to the one in \Cref{T:conv_linear_graded}.
\end{remark}

\begin{remark}[convergence in dimensions $d\ne2$]
We point out that both Theorem \ref{T:conv_linear_graded} (energy error estimates on graded meshes) and Theorem \ref{T:greedy} (quasi-optimal error estimate on bisection meshes) can also be extended to dimensions $d \ne 2$, cf. \cite[Theorem 3.5]{BoLeNo21} and \cite[Theorem 4.5]{BoNo21Constructive}. In the former, an optimal choice of the mesh grading parameter turns out to provide convergence with order $2-s$ in $d=1$ and $\frac1{2(d-1)}$ in $d\ge 2$ (up to logarithmic terms), with respect to $\# \calT$. In the latter, \GREEDY \, delivers the same convergence rates with respect to $\#\calT$, although with a higher power in the logarithmic factor.
\end{remark}

We conclude this section with error estimates in weaker norms than the energy norm. It is important to realize that we pick up additional powers of the global meshsize $h$ rather than the local meshsize $h_T$. This is due to the fact that the duality argument is not completely local unless the meshsize changes slowly; see \cite[Section 0.8]{BrennerScott08} for details for $d=s=1$.

\begin{proposition}[error estimates in weaker norms]\label{P:weaker-norms}
Let $\Omega \subset \R^2$ be a bounded Lipschitz domain satisfying an exterior ball condition, and $u \in \tHs$ denote the solution to \eqref{eq:weak_linear} and $u_h \in \mathbb{V}_h$ denote the solution of the discrete problem \eqref{eq:weak_linear_discrete}, computed over a mesh $\calT$ satisfying \eqref{eq:H} with $\mu = 2$. If $f \in C^{1-s}(\overline{\Omega})$, $\alpha=\min\{s,\frac12\}$, then for $\kappa=\frac12$ if $s\ne\frac12$
and $\kappa=1$ if $s=\frac12$ we have
\begin{equation*} \label{eq:L2-_graded}
\|u - u_h \|_{L^2(\Omega)}  \le C(\Omega,s,\sigma)  
h^{1+\alpha} |\log h|^{1+\kappa} \|f\|_{C^{1-s}(\overline{\Omega})},
\end{equation*}
and if $s<\frac12$
\begin{equation*} \label{Hs-1/2-graded}
\|u - u_h \|_{H^{s-\frac12}}  \le C(\Omega,s,\sigma)  
h^{\frac32} |\log h|^2 \|f\|_{C^{1-s}(\overline{\Omega})}.
\end{equation*}
\end{proposition}
\begin{proof}
We simply carry out the duality argument \eqref{eq:Aubin-Nitsche} and estimate the adjoint solution $u_g\in\wt{H}^s(\Omega)$ as in the proofs of Proposition \ref{prop:Aubin-Nitsche} ($L^2$-error estimate) and Proposition \ref{prop:negativeSob} (error estimate in negative Sobolev spaces). This, together with \eqref{eq:conv_Hs_graded}, gives the assertions.
\end{proof}

%-----------------------------------------------------------------------------
\subsubsection{\bf Problems with finite horizon} \label{sec:finite-horizon-errors}
%-----------------------------------------------------------------------------
%
The conforming finite element discretization of the weak formulation \eqref{eq:weak-delta} with continuous piecewise linear functions $\mathbb{V}_h$ is similar to \eqref{eq:weak_linear_discrete}, namely
\begin{equation}\label{eq:discrete-delta}
u_h\in\mathbb{V}_h: \quad \scalardelta{u_h,v_h}_{\delta,s} = \langle f, v_h \rangle \quad\forall v_h \in \mathbb{V}_h.
\end{equation}
Since $\scalardelta{\cdot,\cdot}_{\delta,s}$ is a scalar product equivalent to $(\cdot,\cdot)_s$ in $\wt{H}^s(\Omega)$, Lax-Milgram guarantees the existence of a unique solution $u_h$. The following novel error estimates mimic those in Section \ref{S:quasi-uniform}.

\begin{theorem}[error estimates for quasi-uniform meshes]\label{T:error-finite-horizon}
Let $\Omega$ be a bounded Lipschitz domain and $\delta>0$ be the horizon.
Let $\phi \colon [0,\infty) \to [0,\infty)$ be a bounded function of class $C^\beta$ on a neighborhood of the origin, for some $\beta \in (0,1]$, and satisfy $\phi(0) > 0$. If $\alpha = \min\{s,\frac{\beta}{2}\}$, then the solutions $u \in \tHs$ of \eqref{eq:weak-delta} and $u_h\in\mathbb{V}_h$ of \eqref{eq:discrete-delta} satisfy the error estimates
\begin{gather*}
  \|u - u_h \|_{\tHs}  + h^{-\alpha} |\log h|^{-\kappa} \|u - u_h \|_{L^2(\Omega)} \le C_1 \,  h^{\frac12} |\log h|^{\frac12} \, \|f\|_{B^{1/2-s}_{2,1}(\Omega)},
  \\
  \|u - u_h \|_{\tHs}  + h^{-\alpha} |\log h|^{-\kappa} \|u - u_h \|_{L^2(\Omega)} \le C_2 \,  h^\alpha |\log h|^\kappa \, \|f\|_{L^2(\Omega)},
\end{gather*}
where $\kappa = \frac12$ if $s\ne\frac12$ and $\kappa=1$ if $s=\frac12$ provided $\beta=1$, whereas $\kappa=\frac12$ if $\beta<1$. Moreover, the constant $C_1=C_1(\Omega,d,\sigma,\delta)$ and $C_2=C_2(\Omega,d,\sigma,s,\delta)$ blows up as $s\to\frac12$ as $|1-2s|^{-\frac12}$ for $\beta=1$.
\end{theorem}
\begin{proof}
Argue as in \Cref{T:conv_linear_uniform} (energy error estimates) and \Cref{prop:Aubin-Nitsche} ($L^2$-error estimate) but utilizing instead \Cref{C:regularity-finite-horizon} (optimal shift property for finite-horizon operators) and \Cref{R:Sobolev-delta} (Sobolev estimates for finite-horizon operators).
\end{proof}

%-----------------------------------------------------------------------------
\subsection{\bf Local error estimates}\label{S:local-error-estimates}
%-----------------------------------------------------------------------------
%
The reduced convergence rates in Theorem \ref{T:conv_linear_uniform} and Propositions \ref{prop:Aubin-Nitsche}  and \ref{prop:H1-convergence} are essentially due to the boundary behavior of solutions. A natural question is whether it is possible to obtain better convergence rates in the interior of the domain. Such a question has been recently addressed in \cite{BoLeNo21,Faustmann:20}. We discuss this next.

%-----------------------------------------------------------------------------
\subsubsection{\bf Caccioppoli estimate}\label{S:Caccioppoli}
%-----------------------------------------------------------------------------
%
This estimate, well-known for second-order PDEs, quantifies the property that solutions do not oscillate. Its derivation is particularly simple and revealing for harmonic functions. If $u\in H^1(\Omega)$ satisfies $\Delta u =0$ in the ball $D_R=D_R(0) \subset \R^d$ of radius $R$ centered at the origin, then
\[
\int_{D_{R/2}} |\nabla u|^2 \lesssim \frac{1}{R^2} \int_{D_R} u^2.
\]
To prove it, let $\eta \in C_0^\infty(D_R)$ be a cut-off function such that $\eta = 1$ in $D_R$, $\eta = 0$ in $D_{3R/4}^c$ and $|\nabla \eta| \lesssim R^{-1}$. Since $\int_{D_R} \nabla u \cdot \nabla v = 0$ for all $v \in H^1(D_R)$, taking $v=\eta^2u \in H^1(D_R)$ yields
\[
0 = \int_{D_R} \nabla u \cdot \nabla (\eta^2 u) = \int_{D_R} |\nabla (\eta u)|^2 - \int_{D_R} u^2 |\nabla \eta|^2,
\]
whence
\[
\int_{D_{R/2}} |\nabla u|^2 \le \int_{D_R} |\nabla (\eta u)|^2 = \int_{D_R} u^2 |\nabla\eta|^2 \lesssim R^{-2}\int_{D_R} u^2.
\]

This simple estimate extends to local $s$-harmonic functions \cite{Cozzi17}. In fact, if $u$ satisfies
$(u,v)_s = 0$ for all $v \in \wt{H}^s(D_R)$ and $\int_{D_R^c}\frac{u(x)}{|x|^{d+2s}} dx < \infty$, then
\[
|u|_{{H}^s(D_{R/2})}^2 \lesssim R^{-2s} \|u\|_{L^2(D_R)}^2 + R^{d+2s} \left(\int_{D_R^c}\frac{u(x)}{|x|^{d+2s}} dx\right)^2.
\]
We consider now sets $\Omega_0 \Subset \Omega_1 \Subset \Omega$ such that $0 <R < \dist(\Omega_0,\partial\Omega_1)$, and a finite covering of $\Omega_0$ with balls of radius $R/2$ centered at points in $\Omega_0$. The localization property \eqref{eq:localization} enables us to extend the previous Caccioppoli estimate as
\[
|u|_{{H}^s(\Omega_0)}^2 \lesssim R^{-2s} \|u\|_{L^2(\Omega_1)}^2 + R^{d+2s} \left(\int_{\Omega_1^c}\frac{u(x)}{|x|^{d+2s}} dx\right)^2,
\]
with a hidden constant depending on the covering cardinality. Indeed, for any ball $D_{R/2} = D_{R/2}(x_j)$ in the covering, we split the integral over $D_R^c$ and apply H\"older's inequality to deduce
\[ \begin{split}
\left(\int_{D_R^c}\frac{u(x)}{|x-x_j|^{d+2s}} dx\right)^2 & \le
2 \left(\int_{\Omega_1 \setminus D_R}\frac{u(x)}{|x-x_j|^{d+2s}} dx\right)^2 
+ 2 \left(\int_{\Omega_1^c}\frac{u(x)}{|x-x_j|^{d+2s}} dx\right)^2 \\
& \le \frac{2\omega_{d-1}}{d+4s} R^{-(d+4s)} \| u \|_{L^2(\Omega_1\setminus D_R)}^2 
+ 2  \left(\int_{\Omega_1^c}\frac{u(x)}{|x-x_j|^{d+2s}} dx\right)^2 .
\end{split} \]

%-----------------------------------------------------------------------------
\subsubsection{\bf Local energy estimates}\label{S:local-estimates}
%-----------------------------------------------------------------------------
%
Local estimates for FEMs go back to the seminal paper by J. Nitsche and A. Schatz \cite{NitscheJA_SchatzAH_1974a} for second order linear PDEs; see the survey \cite{Wahlbin91}. Such estimates are typically of the following form: the energy error in a subdomain $\Omega_0$ is bounded by the interpolation error on a larger subdomain $\Omega_1$ and a pollution (or slush) term that involves a lower-order norm of the error. This structure is instrumental to characterize the pollution effect in $\Omega_0$ due to singularities remote from $\Omega_0$. It is thus natural to wonder whether such a localization is actually possible for the nonlocal problems at hand and, as a consequence, whether the boundary singularity propagates inside the domain $\Omega$ or not. This question has been recently studied in \cite{BoLeNo21,Faustmann:20}.

We start with the estimates from \cite{BoLeNo21}, which measure the slush term with a global $L^2$-norm. Its proof is a nontrivial discrete version of the above Caccioppoli estimate.

\begin{theorem}[local estimates with $L^2$-slush term]\label{P:slush-L2}
Let $\Omega$ be a Lipschitz domain, $u\in\wt{H}^s(\Omega)$ be the solution of \eqref{eq:weak_linear}, and $u_h\in\mathbb{V}_h $ satisfy the local Galerkin orthogonality condition
%be discrete $s$-harmonic in $B_R$, namely
%
\[
(u-u_h,v_h)_s = 0 \quad\forall \, v_h \in \mathbb{V}_h (D_R).
\]
If $\calT_h$ is a shape-regular graded mesh such that $16 h_T \le R$ for all $T\in\calT_h$, $T\subset D_R$, then for all $v_h \in \mathbb{V}_h$ we have
\[
|u-u_h|_{H^s(D_{R/2})} \lesssim \inf_{v_h\in\mathbb{V}_h} \Big( |u-v_h|_{H^s(D_{R})} +  R^{-s}\|u-v_h\|_{L^2(\Omega)} \Big)
+ R^{-s}\|u-u_h\|_{L^2(\Omega)}.
\]
\end{theorem}

Combining this estimate with the localization property \eqref{eq:localization} yields the following local estimates for subdomains:
\[
|u-u_h|_{H^s(\Omega_0)} \lesssim \inf_{v_h\in\mathbb{V}_h} \Big( |u-v_h|_{H^s(\Omega_1)} +  R^{-s}\|u-v_h\|_{L^2(\Omega)} \Big)
+ R^{-s}\|u-u_h\|_{L^2(\Omega)}.
\]
We are now in a position to compare the global error $\|u-u_h\|_{\tHs}$ and the local error $|u-u_h|_{H^s(\Omega_0)}$, where $\Omega_0$ is an interior subdomain of $\Omega$. For {\it quasi-uniform meshes} and $f \in B^{-s+1/2}_{2,1}(\Omega)$ or smoother, the interior estimates exhibit an improvement rate $h^{\min\{s,1/2\}}$ regardless of the regularity of $\Omega$. We summarize this in Table \ref{T:uniform-meshes}, which neglects logarithmic factors for clarity.

\begin{table}[h!] 
\centering
{\begin{tabular}{|c||c|c||c|c|}
\hline 
& \multicolumn{2}{|c||}{Interior rates} & \multicolumn{2}{|c|}{Global rates} \\ \cline{2-5}
& $\Omega$-smooth & $\Omega$-Lipschitz & $\Omega$-smooth & $\Omega$-Lipschitz \\ \hline
$s \le \frac{1}{2}$ & $h^{s+\frac{1}{2}}$ & $h^{s+\frac12}$  & $h^{\frac{1}{2}}$ & $h^{\frac12}$ \\
$s > \frac{1}{2}$ & $h$ & $h$  & $h^{\frac{1}{2}}$ & $h^{\frac{1}{2}}$ \\  
\hline
\end{tabular}}
\medskip
\caption{The interior estimates exhibit an improvement rate $h^{\min\{s,1/2\}}$ over global estimates for $f\in B^{-s+1/2}_{2,1}(\Omega)$ and qusi-uniform meshes regardless of the regularity of $\Omega$, according to Theorem \ref{P:slush-L2}.}
\label{T:uniform-meshes}
\end{table}

In two dimensions, for {\it graded meshes} satisfying the condition $h_T\approx h \, \dist(T,\partial\Omega)^{1/2}$, as discussed in \eqref{eq:H}, the interior estimates exhibit an improvement rate $h^{\min\{s,1-s\}}$ for $\Omega$ either smooth or Lipschitz with an exterior ball condition (e.b.c.). This is documented in Table \ref{T:graded-meshes}.
\begin{table}[h!] 
\centering
{\begin{tabular}{|c||c||c|}
\hline 
& \multicolumn{2}{|c|}{$\Omega$-smooth or Lipschitz e.b.c.} \\ \cline{2-3}
& { Interior rates } & {Global rates} \\ \hline
$s \le \frac{1}{2}$ & $h^{s+1}$ & $h$ \\
$s > \frac{1}{2}$ & $h^{2-s}$ & $h$ \\  
\hline
\end{tabular}}
\medskip 
\caption{The interior estimates exhibit an improvement rate $h^{\min\{s,1-s\}}$ over global estimates for $f\in C^{1-s}(\overline{\Omega})$ and graded meshes for $\Omega$ smooth or Lipschitz satisfying the exterior ball conditon (e.b.c), according to Theorem \ref{P:slush-L2}.}
\label{T:graded-meshes}
\end{table}

In contrast to \cite{BoLeNo21}, the estimates of \cite{Faustmann:20} express the slush term in the $H^{s-1/2}$-norm. This is of interest in case $s<1/2$, and is actually optimal, because the duality argument used to estimate the $H^{s-1/2}$-norm exploits the maximal regularity of the dual problem; see Proposition \ref{prop:negativeSob} (error estimate in negative Sobolev spaces). Another relevant difference between \cite{BoLeNo21} and \cite{Faustmann:20} regards mesh grading: there is no restriction in \cite{BoLeNo21}, which might in turn be of independent interest, whereas the assumption $h_{\textrm{max}}h_{\textrm{min}}^{-2}\lesssim 1$ is required in \cite{Faustmann:20} and agrees with Remark 3.3. The latter is consistent with the mesh grading \eqref{eq:H} with $\mu=2$, which is optimal for $d=2$. We now present a scaled variant of the local estimate from \cite[Theorem 2.3]{Faustmann:20}.

\begin{proposition}[local estimates with $H^{s-1/2}$-slush term]\label{P:slush-Hs-1/2}
Let $\Omega\subset\R^d$ be a domain that satisfies the shift property $\|u\|_{\wt{H}^{1/2+s-\eps}(\Omega)} \le C \|f\|_{H^{1/2-s-\eps}(\Omega)}$ with a constant $C=C(\Omega,d,s,\eps)$. On shape-regular meshes satisfying $h_{\textrm{max}}h_{\textrm{min}}^{-2}\lesssim 1$ and $h_{\textrm{max}}\le 10 R$, we have
\[
| u - u_h |_{H^s(\Omega_0)} \lesssim \left( \inf_{v_h \in \mathbb{V}_h} | u - v_h |_{H^s(\Omega_1)}
+ R^{-s} \|u-v_h\|_{L^2(\Omega_1)} \right) +  R^{-\frac12} \| u - u_h \|_{H^{s-1/2}(\Omega)}.
\]
Moreover, in case $u \in H^1(\Omega_1)$ we also have
\[
| u - u_h |_{H^1(\Omega_0)} \lesssim  \inf_{v_h \in \mathbb{V}_h} \left( | u - v_h |_{H^1(\Omega_1)}
+ \|u-v_h\|_{L^2(\Omega_1)} \right) +  R^{s-\frac32} \| u - u_h \|_{H^{s-1/2}(\Omega)}.
\]
\end{proposition}

We observe that the above shift property is a consequence of Remark \ref{R:intermediate} (intermediate estimate) with $r=\frac12-2\eps$ for $\Omega$ Lipschitz. We also point out that, ignoring logatitmic factors, the interior rates improve by a power $h^{\frac12-s}$ for $s<\frac12$, namely they become $h$ for quasi-uniform meshes instead of $h^{s+\frac12}$ and $h^{\frac32}$ instead of $h^{1+s}$ for graded meshes with respect to Tables \ref{T:uniform-meshes} and \ref{T:graded-meshes}.

%---------------------------------------------
\subsection{Computational examples}
%---------------------------------------------
%
We now present two numerical experiments that explore further two important points discussed earlier, namely that the 
boundary layer \eqref{eq:boundary_behavior} is generic irrespective of the forcing $f$ and domain regularity
and the performance of $\GREEDY$ established in \Cref{T:greedy} (quasi-optimal error estimate on bisection meshes).

\begin{example}[boundary layer effect]\label{ex:boundary-layer}
We asserted that the boundary behavior of the special solution \eqref{eq:getoor} with $f=1$ on the unit sphere is generic. We now probe this assertion in the square $\Omega = (-1, 1)^2$ with excentric right-hand side compactly supported in $\Omega$
\begin{equation}\label{eq:ex-linear-f}
f(x_1, x_2) = -(r^2 - (x_1 - a)^2 - x_2^2)_+
\end{equation}
\begin{figure}[!htb]
	\begin{center}
		\includegraphics[width=0.45\linewidth]{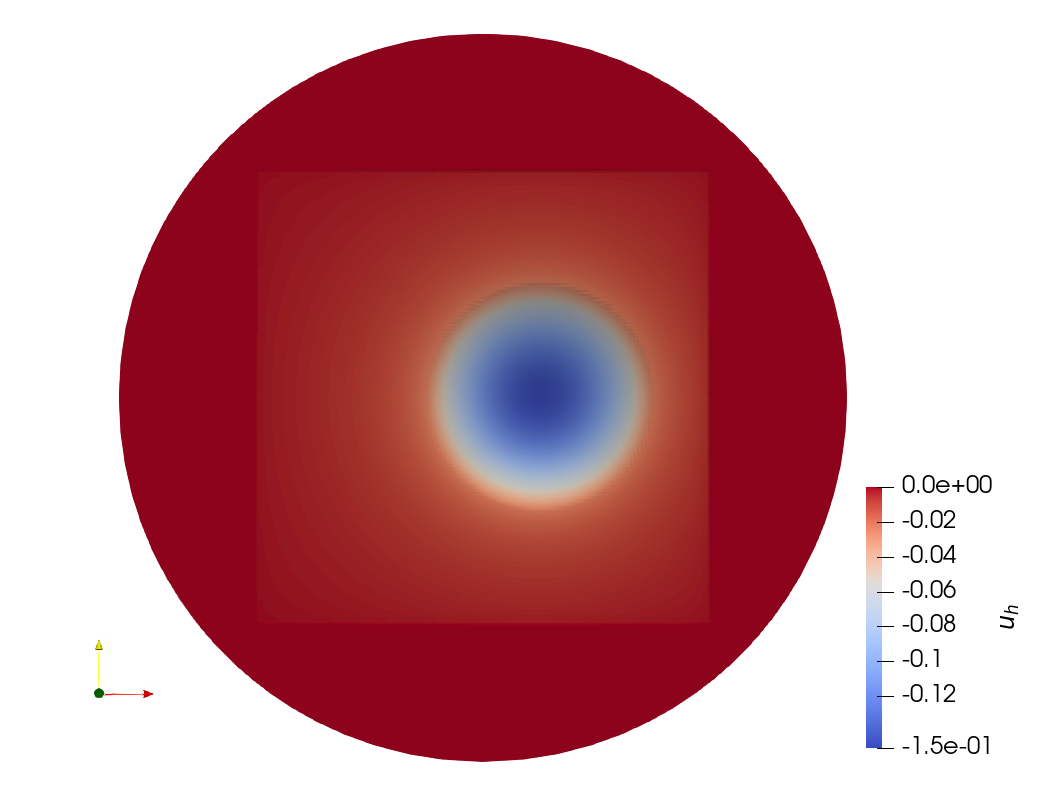}
		\includegraphics[width=0.45\linewidth]{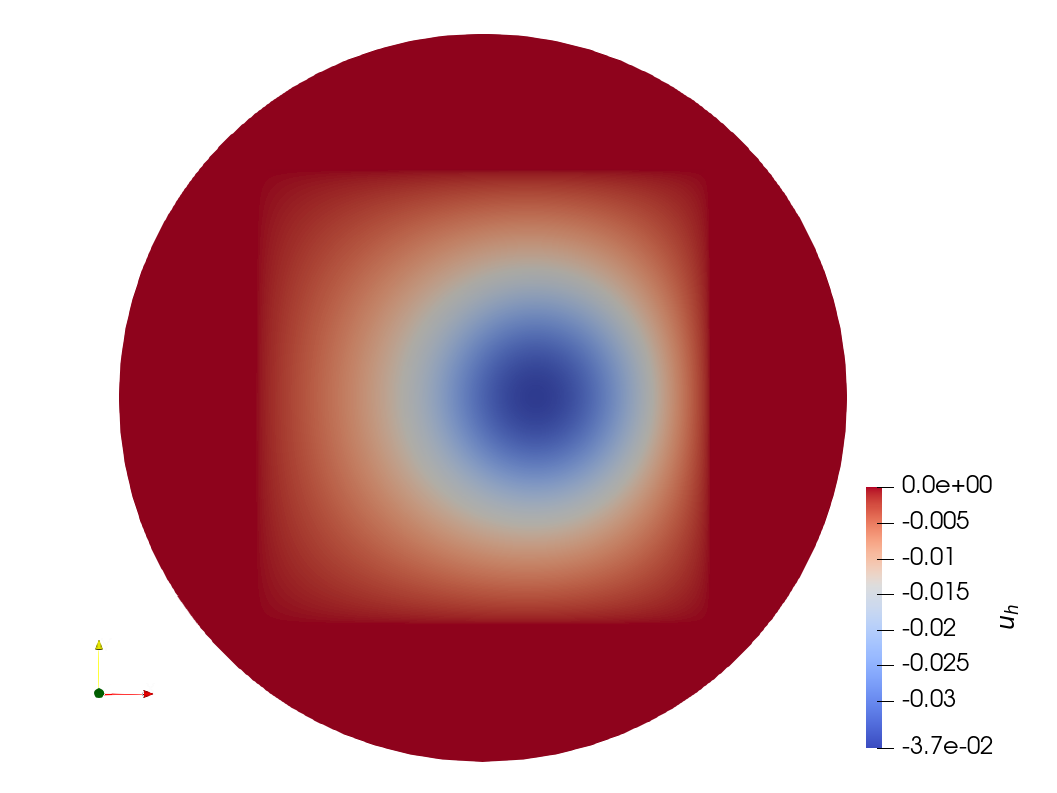} 
	\end{center}
	\caption{Numerical solutions of Example \ref{ex:boundary-layer} for $s = 0.2$ (left panel) and for $s = 0.8$ (right panel). The former is much less diffusive than the latter.} 
	\label{fig:frac-Lap-2d-linear}
\end{figure}
with $a = r = 0.25$. In \Cref{fig:frac-Lap-2d-linear}, we plot the numerical solutions for $s=0.2, 0.8$ on a uniform mesh with size $h=2^{-6}$. To numerically study the boundary behavior, we assume the solution can be approximated by 
\[
	u(x) \approx \dist(x, \partial \Omega)^{\alpha(s)},
\]
where $\alpha(s)$ is to be determined, and consider mesh points near the boundary point $(-1,0)$ along the slice $\{x_2 = 0\}$. We use numerical solutions to fit the power $\alpha(s)$ and report the results we obtain in \Cref{table:frac-Lap-2d-linear-power}.  Even though $f$ vanishes in a neighborhood of $\partial\Omega$, we observe a good agreement of $\alpha(s)$ with the boundary behavior \eqref{eq:boundary_behavior}.
        
	\begin{table}[htbp]
		\centering\scalebox{0.95}{
			\begin{tabular}{|c |c| c| c| c| c| c| c| c| c| c|}
				\hline 
				Value of $s$ & 0.1 & 0.2 & 0.3 & 0.4 & 0.5 & 0.6 & 0.7 & 0.8 & 0.9 \\ \hline
				Value of $\alpha(s)$ &  0.118 & 0.216 & 0.315 & 0.414 & 0.513 & 0.612 &  0.709 &  0.807 & 0.903 \\ \hline  				
			\end{tabular}
		}
                \vskip0.2cm
		\caption{Example \ref{ex:boundary-layer}: Exponents for different values of $s$. We observe the boundary behavior \eqref{eq:boundary_behavior}, which is intrinsic to the integral fractional Laplacian. }
		\label{table:frac-Lap-2d-linear-power}
	\end{table}	
\end{example}

\begin{example}[quasi-optimal convergence with $\GREEDY$ algorithm] \label{ex:linear-greedy}
We present an example from \cite{BoNo21Constructive} to illustrate \Cref{T:greedy} (quasi-optimal error estimate on bisection meshes). We solve \eqref{eq:Dirichlet} on the $L$-shaped domain $\Omega = (-1,1)^2 \setminus ([0,1) \times (-1,0])$ with $f=1$ and $s \in \{ 0.25, 0.5, 0.75\}$, using the finite element setting from \Cref{sec:FE} and the MATLAB code from \cite{ABB} to assemble the resulting stiffness matrices. We construct a family of bisection grids by employing an adaptive mesh refinement algorithm with a greedy marking strategy based on the package provided in \cite{Funken:11}. As an error estimator, we use the surrogate quantity $\mathcal{E}_T(u)$ described in \Cref{rmk:practical-estimator} (practical estimator).

\begin{figure}[htbp]
\begin{center}
 \includegraphics[width=0.34\linewidth]{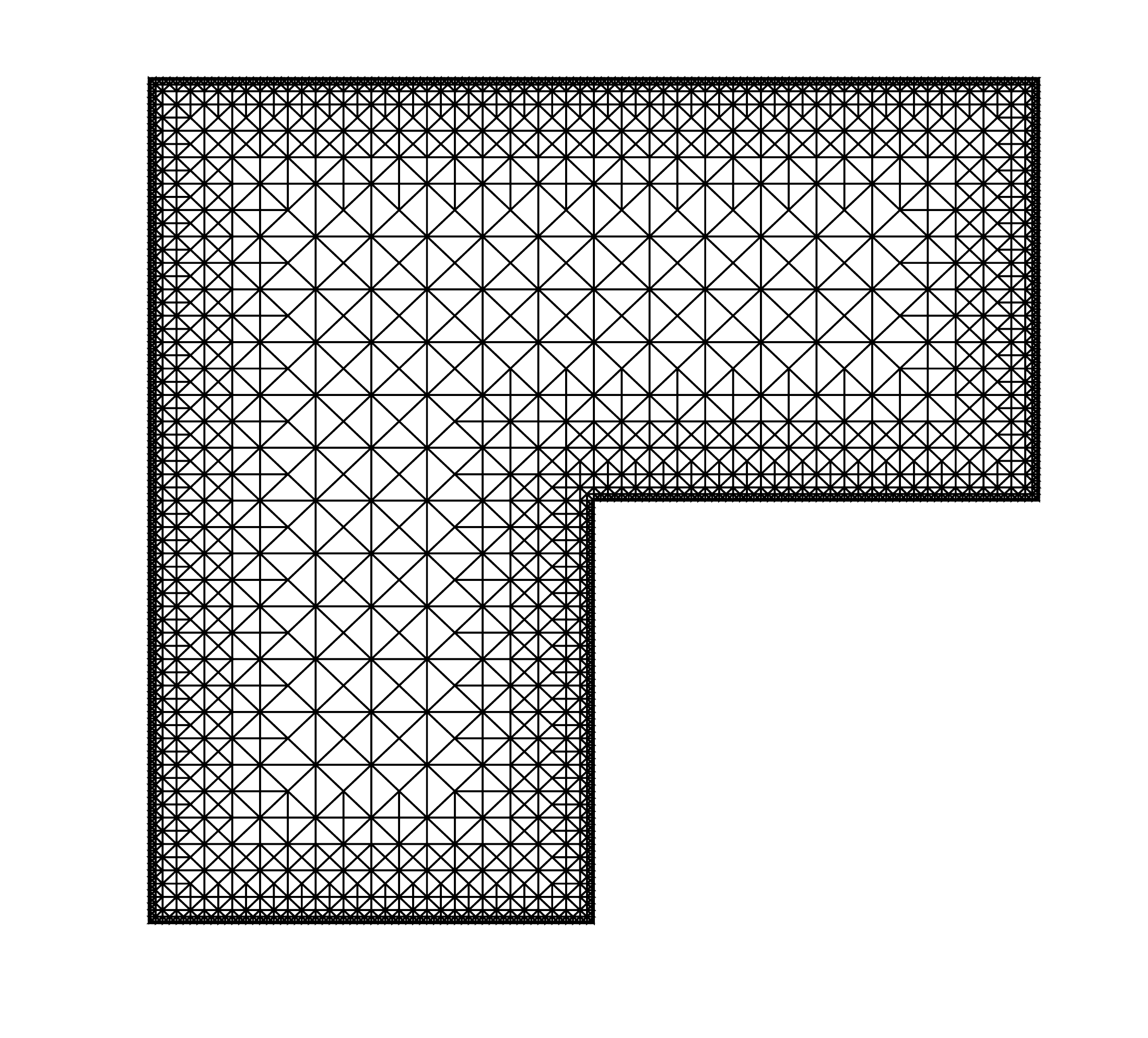} \hspace{-.5cm}
 \includegraphics[width=0.34\linewidth]{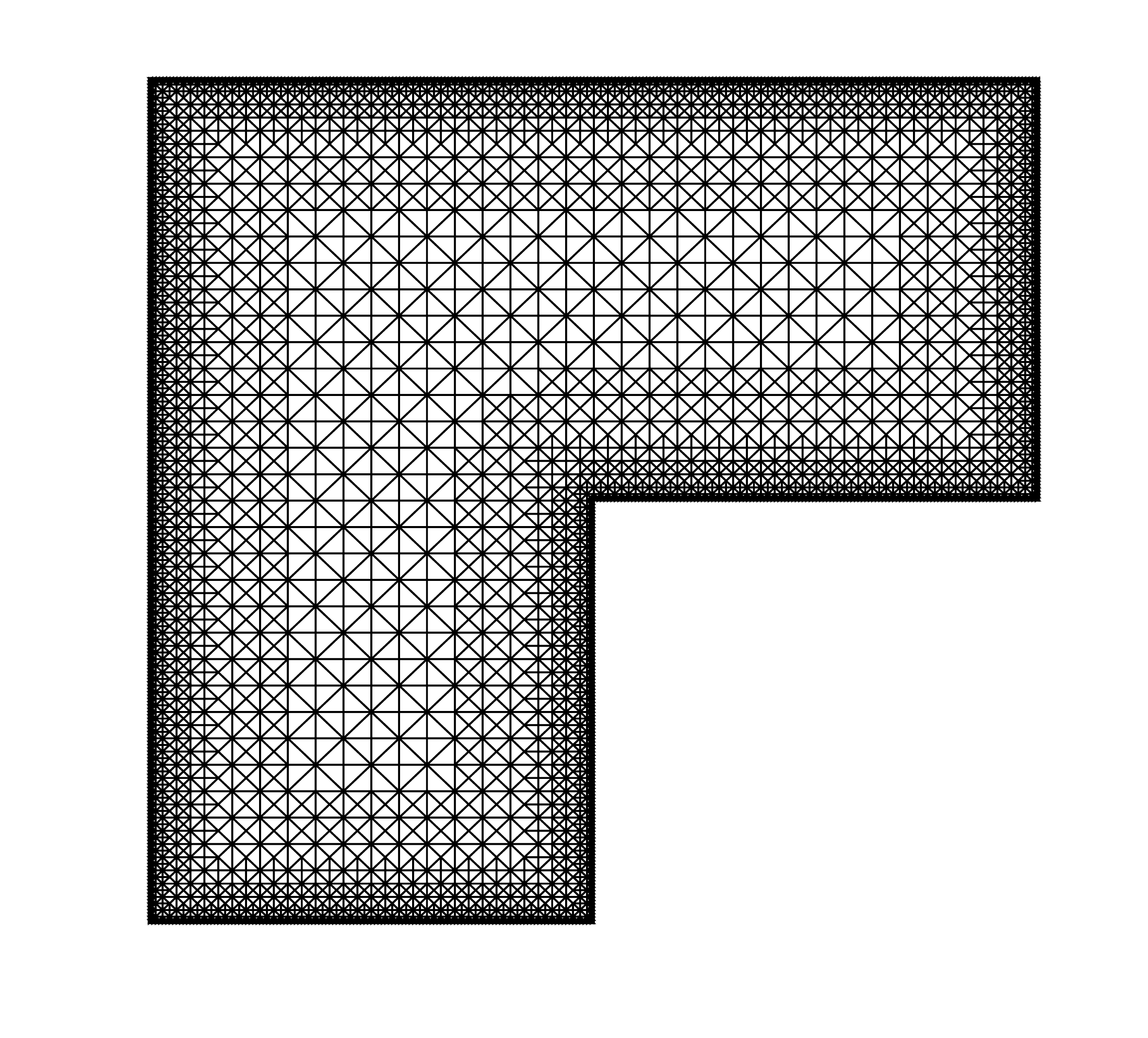}\hspace{-.3cm}
 \includegraphics[width=0.34\linewidth]{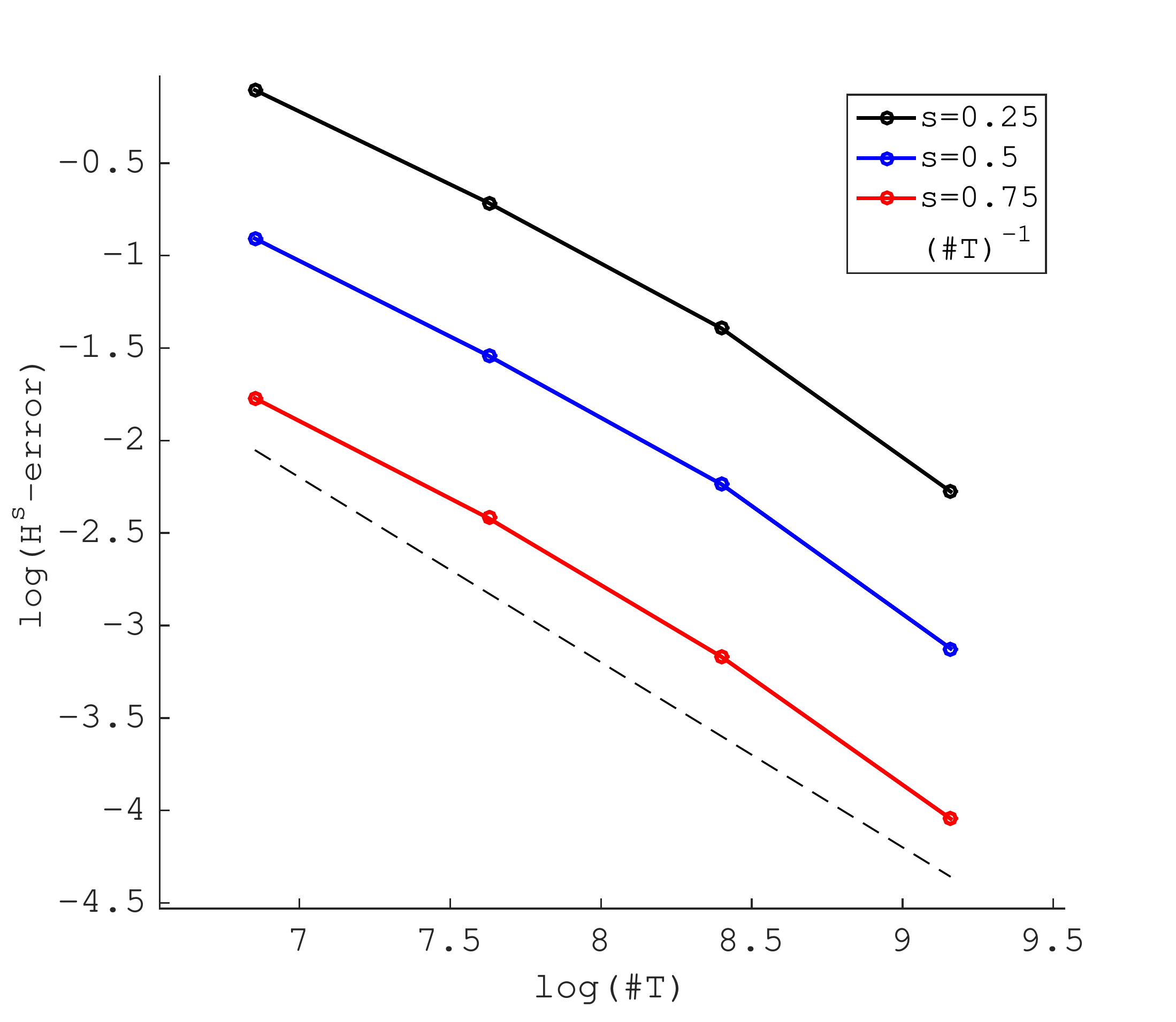}  
\end{center}
\vspace{-0.5cm}
\caption{$\GREEDY$ with the surrogate estimator from \Cref{rmk:practical-estimator} (practical estimator). Left and center: graded bisection meshes with 9504 and 15118 elements, respectively. Right: errors in the $\tHs$-norm for $s \in \{0.25, 0.5, 0.75\}$ and $f = 1$. Computational rates for the $L$-shaped domain are consistent with the expected theoretical rate $\big(\#\calT\big)^{-\frac12}|\log \#\calT|^3$ for solutions $u \in \wt{W}^{s+1-\eps}_{1+\eps}(\Omega)$ (Theorem \ref{T:greedy}).}
\label{fig:numerical}
\end{figure}

 We run the $\GREEDY$ algorithm with tolerance $\delta_k = 2^{-k} \cdot 10^{-2}$, $k = 2, \ldots 5$ in order to construct meshes $\calT_2, \ldots \calT_5$ and examine the error decay $\|u-\Pi_{\calT_k}u\|_{\wt{H}^s(\Omega)}$ in terms of $\#\calT_k$. We emphasize that the marking strategy $\mathcal{E}_T(u) > \delta_k$ is independent of $s$ and insensitive to the presence of reentrant corners. This is reflected in \Cref{fig:numerical}, whose left and middle panels depict meshes with $\#\calT = 9504$ elements and $\#\calT = 15118$ elements, respectively.

To compute the error $\|u-\Pi_{\calT_k}u\|_{\wt{H}^s(\Omega)}$, we resort to a solution on a highly refined mesh because of the lack of a closed analytical expression for the solution $u$ of \eqref{eq:Dirichlet} in this setting. The right panel in \Cref{fig:numerical} exhibits our computational orders of convergence. They show a good agreement with the expected log-linear rate from Theorem \ref{T:greedy}, even though $\Omega$ does not satisfy the sufficient condition (exterior ball condition) leading to \eqref{eq:Holder-regularity}.
\end{example}

%---------------------------------------------
\subsection{Nonconforming FEM based on Dunford-Taylor representation}
%---------------------------------------------

A nonconforming finite element method for the approximation of \eqref{eq:Dirichlet} has been proposed in \cite{BoLePa17}. The method is based on the following representation formula for the $\tHs$-inner product:
\begin{equation} \label{eq:Dunford-Taylor}
(u,v)_s = \frac{2\sin(s \pi)}{\pi} \int_0^\infty t^{-1-2s} \langle u + w(u,t) , v \rangle \; dt \quad \forall u,v \in \tHs.
\end{equation}
Above, $w = w(u,t) \in H^1(\Rd)$ is the solution to the {\em local problem} $w - t^2 \Delta w = -u \mbox{ in } \Rd$.

By using this representation, \cite{BoLePa17} proposes a three-step numerical method.
\begin{enumerate}[$\bullet$]
\item {\em Sinc quadrature:} the change of variables $t = e^{-y/2}$ leads to
\[ (u,v)_s = \frac{\sin(s\pi)}{\pi} \int_{-\infty}^\infty e^{sy}  \langle u + w(u,t(y)) , v \rangle \; dy.
\]
Therefore, given $N>0$, one can consider the sinc quadrature approximation
\[
(u,v)_s \simeq \frac{\sin(s\pi)}{N \pi} \sum_{\ell = -N}^{N} e^{sy_\ell}
\langle u + w(u,t(y_\ell)), v \rangle.
\]

\item {\em Domain truncation:} in principle, for any $t >0$ the function $w(u,t)$ can be supported in the whole space $\Rd$ even if $u$ is supported in $\Omega$. Reference \cite{BoLePa17} proposes to truncate the local problems
to a family of balls $B^M(t)$, that contain $\Omega$ and whose radius depends on either $t$ and a certain parameter $M>0$.

\item {\em Finite element approximation:} one can apply a standard finite element discretization of the local problems on $B^M(t)$. This discretization requires meshes that fit $\Omega$ and $B^M(t) \setminus \Omega$ exactly; furthermore, in order to add the contribution of the local problems, the mesh on $\Omega$ needs to remain fixed. We denote by $\mathbb{V}_h$ and $\mathbb{V}_h^M$ the discrete spaces on $\Omega$ and $B^M(t)$, respectively, and given $\psi \in L^2(\Rd)$, $t>0$ and $M>0$, we set $w_h^M = w_h^M(\psi, t)$ to be the unique function in $\mathbb{V}_h^M$ such that
\[
\int_{B^M(t)} w_h^M v_h + t^2 \nabla w_h^M \cdot \nabla v_h  \; dx = - \int_{B^M(t)} \psi v_h  \; dx \quad \forall \, v_h \in \mathbb{V}_h^M.
\]
\end{enumerate}

By combining the three steps outline above, we arrive at the discrete bilinear form
\[
a_\mathcal{T}^{N,M}(u_h, v_h) := \frac{\sin(s\pi)}{N \pi} \sum_{\ell =-
  N}^{N} e^{sy_\ell} \langle 
  u_h + w_h^M(u_h,t(y_\ell)), v 
  \rangle \quad\forall \, u_h, v_h \in \mathbb{V}_h.
\]

Strang’s Lemma implies that the $\tHs$-error between $u$ and $u_h$ is bounded by an approximation error and the sum of the consistency errors from the three steps outlined above. Thus, we have \cite[Theorem 7.7]{BoLePa17}
\[
\| u - u_h \|_{\tHs} \le C \left( e^{-c\sqrt{N}} + e^{-cM} + h^{\beta-s} |\log h| \right) \| u \|_{\widetilde{H}^\beta(\Omega)}, 
\]
where $\beta \in (s,3/2)$. The regularity estimate \Cref{T:Besov_regularity} indicates that for $f \in B^{-s+1/2}_{2,1}(\Omega)$ we have $\beta = s + 1/2 - \eps$. Taking $M = \mathcal{O}(|\log h|)$ and $N = \mathcal{O}(|\log h|^2)$ gives an error bound with order $h^{1/2} |\log h|$. This is comparable with the rate obtained in \Cref{T:conv_linear_uniform} for quasi-uniform meshes. To the best of the authors' knowledge, implementation of this nonconforming approach over graded meshes, while feasible in theory, has not been addressed yet.

Finally, we comment that formula \eqref{eq:Dunford-Taylor} has also served as the starting point for a spectral Galerkin method \cite{Sheng:20} for problems posed in $\Rd$.

%---------------------------------------------
\subsection{Other numerical approaches}
%---------------------------------------------
In recent years, there has been substantial progress in the implementation of discretization schemes for problems involving the integral fractional Laplacian \eqref{eq:def_Laps} and other nonlocal operators on bounded domains. Here, we briefly comment on some approaches that are not of finite element type. We refer to \cite{DElia20,Lischke_et_al18} for further discussion on these and related approaches.

\medskip

\noindent{\bf Finite difference methods.} Reference \cite{huang2014numerical} proposes a method that combines finite differences
with numerical quadrature, obtains a discrete convolution operator and studies the convergence of such a method. However, the finite difference algorithm is only implemented in $d=1$ dimension, and the convergence analysis requires solutions to be of class $C^4$ up to the domain boundary. More recent finite difference implementations are able to deal with higher-dimensional problems (cf. \cite{duo2018novel,duo2019accurate,Minden:20}, for example). An interesting two-scale finite difference method on finite-element type meshes has been proposed in \cite{Han:21}. The convergence analysis in that work is based on weighted H\"older regularity estimates like \eqref{eq:holder} and the use of suitable discrete barriers, thereby avoiding unrealistic solution regularity assumptions.

\medskip

\noindent{\bf Fourier methods.} The Fourier representation of the fractional Laplacian offers some opportunities for the discretization of such an operator. In particular, for periodic functions such a representation can be exploited to develop spectral approximation schemes. Interesting applications of this approach include phase-field modeling \cite{Ainsworth:17} and image processing \cite{Antil:17}.

In contrast, if one truncates the domain on which the Fourier
transformation is performed, one instead obtains the fractional Laplacian of a function that was periodically extended outside the truncation domain; such an approach, in conjunction with the use of sinc basis functions, has been exploited in \cite{Antil:21} to approximate the Dirichlet problem for the integral fractional Laplacian.

%---------------------------------------------
%---------------------------------------------
\section{Quasilinear problems}\label{sec:quasilinear}
%---------------------------------------------
%---------------------------------------------

Let $\Omega \subset \Rd$ be an open, bounded set with Lipschitz boundary, and functions $f \colon \Omega \to \R$, $g \colon \Omega^c \to \R$ be given. In this section, we discuss some problems that can be succinctly expressed as follows: find $u \colon \Rd \to \R$ that coincides with $g$ on $\Omega^c$ and such that
\begin{equation} \label{eq:quasilinear-weak}
a_u(u,v) = \langle f,v \rangle
\end{equation}
for every $v$ sufficiently smooth and vanishing on $\Omega^c$. Here, the form $a_u(\cdot, \cdot)$ is bilinear albeit it depends on the solution one is seeking,
\begin{equation} \label{E:def-a}
a_u(w,v) := \iint_{Q_{\Omega}} \widetilde{G} \left(\frac{u(x)-u(y)}{|x-y|^t}\right) \frac{(w(x)-w(y))(v(x)-v(y))}{|x-y|^{d+2\sigma}}dx dy.
\end{equation}
Naturally, to determine the problem one needs to specify the values of $\sigma \in (0,1)$, $t \in \{\sigma, 1\}$, and the nonlinearity $\widetilde{G}$ above. In the case $\widetilde{G}$ is a constant function, problem \eqref{eq:quasilinear-weak} reduces to \eqref{eq:weak_linear}. Here, we shall be concerned with two examples:
\begin{enumerate}[$\bullet$]
\item {\em Fractional mean curvature.} If we set $\sigma = s + 1/2$ for some $s \in (0,1/2)$, $t = 1$, and
\[
\widetilde{G}(\rho) = C_{d,s} \int_0^1 (1+ \rho^2 r^2)^{-(d+1+2s)/2} dr,
\]
then problem \eqref{eq:quasilinear-weak} corresponds to finding a function $u$ that coincides with $g$ on $\Omega^c$ and whose subgraph possesses certain {\em fractional mean curvature} equal to $f$ on the cylinder $\Omega \times \R$. In case $f \equiv 0$, the subgraph of $u$ is an {\em $s$-minimal set.} The normalization constant $C_{d,s} := \frac{1-2s}{|D_1(0)|}$ guarantees that, in the limit $s \to \frac12^-$, one recovers the classical mean curvature operator \cite[Lemmas 5.9 and 5.11]{BoLiNo19analysis}. Because in this work we are not concerned with such a limit, we shall omit the constant in the following.

\smallskip
\item {\em Fractional $p$-Laplacian.} Setting $\sigma = t = s$ for some $s \in (0,1)$ and $\widetilde{G}(\rho) = \frac{C_{d,s,p}}{2} |\rho|^{p-2}$ for some $p \in (1,\infty)$, we recover the Dirichlet problem for the $(p,s)$-Laplace operator,
\[
(-\Delta)^s_p v(x) := C_{d,s,p} \int_{\Rd} \frac{|v(x)-v(y)|^{p-2}(v(x)-v(y))}{|x-y|^{d+sp}} \, dy.
\]
The integral above needs to be understood in the principal value sense in case $sp \ge 1$ and we choose the normalizing constant $C_{d,s,p}$ as
\begin{equation}\label{eq:Cdsp}
C_{d,s,p} = \frac{s(1-s) p \; \Gamma(\frac{ps + d}{2}) \; 2^{2s-2}}{\pi^{\frac{d-1}{2}} \Gamma(\frac{(p-2)s + 3}{2}) \Gamma(2-s)}.
\end{equation}
We remark that this choice is somewhat arbitrary. Nevertheless, in case $p=2$ it gives rise to the integral fractional Laplacian \eqref{eq:def_Laps}; moreover,
for every smooth function $v \in C_c^{\infty}(\Rd)$ we recover the asymptotic behaviors \cite{BourBrezMiro2001another,Mazya_BBM,delTeso:21}
\[
\lim_{s \to 1^-}(-\Delta)^s_p v = -\nabla \cdot (|\nabla v|^{p-2} \nabla v), \quad
\lim_{s \to 0^+}  (-\Delta)^s_p v = |v|^{p-2} v.
\]
\end{enumerate}

This section is organized as follows. We first review some theory regarding the fractional mean curvature, including the variational problem we aim to solve, and some properties of nonlocal minimal graphs. Afterwards, we focus on theoretical aspects related to the fractional $p$-Laplacian, including the regularity of solutions to the Dirichlet problem for such an operator on Lipschitz domains. Then, we propose a common discretization technique for both operators, and study the convergence of the discrete method. The section concludes with some numerical experiments.

%---------------------------------------------
\subsection{Minimal graphs}
%---------------------------------------------

In this section, we describe the notion of fractional perimeter and some results regarding fractional minimal sets. Our focus shall be on {\em minimal graphs} on $\R^d$, that arise as minimal sets on cylinders in $\R^{d+1}$ when the exterior data is a subgraph.

%---------------------------------------------
\subsubsection{\bf Fractional perimeter and minimal sets}
%---------------------------------------------
Next, we briefly review the notions of $s$-perimeter and $s$-minimal sets.

\begin{definition}[$s$-perimeter] \label{def:s-perimeter}
Given a domain $B \subset \R^{d+1}$ and $s \in (0, 1/2)$, the $s$-perimeter of a set $A \subset \R^{d+1}$
in $B$ is defined as
\begin{equation}\label{E:NMS-Energy}
P_s(A;B) := \frac{1}{2} \iint_{Q_B} \frac{|\chi_A(x) - \chi_A(y)|}{|x-y|^{d+2s}} \, dy dx,
\end{equation}
where $Q_B = (\R^{d+1} \times \R^{d+1}) \setminus (B^c \times B^c)$ and $B^c = \R^{d+1} \setminus B$.
\end{definition}

Formally, definition \eqref{E:NMS-Energy} coincides with
\begin{equation*}
    P_s(A,B) = \frac{1}{2} \left(|\chi_A|_{W^{2s}_1(\R^{d+1})} - |\chi_A|_{W^{2s}_1(B^c)} \right),
\end{equation*}
and thus we recover the expression \eqref{eq:s-perimeter}.

The sets that minimize the $s$-fractional perimeter among those that coincide with them outside $B$ are deemed as {\em $s$-minimal sets} in $B$. The notion of $s$-minimality involves the behavior of sets in the whole space $\R^{d+1}$, in contrast with the classical (local) minimal sets, that are characterized by their behavior in $\overline{B}$.

\begin{definition}[$s$-minimal set] \label{def:minimal-set}
A set $A$ is $s$-minimal in a open set $B \subset \R^{d+1}$ if $P_s(A,B)$ is finite and
$P_s(A, B) \leq P_s(A', B)$ among all measurable sets $A' \subset \R^{d+1}$ such that
$A' \setminus B = A \setminus B$. 
\end{definition}

Given an open set $B$ and a fixed set $A_0$, the Dirichlet or Plateau problem for nonlocal minimal surfaces aims to find a $s$-minimal set $A$ such that $A \setminus B = A_0 \setminus B$. For a bounded Lipschitz domain $B$ the existence of solutions to the Plateau problem is established in \cite{CaRoSa10}.

%---------------------------------------------
\subsubsection{\bf Formulation of minimal graphs}
%---------------------------------------------
We focus on $s$-minimal sets on a cylinder $B = \Omega \times \R$, where $\Omega$ is a bounded and sufficiently smooth domain in $\R^d$, with exterior data being a subgraph,
\begin{equation*}
A_0 = \left\{ (x', x_{d+1}) \colon x_{d+1} < g(x'), \; x' \in \Rd \setminus \Omega \right\},
\end{equation*}
for some given function $g: \Rd \setminus \Omega \to \R$. We highlight some key features of this problem, and refer to \cite{BoLiNo19analysis} for a detailed discussion.

\begin{enumerate}[$\bullet$]
\item Definition \ref{def:minimal-set} is not appropriate for our problem, because every set $A$ that is a subgraph on $\R^{d+1} \setminus B$ has infinite $s$-perimeter in $B$. To remedy this issue, one needs to look for  {\em locally} $s$-minimal sets \cite{Lomb16Approx}, namely, sets that are $s$-minimal in every bounded open subset compactly supported in $B$.

\item In our setting, there exists a unique locally $s$-minimal set, which turns out to be the subgraph of a certain function $u$ (see \cite{DipiSavinVald16Graph, Lombardini-thesis}). We can therefore restrict our minimization problem to the set of subgraphs of functions that coincide with $g$ on $\Omega^c$.

\item In case $g$ is a bounded function, we can replace the infinite cylinder $B = \Omega \times \R$ by a truncated cylinder $B_M = \Omega \times (-M,M)$, where $M>0$ depends on $g$ \cite[Proposition 2.5]{Lombardini-thesis}.

\item Let $A$ be the subgraph of a certain function $v$ that coincides with $g$ on $\Omega^c$ and $B_M$ as above. The fractional $s$ perimeter of $A$ in $B_M$ can be expressed as
\[
P_s(A,B_M) = I_s[v] + C(M,d,s,\Omega,g),
\]
where $I_s$ is a suitable nonlocal energy functional, see \eqref{E:NMS-Energy-Graph} below \cite[Proposition 4.2.8]{Lombardini-thesis}, \cite[Proposition 2.3]{BoLiNo19analysis}.
\end{enumerate}

With the considerations above, we can express the Plateau problem for nonlocal minimal graphs as follows \cite{BoLiNo19analysis}:
find a function $u: \Rd \to \R$, with the constraint $u = g$ in ${\Omega}^c$, such that it minimizes the strictly convex energy
\begin{equation}\label{E:NMS-Energy-Graph}
I_s[u] := \iint_{Q_{\Omega}} F_s\left(\frac{u(x)-u(y)}{|x-y|}\right) \frac{1}{|x-y|^{d+2s-1}} \;dxdy,
\end{equation}
where $F_s$ is defined as 
\begin{equation*} \label{E:def_Fs}
F_s(\rho) := \int_0^\rho \frac{\rho-r}{\left( 1+r^2\right)^{(d+1+2s)/2}} dr.
\end{equation*}

Let $s \in (0,1/2)$ and $g \in L^\infty(\Omega^c)$ be given. 
We consider the space
\begin{equation*}
\mathbb{V}^g := \{ v \colon \Rd \to \R \; \colon \; v\big|_\Omega \in W^{2s}_1(\Omega), \ v = g \text{ in } {\Omega}^c \}, 
\end{equation*}
equipped with the norm
\[
\| v \|_{\mathbb{V}^g } := \| v \|_{L^1(\Omega)} + | v |_{\mathbb{V}^g },
\]
where
\[
| v |_{\mathbb{V}^g} := \iint_{Q_{\Omega}}  \frac{|v(x)-v(y)|}{|x-y|^{d+2s}} dxdy , 
\]
and we recall $Q_{\Omega} = (\mathbb{R}^d\times\mathbb{R}^d) \setminus (\Omega^c\times\Omega^c)$. The seminorm in $\mathbb{V}^g$ does not take into account interactions over $\Omega^c \times \Omega^c$, because these are fixed for the class of functions we consider; therefore, we do not need to assume $g$ to be a function in the class $W^{2s}_1(\Omega^c)$. In particular, $g$ may not decay at infinity.
In case $g$ is the zero function, the space $\mathbb{V}^g$ coincides with the standard zero-extension Sobolev space $\widetilde{W}^{2s}_1(\Omega)$; for consistency of notation, we denote such a space by $\mathbb{V}^0$.

Next, given $u \in \mathbb{V}^g$, we consider the bilinear form $a_u \colon \mathbb{V}^g \times \mathbb{V}^0 \to \mathbb{R}$,
\begin{equation} \label{E:def-a-NMS}
a_u(w,v) := \iint_{Q_{\Omega}} \widetilde{G}\left(\frac{u(x)-u(y)}{|x-y|}\right) \frac{(w(x)-w(y))(v(x)-v(y))}{|x-y|^{d+1+2s}}dx dy, 
\end{equation}
with $\wt{G}(\rho)=\rho^{-1} F'_s(\rho)$ given by
\begin{equation} \label{eq:Gts}
\widetilde{G}(\rho) := \int_0^1 (1+ \rho^2 r^2)^{-(d+1+2s)/2} dr.
\end{equation}
Clearly, the weight $\widetilde{G}$ satisfies $0 < \widetilde{G}(\rho) \le 1$ for all $\rho \in \R$.
We emphasize the behavior $\widetilde{G}(\rho) \to 0$ as $|\rho| \to \infty$, that implies that the weight in \eqref{E:def-a} degenerates whenever the difference quotient $\frac{|u(x)-u(y)|}{|x-y|}$ blows up. 
We also point out to the exponent $d+1+2s$ in \eqref{E:def-a-NMS} that indicates that, whenever the weight $\widetilde{G}(\rho)$ does not degenerate, the form $a_u$ behaves as the one arising in a fractional diffusion problem of order $s+1/2$ in $\mathbb{R}^d$.

The weak formulation of the fractional minimal graph problem can be obtained in a standard fashion, namely by the taking first variation of $I_s[u]$ in \eqref{E:NMS-Energy-Graph}. As described in \cite{BoLiNo19analysis}, such a  formulation reads: find $u \in \mathbb{V}^g$ satisfying
\begin{equation}\label{E:WeakForm-NMS-Graph}
a_u(u,v) = 0 \quad \forall v \in \mathbb{V}^0.
\end{equation}

%---------------------------------------------
\subsubsection{\bf Regularity and stickiness} \label{sec:regularity-stickiness}
%---------------------------------------------
An outstanding feature of the fractional Plateau problem is the emergence of {\em stickiness phenomena.} In the setting of this paper, this means that the minimizer may be discontinuous across $\pp\Omega$. As shown by Dipierro, Savin and Valdinoci \cite{DipiSavinVald19nonlocal}, stickiness is indeed the generic behavior of nonlocal minimal graphs in case $\Omega \subset \R$. 

Let us illustrate this phenomenon with a quantitative example proposed in \cite[Theorem 1.2]{DipiSavinVald17} and studied numerically in \cite{BoLiNo21}. We solve \eqref{E:WeakForm-NMS-Graph} on a fixed mesh for $\Omega = (-1,1) \subset \R$ and $g(x) = M\, \textrm{sign}(x)$, where $M > 0$. Reference \cite{DipiSavinVald17} proves that, for every $s \in (0,1/2)$, there exists $M>0$ such that the corresponding solution $u^M$ is discontinuous across $\partial \Omega$,
and that there exists an optimal constant $c_0$ such that
\begin{equation} \label{eq:bdry-behavior-1d}
\sup_{x \in \Omega} u^M(x) < c_0 M^{\frac{1+2s}{2+2s}}, \quad
\inf_{x \in \Omega} u^M(x) > -c_0 M^{\frac{1+2s}{2+2s}}.
\end{equation}
The left panel in Figure \ref{fig:stickiness_thm2} shows the computed solutions with $M = 16$ and  $s=0.1,0.25,0.4$. Because in all cases the discrete solutions $u_h$ are monotonically increasing in $\Omega$, we regard the value of $u^M_h(x_1)$ as an approximation of $\sup_{x \in \Omega} u^M(x)$, where $x_1$ is the free node closest to $1$. The right panel in Figure \ref{fig:stickiness_thm2} shows how $u^M_h(x_1)$ varies with respect to $M$ for different values of $s$. 

\begin{figure}[!htb]
	\begin{center}
		\includegraphics[width=0.45\linewidth]{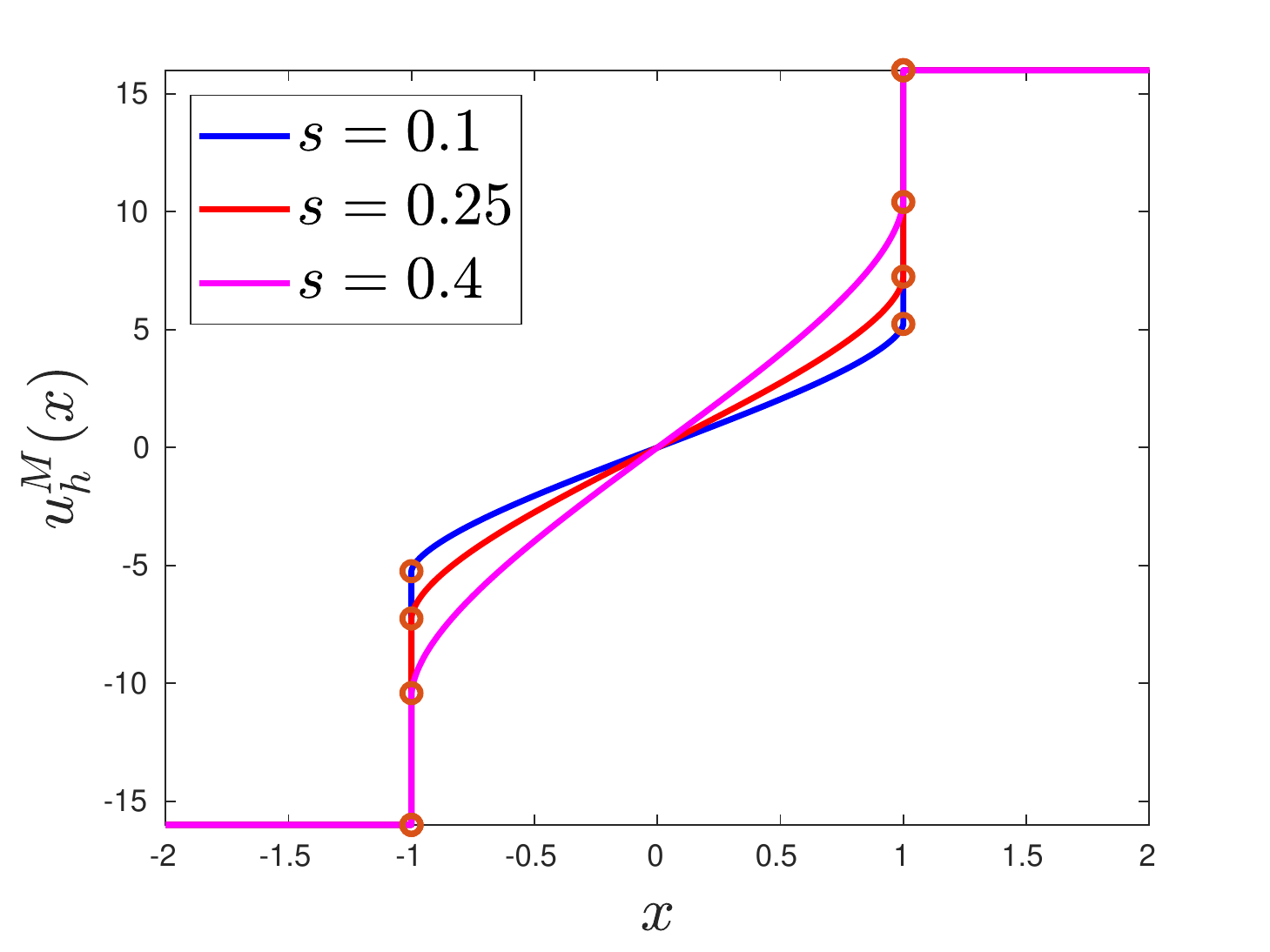} \hspace{-0.7cm}
		\includegraphics[width=0.45\linewidth]{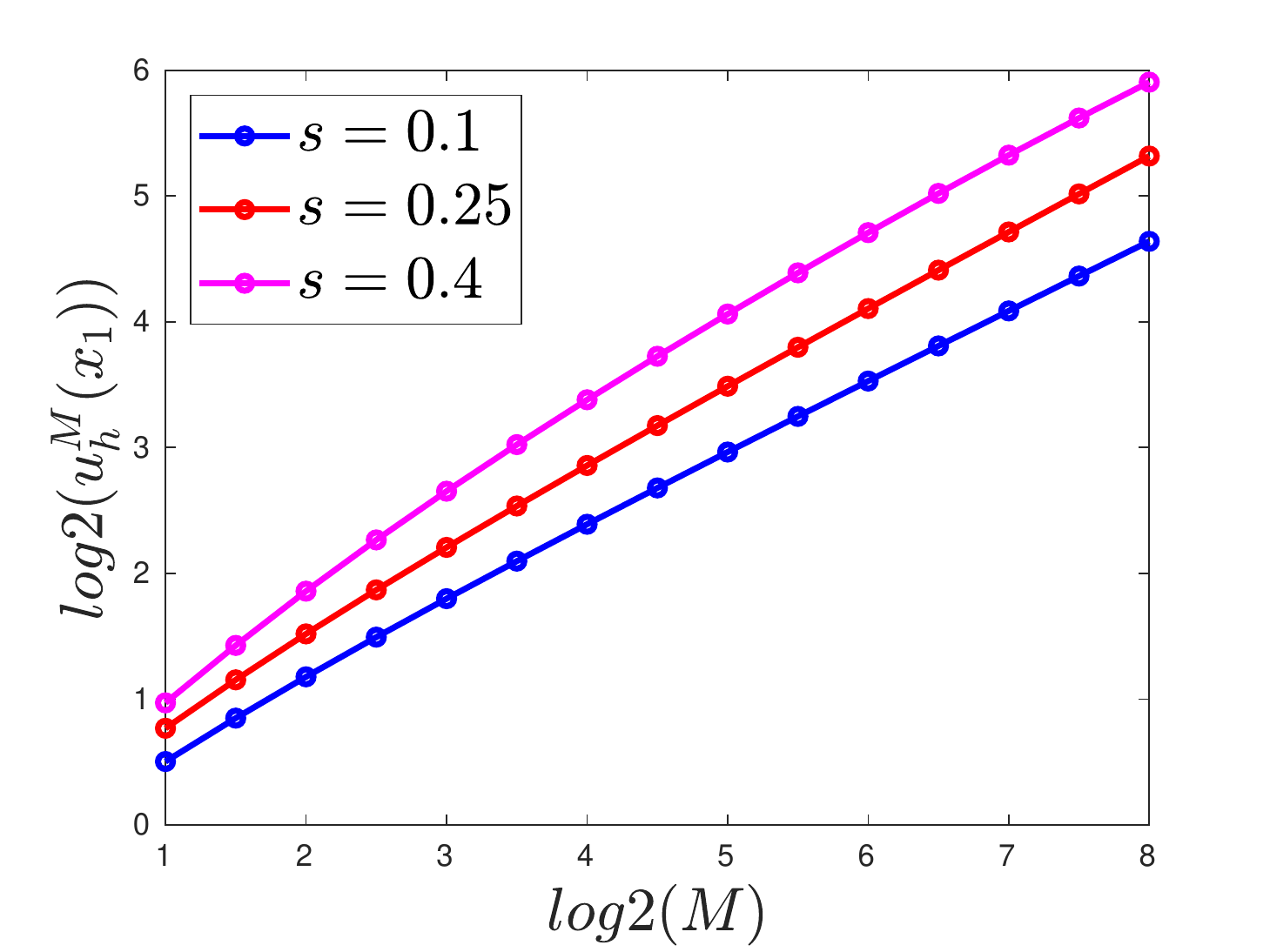} 
		\caption{
Stickiness in 1d. The left panel displays the finite element solutions $u^M_h$ for $M = 16$ and $s\in \{0,1, 0.25, 0.4\}$. The right panel shows the value of $u^M_h(x_1)$ as a function of $M$ for $s \in \{0.1, 0.25, 0.4\}$, which is expected to behave according to \eqref{eq:bdry-behavior-1d}.
}\label{fig:stickiness_thm2}
	\end{center}
\end{figure}

For $s=0.1$ and $s=0.25$ the slopes of the curves are slightly larger than the theoretical rate $M^{\frac{1+2s}{2+2s}}$ whenever $M$ is small. However, as $M$ increases, we see a good agreement with theory. Comparing results for $M=2^{7.5}$ and $M=2^8$, we observe approximate rates $0.553$ for $s=0.1$ and $0.602$ for $s=0.25$, where the expected rates are $6/11 \approx 0.545$ and $3/5 = 0.600$, respectively. However, the situation is different for $s=0.4$: the plotted curve does not correspond to a flat line, and the last two nodes plotted, with $M=2^{7.5}$ and $M=2^8$, show a relative slope of about $0.57$, which is off the expected rate $9/14 \approx 0.643$. This issue is related to poor boundary resolution, and can be corrected by refining the mesh accordingly (cf. \cite[Table 1]{BoLiNo21}).

Even though stickiness is the typical behavior of fractional minimal graphs, an interesting phenomenon arises when such graphs happen to be continuous at some point on the boundary of the domain.
In the case $\Omega \subset \R^2$, reference \cite{DipiSavinVald19boundary} proves that, at any boundary points at which stickiness does not happen, the tangent planes of the traces from the interior must coincide with those of the exterior datum. Such a hard geometric constraint is in sharp contrast with the case of classical minimal graphs. We illustrate this behavior with the computational results \cite{BoLiNo21} of an experiment first proposed in \cite{DipiSavinVald19boundary}. We consider $\Omega = (0,1) \times (-1,1)$ and the Dirichlet datum
\begin{equation} \label{eq:def-g-NMS-rigidity}
g(x, y) = \gamma \left( \chi_{(-1,-a) \times (0,1)}(x,y) - \chi_{(-1,-a) \times (-1,0)}(x,y) \right)
\end{equation}
where $a \in [0,1]$ and $\gamma > 0$ are parameters to be chosen. 
Figure \ref{fig:2d-stick1-uh} (left panel) displays a numerical solution $u_h$ corresponding to $s=0.25$ and $a = 1/8$. The right panel in Figure \ref{fig:2d-stick1-uh} exhibits slices of the computed solution at $x=2^{-3}, 2^{-6}$, and $2^{-9}$. The flattening of the curves as $x \to 0^+$ is apparent. 

\begin{figure}[!htb]
\includegraphics[width=0.4\linewidth]{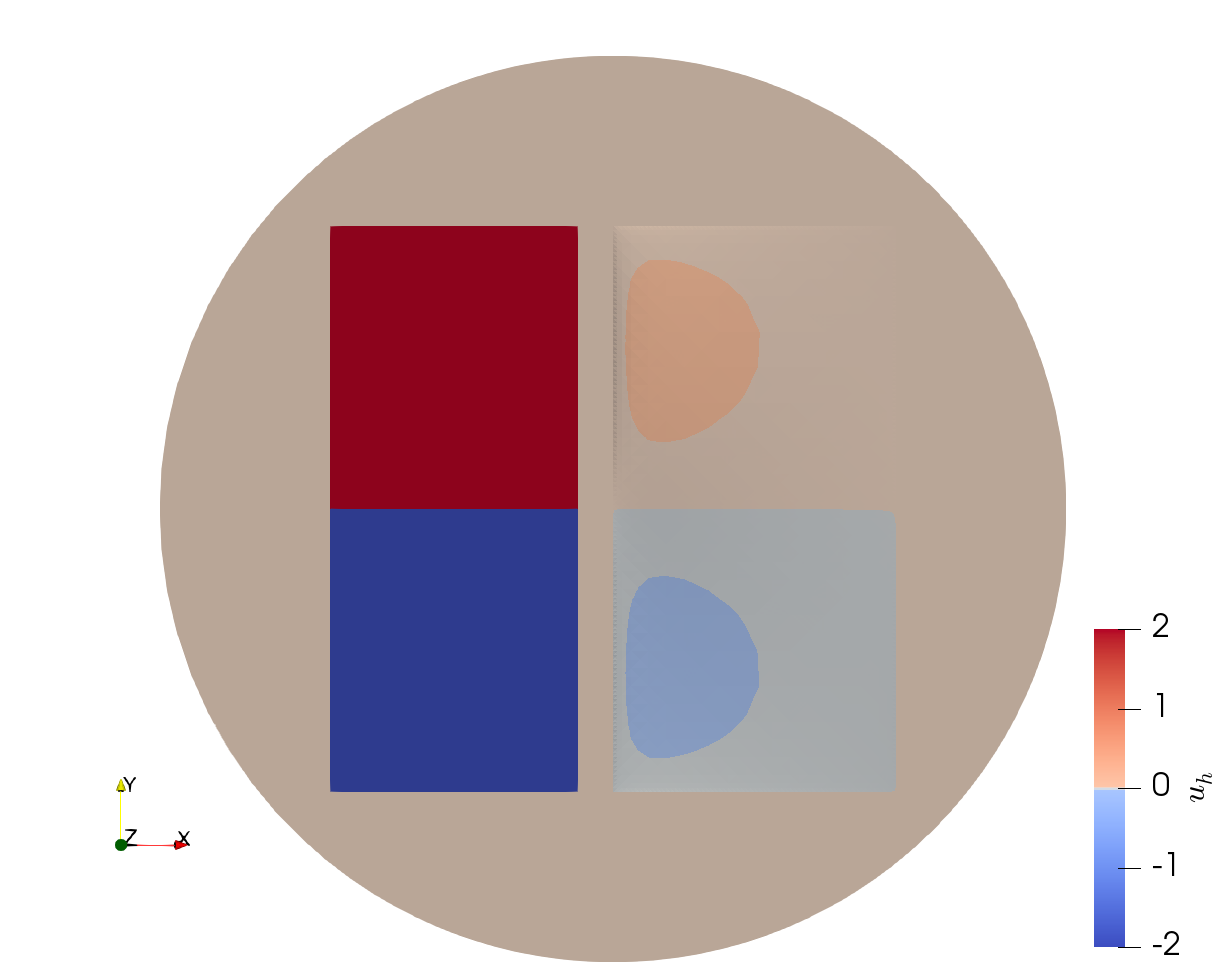}
\includegraphics[width=0.5\linewidth]{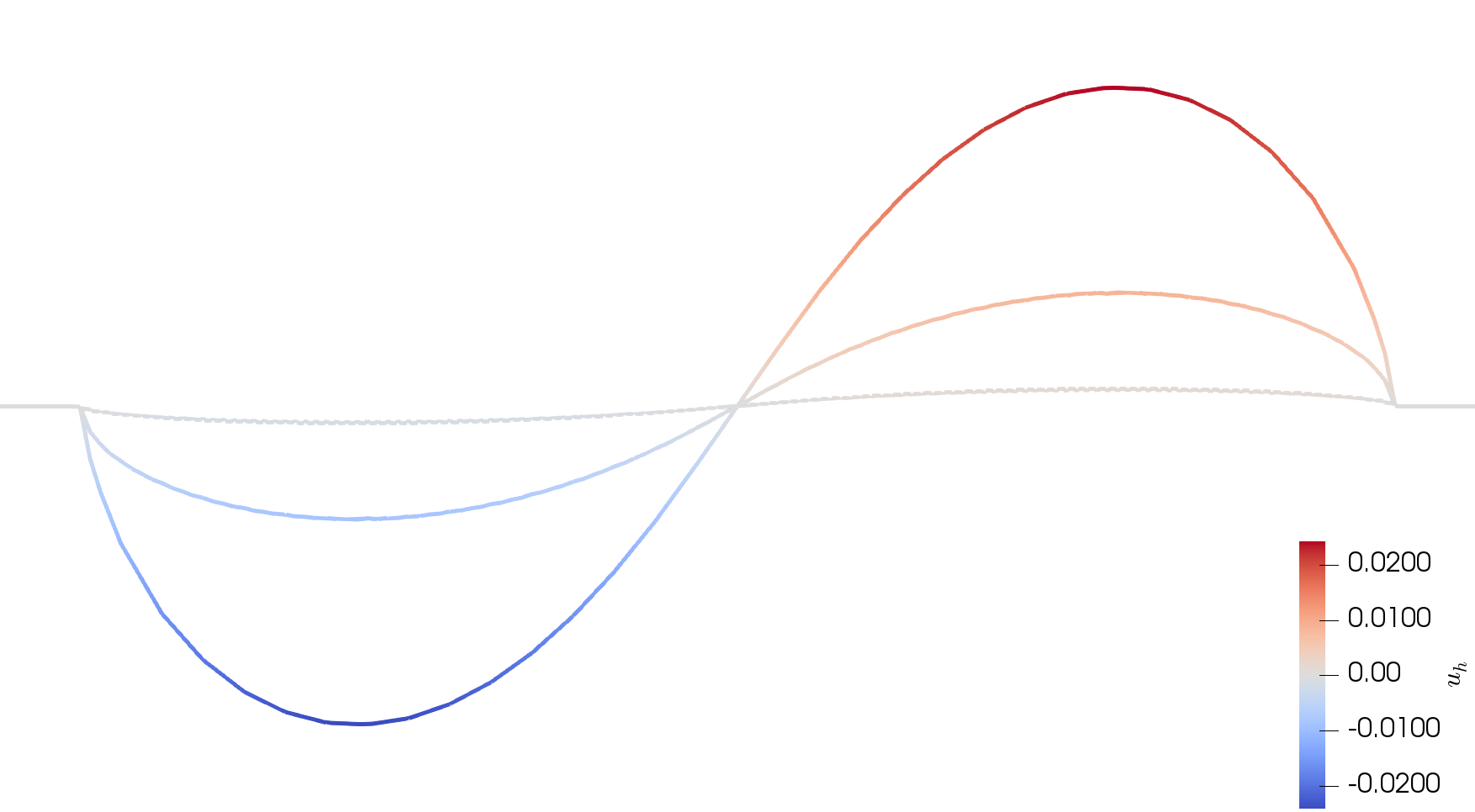}
	\caption{\small Plot of the finite element solutions to a fractional graph Plateau problem with $g$ given as in \eqref{eq:def-g-NMS-rigidity} and $a = 1/8$, $s=0.25$. Left panel: top view of the solution. Right panel: slices at $x=2^{-3}, 2^{-6}$ and $2^{-9}$. The fractional minimal graph flattens as $x \to 0^+$, in agreement with the fact that for such a minimizer being continuous at some point $x\in \pp\Omega$ implies having continuous tangential derivatives at such a point.
}
	\label{fig:2d-stick1-uh}
\end{figure}

In spite of their rich boundary behavior, minimal graphs are smooth in the interior of the domain. The following theorem is stated in \cite[Theorem 1.1]{CabreCozzi2017gradient}, where an estimate for the gradient of the minimal function is derived. Once such an estimate is obtained, the claim follows by the arguments from \cite{Barrios2014bootstrap} and \cite{Figalli2017regularity}.

\begin{theorem}[interior smoothness of graph nonlocal minimal sets] \label{thm:smoothness}
Let $A \subset \mathbb{R}^{d+1}$ be a locally $s$-minimal set in the cylinder $B = \Omega \times \mathbb{R}$, given by the subgraph of a measurable function $u$ that is bounded in an open set $\Lambda \supset \Omega$.
Then, $u \in C^\infty (\Omega)$.
\end{theorem}

%---------------------------------------------
\subsection{Fractional $(p, s)$-Laplacians}
%---------------------------------------------
In this section, we introduce the fractional $(p, s)$-Laplacian problem and discuss the regularity of solutions. Given $s \in (0,1)$, $p \in (1, \infty)$ and a bounded domain $\Omega \subset \Rd$, we introduce the Sobolev space of functions vanishing on $\Omega^c$,
\[
\widetilde{W}^s_p(\Omega) := \big\{ v \in W^s_p(\mathbb{R}^d) \colon \supp v \subset \overline{\Omega} \big\}
\]
with the norm
\begin{equation*} \label{eq:Gagliardo-seminorm}
\| v \|_{\widetilde{W}^s_p(\Omega)} := |v|_{W^s_p(\Rd)} =  \left(\frac{C_{d,s,p}}{2} \iint_{Q_{\Omega}} \frac{|v(x)-v(y)|^p}{|x-y|^{d+s p}} \, dx \, dy \right)^{1/p},
\end{equation*}
where $C_{d,s,p}$ is given by \eqref{eq:Cdsp}.

%-----------------------------------------------------------------------------
\subsubsection{\bf Problem formulation.}
%-----------------------------------------------------------------------------
For a given $f \colon \Omega \to \R$, we consider the Dirichlet problem for the fractional $(p, s)$-Laplacian
\begin{equation}\label{eq:Dirichlet-p}
\left\lbrace\begin{array}{rl}
(-\Delta)^s_p u = f &\mbox{ in }\Omega, \\
u = 0 &\mbox{ in }\Omega^c.
\end{array} \right.
\end{equation}
The solution $u$ to \eqref{eq:Dirichlet-p} minimizes the strictly convex energy
\begin{equation}\label{eq:frac-pLap-energy}
I_{s,p}[v] := \frac1p \| v \|_{\widetilde{W}^s_p(\Omega)}^p - \int_{\Omega} f v
= \frac{C_{d,s,p}}{2p} \iint_{Q_{\Omega}} \frac{|v(x) - v(y)|^p}{|x-y|^{d+sp}} dxdy - \int_{\Omega} f v
\end{equation}
among functions in $\widetilde{W}^s_p(\Omega)$. From now on, we shall assume that $f \in W^{-s}_{p'}(\Omega) := [\widetilde{W}^s_p(\Omega)]'$, so that we can guarantee the well-posedness of this minimization problem in $\widetilde{W}^s_p(\Omega)$.
Taking the first variation $\frac{\delta I_{s,p}[u]}{\delta u}(v)$ of \eqref{eq:frac-pLap-energy} in the direction $v\in \widetilde{W}^s_p(\Omega)$, and setting it to zero, such a minimization problem
can also be written as the weak formulation: find $u \in \widetilde{W}^s_p(\Omega)$ such that
\begin{equation}\label{eq:frac-pLap-weak}
a_u(u, v) = \langle f , v \rangle \quad \forall v \in \widetilde{W}^s_p(\Omega),
\end{equation}
where $\langle \cdot, \cdot \rangle$ denotes the duality pairing between $W^{-s}_{p'}(\Omega)$ and $\widetilde{W}^s_p(\Omega)$ and 
\begin{equation} \label{E:def-a-pLap}
a_u(w, v) := \frac{C_{d,s,p}}{2} \iint_{Q_{\Omega}}  \frac{|u(x)-u(y)|^{p-2} (w(x) - w(y)) (v(x)-v(y))}{|x-y|^{d+sp}} \, dx \, dy.
\end{equation}

Setting $v = u$ in \eqref{eq:frac-pLap-weak} leads to the stability estimate
\begin{equation} \label{eq:stability}
\|u\|_{\widetilde{W}^s_p(\Omega)} \lesssim \| f \|_{W^{-s}_{p'}(\Omega)}^{\frac1{p-1}}.
\end{equation}

Next, let us recall the following auxiliary identities, that follow by 
\cite[Lemmas 5.1--5.4]{Glowinski:75}: for all $a,b \in \R$, we have
\begin{equation*} \label{eq:aux-ineq}
\left| |a|^{p-2} a - |b|^{p-2} b \right| \le 
\left\lbrace \begin{array}{rl}
C |a-b|^{p-1} & \mbox{if } 1 < p \le 2, \\
C |a-b| (|a|+|b|)^{p-2} & \mbox{if } 2 \le p < \infty,
\end{array} \right.
\end{equation*}
and
\begin{equation} \label{eq:aux-ineq2}
\left( |a|^{p-2} a - |b|^{p-2} b \right) (a - b) \ge 
\left\lbrace \begin{array}{rl}
\alpha |a-b|^2 (|a|+|b|)^{p-2} & \mbox{if } 1 < p \le 2, \\
\alpha  |a-b|^p & \mbox{if } 2 \le p < \infty.
\end{array} \right.
\end{equation}
The constants $C$ and $\alpha$ above only depend on $p$.
The inequalities \eqref{eq:aux-ineq2} immediately give rise to the following monotonicity estimates for the fractional $(p,s)$-Laplacian.

\begin{lemma}[monotonicity]\label{lem:pLap-mono}
If $2 \le p < \infty$, there exists $\alpha > 0$ such that 
\begin{equation}\label{eq:monotonicity-pge2}
\langle (-\Delta)^s_p u - (-\Delta)^s_p v, u - v \rangle = a_u(u, u-v) - a_v(v, u-v) \ge \alpha \| u - v\|_{\widetilde{W}^s_p(\Omega)}^p \quad \forall u,v \in \widetilde W^s_p(\Omega).
\end{equation}
In the case $1< p < 2$, we have
\begin{equation}\label{eq:monotonicity-ple2}
\begin{aligned}
\langle (-\Delta)^s_p u - (-\Delta)^s_p v, u - v \rangle &= a_u(u, u-v) - a_v(v, u-v) \\
&\ge \alpha \|u - v|_{\widetilde{W}^s_p(\Omega)}^2 \left( \|u\|_{\widetilde{W}^s_p(\Omega)} + \|v\|_{\widetilde{W}^s_p(\Omega)} \right)^{p-2} \quad \forall u,v \in \widetilde W^s_p(\Omega).
\end{aligned}
\end{equation}	
\end{lemma}

For $p \ge 2$ the operator $(-\Delta)^s_p$ is $p$-coercive but for $1<p<2$ the operator is $2$-coercive on bounded sets of $\widetilde W^s_p(\Omega)$. Namely, if $u,v \in \widetilde W^s_p(\Omega)$ satisfy $\|u\|_{\widetilde W^s_p(\Omega)}, \|v\|_{\widetilde W^s_p(\Omega)} \le C$, then \eqref{eq:monotonicity-ple2} gives
\[
C^{2-p} \langle (-\Delta)^s_p u - (-\Delta)^s_p v, u - v \rangle \gtrsim  \|u - v\|_{\widetilde{W}^s_p(\Omega)}^2.
\]

We can now assess the continuity properties of the solution operator $f \mapsto u_f$ of \eqref{eq:frac-pLap-weak}. If we assume $f \in W^{-s}_{p'}(\Omega)$ with $\| f \|_{W^{-s}_{p'}(\Omega)} = K$, then $\|u\|_{\widetilde W^s_p(\Omega)} \le K^{\frac{1}{p-1}}$ by \eqref{eq:stability}. Therefore, the operator $f \mapsto u_f$ defined on
\[
\overline B_K := \{ f \in W^{-s}_{p'}(\Omega) \colon \| f \|_{W^{-s}_{p'}(\Omega)} \le K \} \mapsto \widetilde W^s_p(\Omega)
\]
happens to be Lipschitz continuous on $\overline B_K$ provided $1 < p < 2$, namely
\begin{equation*} \label{eq:local-Lip-solution-op}
\|u_f - u_g \|_{\widetilde{W}^s_p(\Omega)} \le K^{\frac{2-p}{p-1}} \| f - g \|_{W^{-s}_{p'}(\Omega)} \quad \forall f,g \in \overline B_K.
\end{equation*}
In contrast, for $p \ge 2$, the solution map is H\"older continuous on $W^{-s}_{p'}(\Omega)$ in view of\eqref{eq:monotonicity-pge2},
\[
\|u_f - u_g \|_{\widetilde{W}^s_p(\Omega)} \le \alpha^{-\frac{1}{p-1}} \| f - g \|_{W^{-s}_{p'}(\Omega)}^{\frac{1}{p-1}}.
\]

%------------------------------------------------------------------------------------
\subsubsection{\bf Besov regularity.} \label{sec:Besov-regularity-p}
%------------------------------------------------------------------------------------
Using the coercivity of the energy functional $I_{s,p}$, it is possible to proceed in the same way as in \Cref{sec:Besov-regularity} to obtain Besov regularity estimates provided the bounded domain $\Omega$ has a Lipschitz boundary. As a consequence of \Cref{lem:pLap-mono}, the following inequalities hold for the solution $u$ of \eqref{eq:frac-pLap-weak}:
\begin{alignat}{2}
	\|v - u\|_{\widetilde{W}^s_p(\Omega)}^p &\lesssim I_{s,p}[v] - I_{s,p}[u] \quad && \mbox{ if } p \ge 2, \label{eq:energy-coercive-p>2} \\
	\left( \|u\|_{\widetilde{W}^s_p(\Omega)} + \|v\|_{\widetilde{W}^s_p(\Omega)} \right)^{p-2} \|v - u\|_{\widetilde{W}^s_p(\Omega)}^2 & \lesssim I_{s,p}[v] - I_{s,p}[u] \quad && \mbox{ if } 1 <p < 2 \label{eq:energy-coercive-p<2}.
\end{alignat}        

In our proof of Besov regularity theory, these inequalities play the same role as \eqref{eq:stationarity} does in the linear case. We write $I_{s,p}[v] = \calF_p(v) - \calF_1(v)$, where
\begin{equation*}\label{eq:pLap}
\calF_1(v) := \int_\Omega fv, \quad
\calF_p(v) := \frac1p \| v \|_{\wt{W}^s_p(\Omega)}^p \quad\forall \, v \in \wt{W}^s_p(\Omega),
\end{equation*}
and consider the localized translation operator $T_h$ defined in \eqref{eq:translation}. 
Then, for $\sigma \in (0,1]$ and $t \in (-1,1+\sigma)$, we have the following modifications to \eqref{eq:regularity-F1} and  \eqref{eq:regularity-F2} \cite{BoLiNo22}:
\begin{align*}
\sup_{h\in D}
  \frac{|\calF_1(T_hv) - \calF_1(v)|}{|h|^\sigma}
  & \lesssim \|f\|_{B_{p',1}^{t}(D_{2\rho}(x_0) \cap\Omega)} |v|_{B^{\sigma-t}_{p,\infty}(D_{2\rho}(x_0))}
  \qquad\forall v \in \dot{B}^{\sigma-t}_{p,\infty}(D_{2\rho}(x_0));
\\
\sup_{h\in D}
\frac{|\calF_p(T_hv) - \calF_p(v)|}{|h|^\sigma}
& \lesssim \iint_{Q_{D_{2\rho}(x_0)}} \frac{|v(x)-v(y)|^p}{|x-y|^{d+sp}} dydx  \qquad \forall v \in \wt{W}^s_p(\Omega).
\end{align*}
Above, $D=\calC_\rho(\vn(x_0),\theta)$ is a cone of admissible directions (cf. \eqref{eq:admissible-cone}), so that we have $T_h v \in \wt{W}^s_p(\Omega)$ for all $v \in \wt{W}^s_p(\Omega)$ and $h \in D$.

Arguing similarly to the proof of \Cref{T:Besov_regularity}, in  \cite{BoLiNo22} we obtain the following result.
\begin{theorem}[shift property for the fractional $(p, s)$-Laplacian]\label{T:Besov_regularity_pLap}
	Let $p \in (1,\infty), r \in (0, 1]$, $\Omega$ be a bounded Lipschitz domain, and $u$ be the solution to the fractional $(p, s)$-Laplace equation \eqref{eq:frac-pLap-weak}.
	
	If $p \ge 2$ and $f \in B^{-s+\frac{r}{p'}}_{p',1}(\Omega)$, then we have $u \in \dot B^{s+\frac{r}{p}}_{p,\infty}(\Omega)$ with 
	\begin{equation*} \label{eq:max-reg-p-big}
	\| u \|_{\dot B^{s+\frac{r}{p}}_{p,\infty}(\Omega)} \le C(\Omega,d,p) \| f \|_{B^{-s+\frac{r}{p'}}_{p',1}(\Omega)}^{\frac1{p-1}}.
	\end{equation*}
	
	If $p < 2$ and $f \in B^{-s+\frac{r}{2}}_{p',1}(\Omega)$, then we have $u \in \dot B^{s + \frac{r}{2}}_{p,\infty}(\Omega)$ with 
	\begin{equation*} \label{eq:max-reg-p-small}
	\| u \|_{\dot B^{s + \frac{r}{2}}_{p,\infty}(\Omega)} \le C(\Omega,d,p) \|f\|_{W^{-s}_{p'}(\Omega)}^{\frac{2-p}{p-1}} \| f \|_{B^{-s+\frac{r}{2}}_{p',1}(\Omega)}.
	\end{equation*}
\end{theorem}

By combining \Cref{T:Besov_regularity_pLap} with the embedding $\dot{B}^\sigma_{p,\infty}(\Omega) \subset \wt{W}^{\sigma-\eps}_p(\Omega)$ 
\[
\|v\|_{\wt{W}^{\sigma-\eps}_p(\Omega)} \lesssim \eps^{-1/p} \|v\|_{\dot{B}^\sigma_{p,\infty}(\Omega)},
\]
valid for $\sigma \in (0,2)$ and $\eps > 0$ such that $\sigma-\eps$ is not an integer\footnote{This estimate follows by the same steps as \eqref{eq:Besov-Sobolev-emb}, with the caveat that $\dot B^{1}_{p,p}(\Omega) = \wt{W}^{1}_p(\Omega)$ only if $p = 2$ \cite[\S7.67]{AdamsFournier:2003}.} , we obtain the Sobolev estimates for the fractional $(p, s)$-Laplacian for  $r \in (0, 1]$
\begin{alignat}{2}
&\| u \|_{\wt{W}^{s+\frac{r}{p} - \eps}_p(\Omega)} \le \frac{C(\Omega,d,p)}{\eps^{1/p}} \| f \|_{B^{-s+\frac{r}{p'}}_{p',1}(\Omega)}^{\frac1{p-1}} && \qquad \text{ if }  p > 2, \; f \in B^{-s+\frac{r}{p'}}_{p',1}(\Omega), 
\label{eq:sob-regularity-pLap-p>2}
\\
& \| u \|_{\wt{W}^{s+\frac{r}{2} - \eps}_p(\Omega)} \le \frac{C(\Omega,d,p)}{\eps^{1/p}} \|f\|_{W^{-s}_{p'}(\Omega)}^{\frac{2-p}{p-1}} \| f \|_{B^{-s+\frac{r}{2}}_{p',1}(\Omega)} && \qquad \text{ if }  1 < p < 2, \; f \in B^{-s+\frac{r}{2}}_{p',1}(\Omega).
\label{eq:sob-regularity-pLap-p<2}
\end{alignat}

These estimates are valid up to the boundary of the domain and, similarly to the linear case, indicate that the upper bound on regularity is due to boundary behavior.
In the superquadratic case $p\ge 2$, reference \cite{MR3558212} derives the following higher-order interior regularity provided $f \in W^s_{p'}(\Omega)$:
\[ \begin{split}
\mbox{if } s \le \frac{p-1}{p+1}, & \quad \mbox{then } u \in \bigcap_{\eps >0} W^{s \frac{p+1}{p-1} - \eps}_{p, loc}(\Omega), \\
\mbox{if } s > \frac{p-1}{p+1}, & \quad \mbox{then } u \in W^1_{p,loc}(\Omega) \mbox{ and } \nabla u \in \bigcap_{\eps >0} W^{s \frac{p+1}{p} - \frac{p-1}{p} - \eps}_{p, loc}(\Omega).
\end{split}\]

%--------------------------------------------------------------------------------------
\subsubsection{\bf H\"older regularity.} \label{sec:Holder-regularity-p}
%--------------------------------------------------------------------------------------

In recent years, there has been a significant progress in the study of the H\"older regularity of solutions to \eqref{eq:Dirichlet-p}. We refer to \cite{MR3631069} for a thorough discussion and references on the topic, and here we briefly mention some results regarding boundary regularity. If the data $f$ is pointwise bounded and the domain $\Omega$ satisfies regularity assumptions similar to \Cref{sec:Holder-regularity} but stronger than \Cref{sec:Besov-regularity-p}, a remarkable result follows.

\begin{theorem}(global H\"older regularity for fractional $(p, s)$-Laplacians) \label{thm:Holder-reg-p}
If $\Omega \subset \Rd$ satisfies the exterior ball condition, then there exists $\alpha \in (0, s]$ depending only on $d, p$ and $s$ such that all solutions to the fractional $(p,s)$-Laplace problem \eqref{eq:frac-pLap-weak} satisfy
\[
\Vert u \Vert_{C^{\alpha}(\overline{\Omega})} \le C(\Omega, d, p, s) \Vert f \Vert_{L^{\infty}(\Omega)}^{1/(p-1)}.
\]
In addition, for $p \ge 2$, the above estimate holds for $\alpha = s$. 
\end{theorem}
The theorem above follows from \cite[Theorem 1.1]{Iannizzotto:14} and \cite[Theorem 1.4]{MR3861716}. In the superquadratic case, the boundary behavior can be refined \cite{MR4109087}: the function $x \mapsto u(x)/d(x,\partial\Omega)^s$ is of class $C^\alpha$ up to the boundary of the domain for some $\alpha \in (0,s]$.

\Cref{thm:Holder-reg-p} can be regarded as an extension of \eqref{eq:Holder-regularity} to the nonlinear problem with $p \neq 2$.
In fact, for $f \in L^\infty(\Omega)$, the boundary behavior of solutions is the same as in the linear problem independently of the value of $p \in (1,\infty)$:
\cite[Lemma 3.1]{Iannizzotto:14} shows that $u(x) = x_+^s$ satisfies
\[
(-\Delta)^s_p u(x) = 0
\]
for $x > 0$. To illustrate the boundary behavior
\begin{equation}\label{eq:u-boundary}
|u(x)| \lesssim \dist(x, \partial \Omega)^s,
\end{equation}
proved in \cite[Theorem 4.4]{Iannizzotto:14} for bounded domains $\Omega$ satisfying the exterior ball condition,
in \Cref{fig:frac-ps-Lap-1d-multi-p-zoom} we present some numerical experiments for $\Omega = (-0.5, 0.5) \subset \R$ and $f = 1$ and different choices of $p$.
We plot the computed solution $u_h$ near the boundary,
and observe a stable boundary behavior for different $p$, in agreement with \eqref{eq:u-boundary}.

\begin{figure}[!htb]
	\begin{center}
		\includegraphics[width=0.45\linewidth]{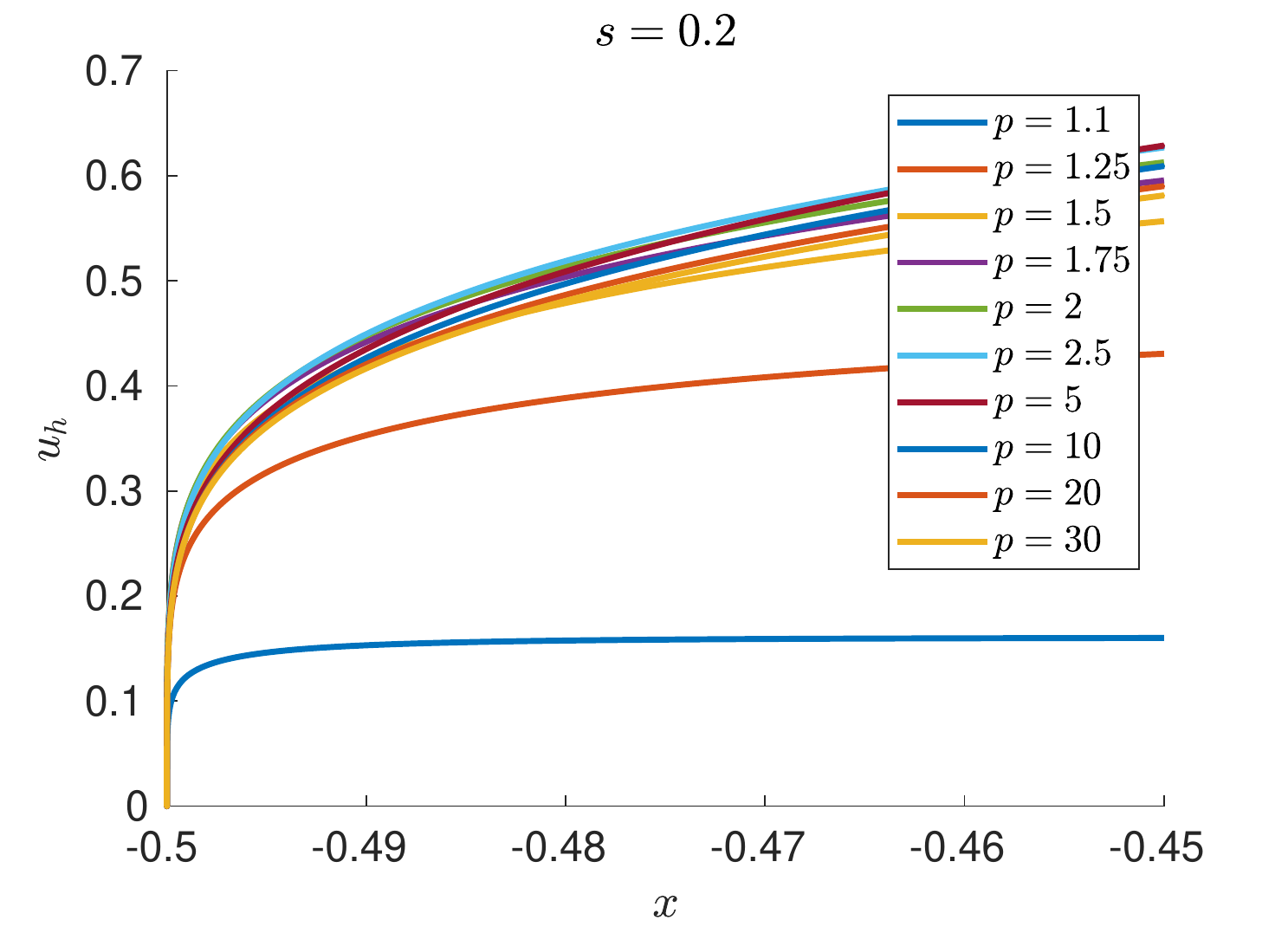}
		\includegraphics[width=0.45\linewidth]{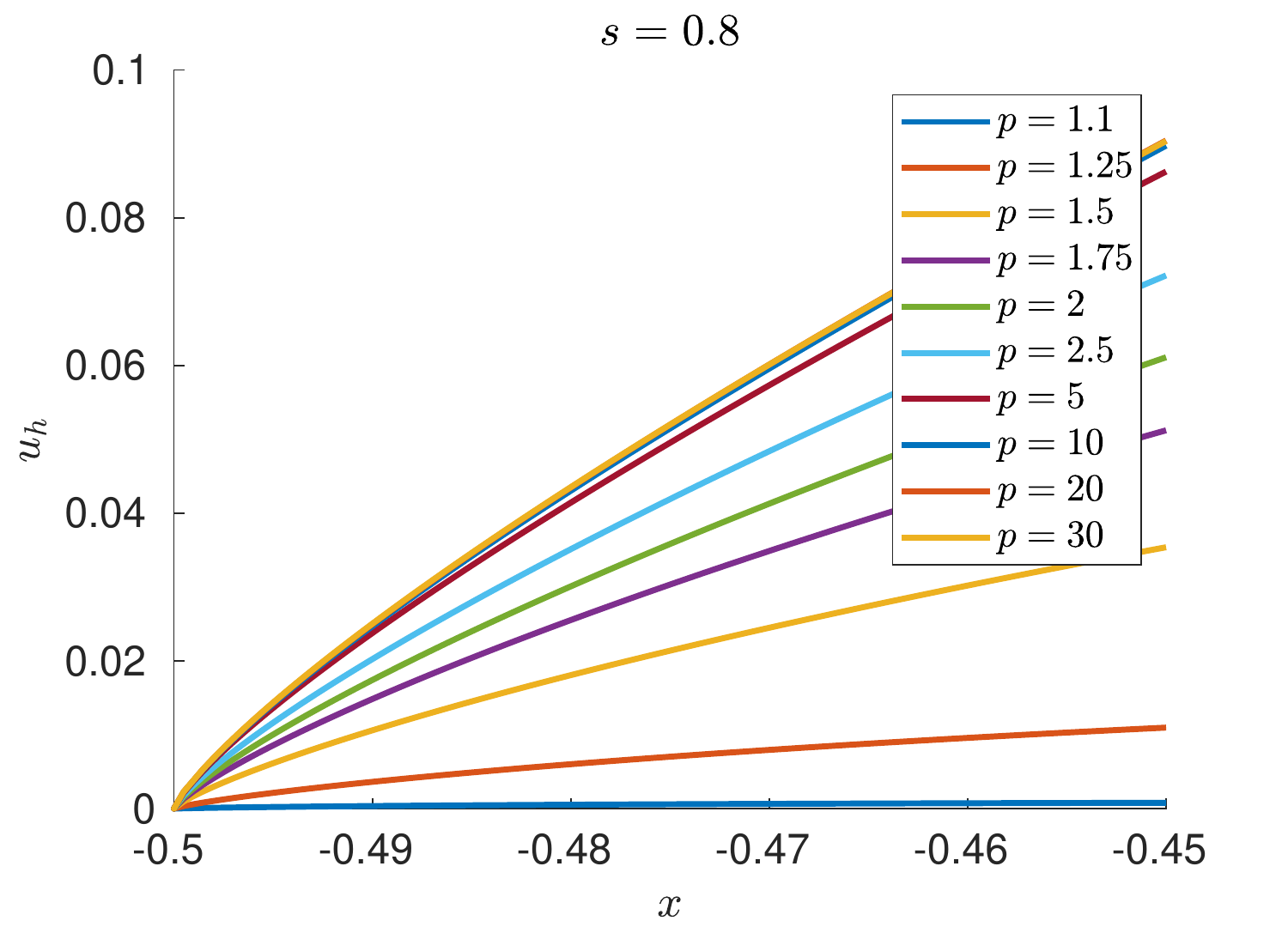} 
	\end{center}
	\caption{Boundary behavior of numerical solutions in $\Omega = (-0.5, 0.5)$ for $s = 0.2, 0.8$ and $f = 1$. }
	\label{fig:frac-ps-Lap-1d-multi-p-zoom}
\end{figure}

Unlike the linear problem with $p = 2$, even if the right hand side $f$ is smooth, in general one cannot expect the solution $u$ to be smooth in the interior of $\Omega$. 
In the superquadratic case $p\ge 2$, \cite[Theorem 1.4]{MR3861716} obtains the following interior H\"older regularity 
	\[
	u \in \bigcap_{\eps >0} C^{\theta - \eps}_{loc}(\Omega) \; \text{ with } \; \theta = \min\left\{\frac{1}{p-1}\left(sp - \frac{d}{q}\right), 1 \right\},
	\]
	provided that $f \in L^q_{\text{loc}}(\Omega)$ for $q \ge 1, q > d/(sp)$. As illustrated by \cite[Examples 1.5 and 1.6]{MR3861716}, this regularity is almost sharp when $s \le (p-1)/p + d/(pq)$.

Finally, we point out that in case $f$ is a Borel measure with finite total mass, reference \cite{MR3339179} proves the existence of the so-called SOLA (Solutions Obtained as Limits of Approximations) provided $p > 2 - s/d$. Additionally, in case $f$ belongs to the Lorentz space $L_{\text{loc}}^{\frac{d}{sp},\frac1{p-1}}(\Omega)$ with $sp < d$, \cite[Corollary 1.2]{MR3339179} proves the sharper regularity $u \in C(\Omega)$.

%---------------------------------------------
\subsection{Finite element discretization and implementation}
%---------------------------------------------
We discuss the finite element approximation of fractional minimal graphs \eqref{E:WeakForm-NMS-Graph}--\eqref{E:def-a-NMS} (or prescribed fractional mean curvature problems, that correspond to a non zero right-hand side) and of fractional $p$-Laplacians \eqref{eq:frac-pLap-weak}--\eqref{E:def-a-pLap}.

An important point to address in the implementation of nonlocal operators with infinite interaction range is the treatment of boundary conditions, which are prescribed on the unbounded set $\Omega^c$. For homogeneous conditions, one can perform tricks as discussed in Section \ref{sec:FE}; in general, we resort to computational domains containing $\Omega$ and a suitable truncation of the exterior data. Given $H>0$, we let $\Omega_H$ be a bounded open domain with $\Omega \subset\Omega_H$ and $d(x, \overline\Omega) \simeq H$ for all $x \in \partial \Omega_H$.  
We fix a cutoff function $\eta_H \in C^\infty(\Omega^c)$ satisfying
\begin{equation*} \label{eq:cutoff}
0\le \eta_H \le 1, \quad  \text{supp}(\eta_H)\subset \overline{\Omega}_{H+1} \setminus \Omega,  \quad \eta_H(x)=1 \quad \mbox{in} \quad \Omega_{H} \setminus \Omega  .
\end{equation*}
We replace $g$ by $g_H := g \eta_H$ as Dirichlet condition in problem \eqref{E:WeakForm-NMS-Graph}. Here we shall not discuss the effect of this boundary truncation, but refer to \cite{BoLiNo21} for a detailed study; naturally, one requires $H \to \infty$ as $h \to 0$ in order to guarantee convergence. From now on, we assume there exists a bounded set $\Lambda$ such that $\text{supp}(g)\subset \Lambda \subset \Omega_H$, where $\Omega_H$ is the computational domain.

To impose the condition $u = g$ in $\Omega^c$ at the discrete level, we introduce an exterior interpolation operator
\begin{equation*} \label{eq:ex-Clement}
\Pi_h^c g := \sum_{\x_i \in \mathcal{N}_h^c} (\Pi_h^{\x_i} g) (\x_i) \; \varphi_i,
\end{equation*}
where $\Pi_h^{\x_i} g$ is the $L^2$-projection of  $g\big|_{S_i \cap \Omega^c}$ onto $\mathcal{P}_1(S_i \cap \Omega^c)$. Thus, $\Pi_h^c g (\x_i)$ coincides with the standard Cl\'ement interpolation of $g$ on $\x_i$ for all nodes $\x_i$ such that $S_i \subset \Rd \setminus \overline{\Omega}$. On the other hand, for nodes $\x_i \in\partial\Omega$, $\Pi_h^c$ only averages over the elements in $S_i$ that lie in $\Omega^c$.

In the same fashion as in Section \ref{sec:FE}, we use discrete spaces consisting in piecewise linear, continuous functions over $\calT$,
\[
\mathbb{V}_h :=  \{ v \in C(\Omega_H) \colon v|_T \in \mathcal{P}_1 \; \forall T \in \calT \}.
\]
It is clear that these spaces depend on the computational domain size parameter $H>0$; for easiness of notation and because we assume $H$ is fixed (and sufficiently large so that $\Lambda \subset \Omega_H$), we shall omit such a dependence.
To account for the exterior data, we define the discrete counterpart of $\mathbb{V}^g$,
\[
\mathbb{V}_h^g := \{ v \in \mathbb{V}_h \colon \ v|_{\Omega_H \setminus \Omega} = \Pi_h^c g\}.
\]
Additionally, we denote by $\mathbb{V}_h^0$ the corresponding space in case $g \equiv 0$.
The discrete counterpart to \eqref{eq:quasilinear-weak} reads: find $u_h \in \mathbb{V}^g_h$ such that
\begin{equation}\label{E:WeakForm-discrete}
a_{u_h}(u_h, v_h) =  \langle f,v_h \rangle \quad \mbox{for all } v_h \in \mathbb{V}^0_h. 
\end{equation}

To solve this nonlinear discrete problem, we use the damped Newton scheme described in \Cref{alg:Damped-Newton}.
\begin{algorithm}
	\caption{Damped Newton Algorithm}
	\label{alg:Damped-Newton}
	\begin{algorithmic} %[1] 
		\State Select an arbitrary initial $u_h^0 \in \mathbb{V}^g_h$ and let $k = 0$. Choose 
		a small number $\text{Tol} > 0$.
		\While{ $ \Vert \{ a_{u_h^k}(u_h^k, \varphi_i) \}_{i=1}^m \Vert_{l^2} > \text{Tol}$ }
		\State Find $w_h^k \in \mathbb{V}^0_h$ such that
		\begin{equation}\label{E:damped-Newton-wh}
		\frac{\delta a_{u_h^k}(u_h^k, v_h)}{\delta u_h^k}(w_h^k)
		= -a_{u_h^k}(u_h^k, v_h), \qquad \forall v_h \in \mathbb{V}^0_h.
		\end{equation}
		\State Determine the minimum $n \in \mathbb{N}$ such that
		$u_h^{k,n} := u_h^k + 2^{-n} w_h^k$ satisfies
		\begin{equation*}
		\Vert \{ a_{u_h^k}(u_h^{k,n}, \varphi_i) \}_{i=1}^m \Vert_{l^2} \leq (1 - 2^{-n-1})
		\Vert \{ a_{u_h^k}(u_h^k, \varphi_i) \}_{i=1}^m \Vert_{l^2}
		\end{equation*}
		\State Let $u_h^{k+1} = u_h^{k,n} $ and $k = k+1$.
		\EndWhile
	\end{algorithmic}
\end{algorithm}
We exploit that discrete functions $v_h\in\mathbb{V}_h$ are Lipschitz and rewrite 
\[
a_{u_h}(u_h, v_h) = \iint_{Q_{\Omega}} G\Big(\frac{u_h(x)-u_h(y)}{|x-y|} \Big) \frac{v_h(x)-v_h(y)}{|x-y|^{d+2\sigma-1}}dxdy,
\]
where $G(r) := r \wt{G}(r)$ and $\wt{G}(r)$ is given in \eqref{E:def-a} with $t = 1$. We now compute the first variation
$\frac{\delta a_{u_h}(u_h, v_h)}{\delta u_h}(w_h)$ of $a_{u_h}(u_h, v_h)$ with respect to $u_h$
in the direction $w_h\in \mathbb{V}^0_h$,
or equivalently the second variation of $I_{s,p}[u_h]$ in the directions $v_h,w_h\in \mathbb{V}^0_h$, to get
\begin{equation*}\label{eq:2nd-variation}
 \frac{\delta a_{u_h}(u_h,v_h)}{\delta u_h}(w_h) = \iint_{Q_{\Omega}} \psi \left( \frac{u_h(x)-u_h(y)}{|x-y|} \right) \frac{(w_h(x)-w_h(y))(v_h(x)-v_h(y))}{|x-y|^{d+2\sigma}}dxdy,
\end{equation*}
where $\psi(r):=G'(r)$. For the nonlocal minimal graph problem, we have
\begin{equation*}\label{eq:2nd-variation-nmg}
  \psi(r) = (1+r^2)^{-(d+1+2s)/2}, \quad \sigma = s + \frac12.
\end{equation*}
In contrast, for the fractional $(p, s)$-Laplace equation, we let
\begin{equation*}\label{eq:2nd-variation-pLap}
  \psi(r) = \frac{(p-1)C_{d,s,p}}{2} \, |r|^{p-2}, \quad \sigma = 1 + \frac{(s-1)p}{2},
\end{equation*}
but realize that $\psi(0)$ is not well-defined whenever $p \in (1, 2)$. To overcome this issue in such a case, we introduce a small parameter $\eps >0$ and regularize the energy in \eqref{eq:frac-pLap-energy}, namely
\begin{equation*}\label{eq:frac-pLap-energy-reg}
\wt{I}_{s,p}[v] := \frac{C_{d,s,p}}{2p} \iint_{Q_{\Omega}} 
\left( \frac{|v(x) - v(y)|^2}{|x-y|^2} + \eps^2 \right)^{p/2}
\frac{dxdy}{|x-y|^{d+(s-1)p}}  - \int_{\Omega} f v.
\end{equation*}
Moreover, the discrete minimizer $u_h\in\mathbb{V}_h^g$ of $\wt{I}_{s,p}$ satisfies
\begin{equation}\label{eq:WeakForm-discrete-reg}
\wt{a}_{u_h}(u_h, v_h) =  \langle f,v_h \rangle \quad \mbox{for all } v_h \in \mathbb{V}^0_h,
\end{equation}
with
\[
\wt{a}_{u_h}(u_h, v_h) =  \iint_{Q_{\Omega}} 
G_\eps \Big( \frac{u_h(x) - u_h(y)}{|x-y|} \Big)
\frac{v_h(x) - v_h(y))}{|x-y|^{d+1 + (s-1)p}} dx dy,
\quad
G_\eps(r) = \frac{C_{d,s,p}}{2} \big( r^2 + \eps^2  \big)^{\frac{p}{2}-1} r.
\]
We then apply \Cref{alg:Damped-Newton} to find the discrete solution to the regularized problem \eqref{eq:WeakForm-discrete-reg}. Since
\[
\frac{\delta \wt{a}_{u_h}(u_h,v_h)}{\delta u_h}(w_h) = \iint_{Q_{\Omega}} \psi \left( \frac{u_h(x)-u_h(y)}{|x-y|} \right) \frac{(w_h(x)-w_h(y))(v_h(x)-v_h(y))}{|x-y|^{d+2+(s-1)p}}dxdy,
\]
with $\psi (r) = G'_\eps(r)$ given by
\[
\psi_\eps(r) = \frac{C_{d,s,p}}{2} \big((p-1)r^2 + \eps^2\big)\big(r^2 + \eps^2\big)^{\frac{p}{2} - 2},
\quad p \in (1,2),
\]
we deduce that $\frac{\delta \wt{a}_{u_h}(u_h,v_h)}{\delta u_h}(w_h)$ is well-defined for all $u_h \in \mathbb{V}_h^g$, $v_h,w_h \in \mathbb{V}_h^0$.
  
At each step of \Cref{alg:Damped-Newton}, problem \eqref{E:damped-Newton-wh} boils down to solving a linear system $\vK^k \vW^k = \vF^k$. The matrix $\vK^k = (\vK^k_{ij})$, given by
\[
\vK^k_{ij} = \iint_{Q_{\Omega}} \psi \left(\frac{u_h^k(x)-u_h^k(y)}{|x-y|}\right) \frac{(\varphi_i(x)-\varphi_i(y))(\varphi_j(x)-\varphi_j(y))}{|x-y|^{d+2\sigma}}dxdy ,
\]
is the stiffness matrix for a weighted linear problem of order $\sigma < 1$: in the nonlocal minimal graph problem we have $\sigma = s + 1/2$, while $\sigma = 1 - (1-s) p/2$ for the fractional $(p, s)$-Laplacian.
We compute this matrix and the right hand side vector $\vF^k_i = -a_{u_h^k}(u_h^k, \varphi_i)$ in a similar way as in the linear problem discussed in \Cref{sec:FE}.

%---------------------------------------------
\subsection{Convergence}
%---------------------------------------------
We now discuss the convergence of the finite element solutions to \eqref{E:WeakForm-discrete} towards the solution of the continuous problem \eqref{eq:quasilinear-weak} as the mesh size $h$ tends to zero. We shall follow different strategies for the fractional mean curvature problem and the Dirichlet problem for the fractional $p$-Laplacian.

%---------------------------------------------
\subsubsection{\bf Minimal graphs}
%---------------------------------------------
To prove the convergence of the finite element scheme, the approach in \cite{BoLiNo19analysis} consists of proving that the discrete energy is consistent and using a compactness argument.

\begin{theorem}[convergence for the nonlocal minimal graph problem]  \label{thm:consistency}
	Let $s \in (0,1/2)$, $\Omega$ be a bounded Lipschitz domain, and $g$ be uniformly bounded and satisfy $\text{supp}(g) \subset \Lambda$ for some bounded set $\Lambda$. Let $u$ and $u_h$ be the solutions to \eqref{E:WeakForm-NMS-Graph} and \eqref{E:WeakForm-discrete}, respectively. Then, it holds that
	\[
	\lim_{h \to 0} I_s[u_h] = I_s[u] \quad
	\mbox{ and } \quad
	\lim_{h \to 0} \| u - u_h \|_{W^{2r}_1(\Omega)} = 0 \quad \forall r \in [0,s).
	\]
\end{theorem}

The theorem above has the important feature of guaranteeing convergence without any regularity assumption on the solution. However, it does not offer any convergence rates.  We now show estimates for a geometric notion of error that mimics the one analyzed in \cite{FierroVeeser03} for the classical Plateau problem (see also \cite{BaMoNo04, DeDzEl05}). In the local setting, such a notion of error is given by
\begin{equation*}\label{eq:def-e}
\begin{aligned}
e^2(u,u_h) & :=\int_{\Omega} \ \Big| \widehat{\nu}(\nabla u) - \widehat{\nu}(\nabla u_h) \Big|^2 \;\frac{Q(\nabla u) + Q(\nabla u_h)}{2} \ dx , \\
& = \int_{\Omega} \ \Big( \widehat{\nu}(\nabla u) - \widehat{\nu}(\nabla u_h) \Big) \cdot \ \Big( \nabla (u-u_h), 0 \Big) dx ,
\end{aligned}
\end{equation*}
where $Q(\pmb{a}) = \sqrt{1+|\pmb{a}|^2}$, $\widehat{\nu}(\pmb{a}) = \frac{(\pmb{a},-1)}{Q(\pmb{a})}$.
Because $\widehat{\nu}(\nabla u)$ is the normal unit vector to the graph of $u$, the quantity $e(u,u_h)$ is a weighted $L^2$-discrepancy between the normal vectors. 
For the nonlocal minimal graph problem, we introduced in \cite{BoLiNo19analysis} the fractional counterpart $e_s$ of $e$
\begin{equation*}\label{eq:def-es} 
\begin{aligned}
	e_s(u,u_h) := \left( \widetilde C_{d,s} \iint_{Q_{\Omega}} \Big( G_s\left(d_u(x,y)\right) - G_s\left(d_{u_h}(x,y)\right) \Big) \frac{d_{u-u_h}(x,y)}{|x-y|^{d-1+2s}} dxdy \right)^{1/2} ,
	\end{aligned}
\end{equation*}
where $G_s(\rho) = F_s'(\rho)$, the constant $\widetilde C_{d,s} = \frac{1 - 2s}{\alpha_{d}}$, $\alpha_{d}$ is the volume of the $d$-dimensional unit ball and $d_v$ is the difference quotient of the function $v$, 
\begin{equation*}\label{eq:def-d_u}
d_v(x,y) := \frac{v(x)-v(y)}{|x-y|}.
\end{equation*}
In \cite{BoLiNo19analysis}, we showed this novel quantity $e_s(u,u_h)$ to be connected with a notion of nonlocal normal vector, and established its asymptotic behavior as $s\to 1/2^-$.

\begin{theorem}[asymptotics of $e_s$] \label{Thm:asymptotics-es}
	For all $u,v \in H^1_0(\Lambda)$, we have
	\[
	\lim_{s \to {\frac{1}{2}}^-} e_s(u,v) = e(u,v).
	\]
\end{theorem}

A simple Galerkin orthogonality-type argument allows us to derive an error estimate for $e_s(u,u_h)$ without additional regularity (cf. \cite[Theorem 5.1]{BoLiNo19analysis}).

\begin{theorem}[geometric error]\label{thm:geometric-error}
	Under the same hypothesis as in \Cref{thm:consistency}, it holds that
	\begin{equation} \label{eq:geometric-error} \begin{aligned}
	e_s(u,u_h) &\le C (d,s) \, \inf_{v_h \in \mathbb{V}_h^g}  \left( \iint_{Q_{\Omega}} \frac{|(u-v_h)(x)-(u-v_h)(y)|}{|x-y|^{d+2s}} dxdy \right) ^{1/2}.
	\end{aligned} \end{equation}
\end{theorem}

Therefore, to obtain convergence rates with respect to $e_s(u,u_h)$, it suffices to prove interpolation estimates for the nonlocal minimizer in a sort of $W^{2s}_1$-norm. Although minimal graphs are expected to be discontinuous across the boundary, we still expect that $u \in BV(\Lambda)$ in general. Under this further regularity assumption, the error estimate \eqref{eq:geometric-error}  leads to
\[
e_s(u,u_h) \le C(d,s) \, h^{1/2-s} |u|^{1/2}_{BV(\Lambda)}.
\]

%---------------------------------------------
\subsubsection{\bf Fractional $(p, s)$-Laplacian}
%---------------------------------------------

To prove the convergence of the finite element solution $u_h$ in \eqref{E:WeakForm-discrete} to the continuous solution $u$ in \eqref{eq:frac-pLap-weak}, one needs to control $\Vert u - u_h \Vert_{\wt{W}^s_p(\Omega)}$ in terms of the best approximation error $\inf_{v_h \in \mathbb{V}_h} \| u-v_h \|_{\wt{W}^s_p(\Omega)}$. Following the idea of \cite{Chow89} for the classical $p$-Laplacian, in \cite{BoLiNo22} we obtain such a result for its fractional counterpart.

\begin{theorem}[error bounds for the fractional $(p, s)$-Laplacian]
	\label{thm:FE-error-Chow}
	Let $u$ and $u_h$ be the respective solutions of the continuous and discrete Dirichlet problems for the fractional $(p, s)$-Laplacian \eqref{eq:frac-pLap-weak} and \eqref{E:WeakForm-discrete}. If $p \in (1,2]$, then there holds
	\begin{equation*} \label{eq:error_smallp}
	\|u - u_h \|_{\wt{W}^s_p(\Omega)} \le C \inf_{v_h \in \mathbb{V}_h} \| u-v_h \|_{\wt{W}^s_p(\Omega)}^{\frac{p}2}.
	\end{equation*}
	On the other hand, if $p \in [2,\infty)$, then the following alternative error bound is valid
	\begin{equation*} \label{eq:error_bigp}
	\|u - u_h \|_{\wt{W}^s_p(\Omega)} \le C \inf_{v_h \in \mathbb{V}_h} \| u-v_h \|_{\wt{W}^s_p(\Omega)}^{\frac2p}.
	\end{equation*}
\end{theorem}

Combining \Cref{thm:FE-error-Chow} with the local interpolation estimates from \Cref{L:interpolation-error} and the Sobolev regularity estimates \eqref{eq:sob-regularity-pLap-p>2}--\eqref{eq:sob-regularity-pLap-p<2}, the following convergence rates are derived in \cite{BoLiNo22}.

\begin{theorem}[convergence rates] \label{thm:FE-error-pLap-rate}
  Let $p > 1$, $\Omega$ be a bounded Lipschitz domain, and let $u$ and $u_h$ be the solutions of the continuous and discrete fractional $(p, s)$-Laplacian problems \eqref{eq:frac-pLap-weak} and \eqref{E:WeakForm-discrete}, respectively. Let $q = \max \{ 1 / p', 1 /2 \}$ and assume that $f \in B^{-s+rq}_{p',1}(\Omega)$ for some $r \in (0,1]$. If $p \in (1,2]$, then there holds
\begin{equation*} \label{eq:error_smallp_rate}
\|u - u_h \|_{\wt{W}^s_p(\Omega)} \le C h^{\frac{pr}4 } | \log h |^{\frac12} \|f\|_{W^{-s}_{p'}(\Omega)}^{\frac{2-p}{p-1}} \| f \|_{B^{-s+\frac{r}{2}}_{p',1}(\Omega)}.
\end{equation*}
On the other hand, if $p \in [2,\infty)$, then the following alternative error bound is valid
\begin{equation*} \label{eq:error_bigp_rate}
\|u - u_h \|_{\wt{W}^s_p(\Omega)} \le  C h^{\frac{2r}{p^2}} | \log h |^{\frac{2}{p^2}} \| f \|_{B^{-s+\frac{r}{p'}}_{p',1}(\Omega)}^{\frac1{p-1}}.
\end{equation*}
\end{theorem}

%---------------------------------------------
\subsection{Numerical experiments}
%---------------------------------------------
We conclude by presenting some numerical experiments for the prescribed fractional mean curvature problem and the Dirichlet problem for the fractional $(p, s)$-Laplacian. Regarding nonlocal minimal graphs, we recall the two examples we provided in \Cref{sec:regularity-stickiness}, concretely in Figures \ref{fig:stickiness_thm2} and \ref{fig:2d-stick1-uh}. We refer to \cite{BoLiNo19analysis, BoLiNo21} for several additional experiments that illustrate qualitative features of minimal graphs and explore computational aspects of the problem, such as conditioning and the effect of data truncation.

\begin{figure}[h!]
	\begin{center}
		\begin{tabular}{c c} 
			\includegraphics[width=0.42\linewidth]{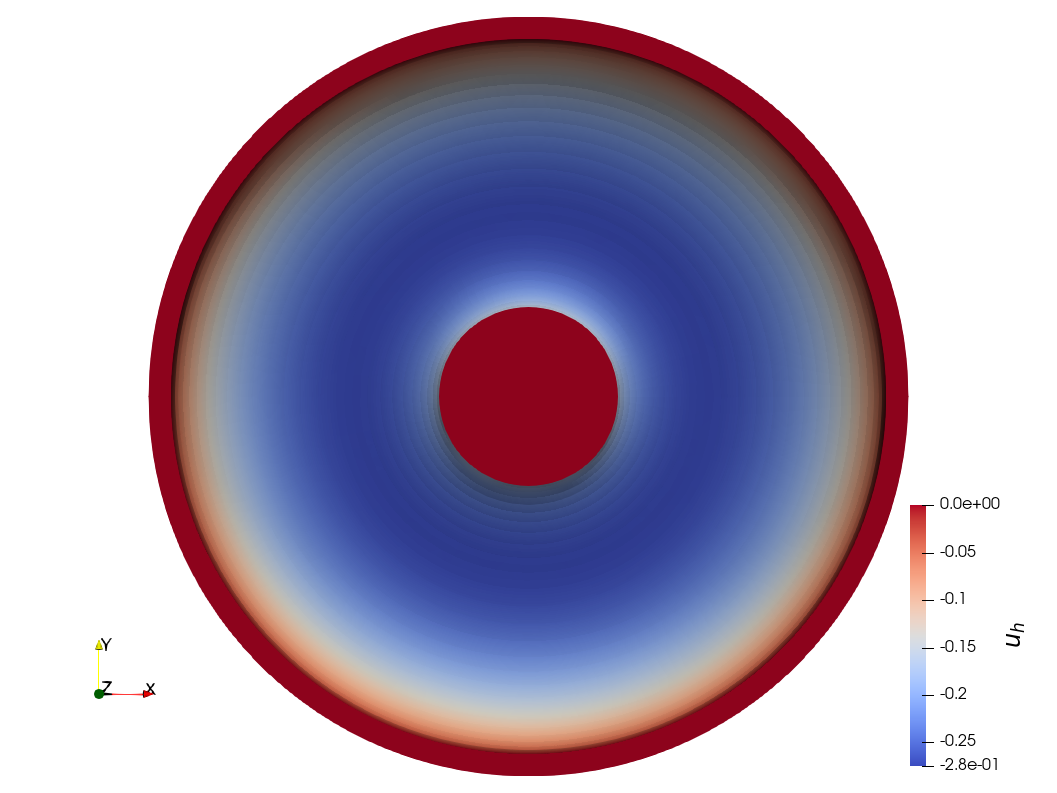}& 
			\includegraphics[width=0.42\linewidth]{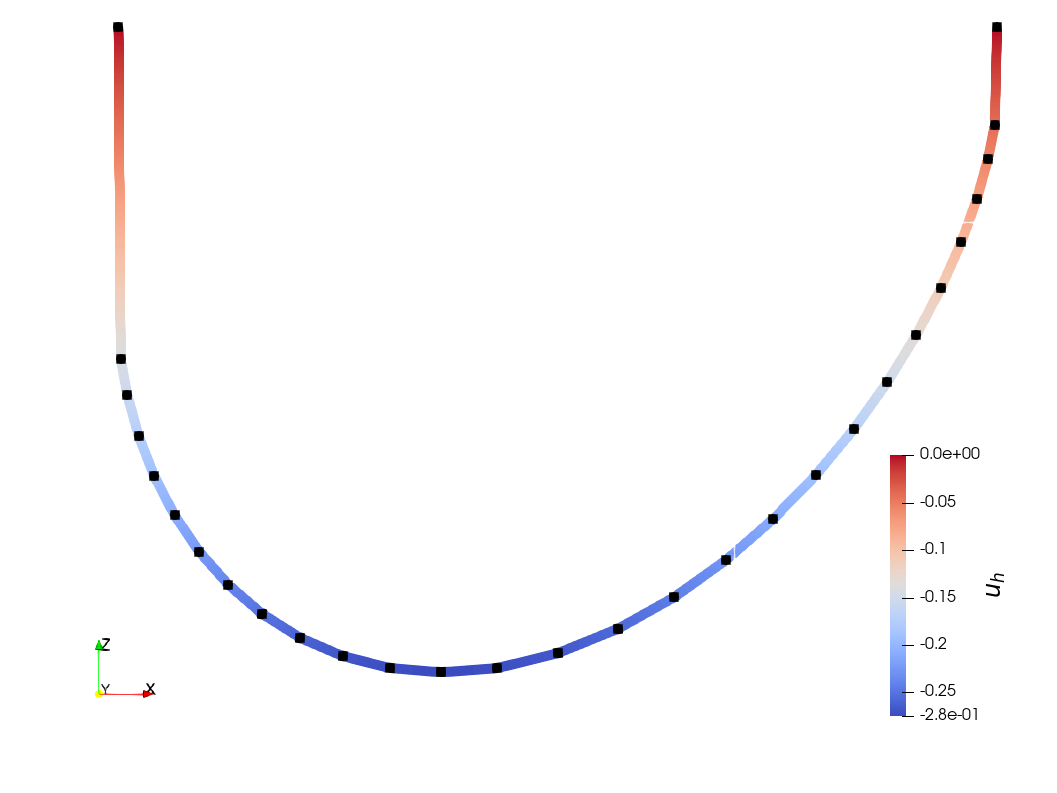}\\
			\includegraphics[width=0.42\linewidth]{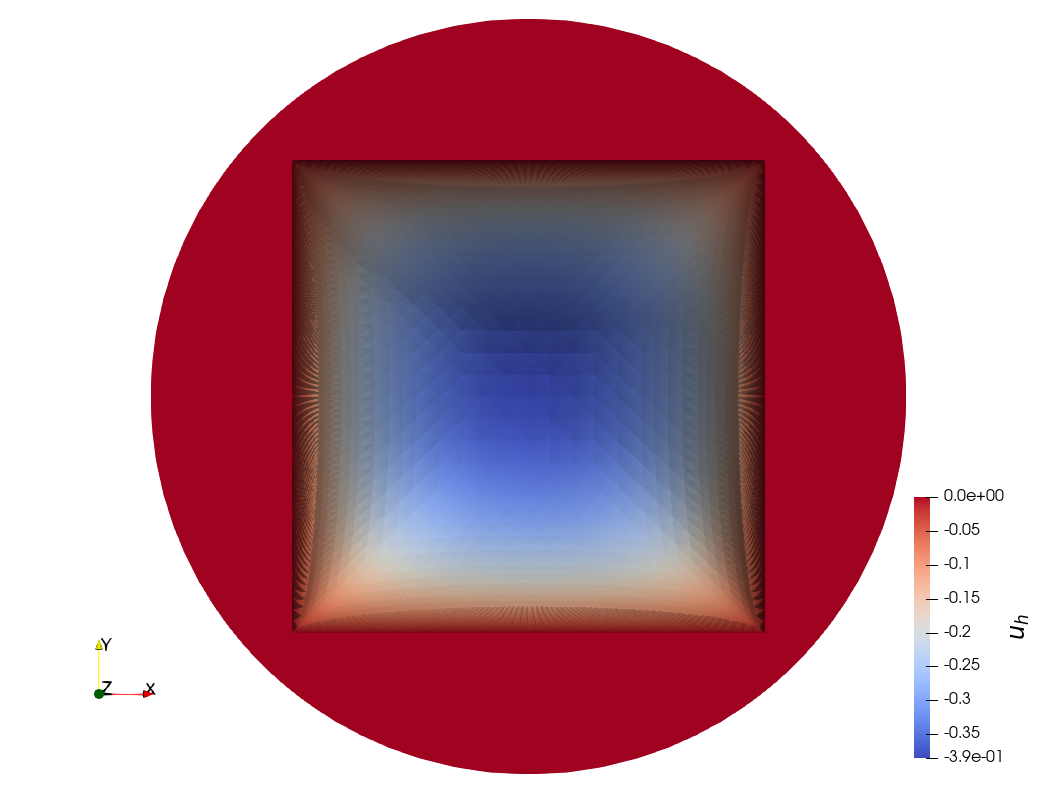} 
			&
			\includegraphics[width=0.42\linewidth]{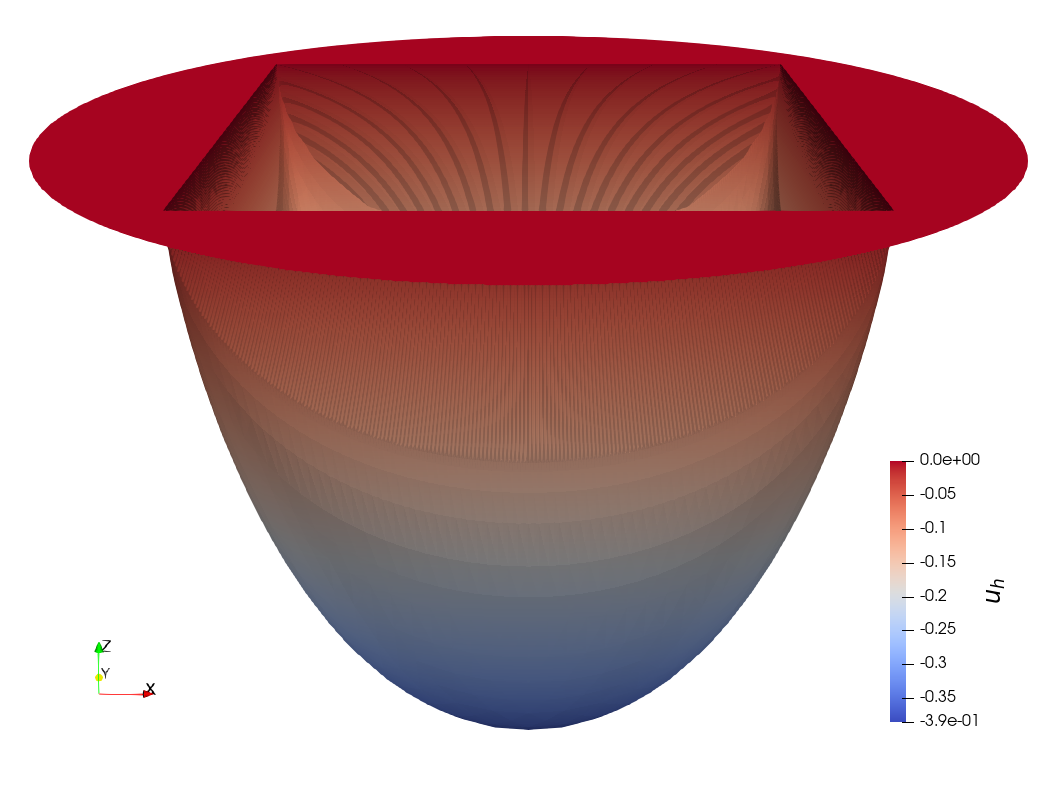} \\
			\includegraphics[width=0.42\linewidth]{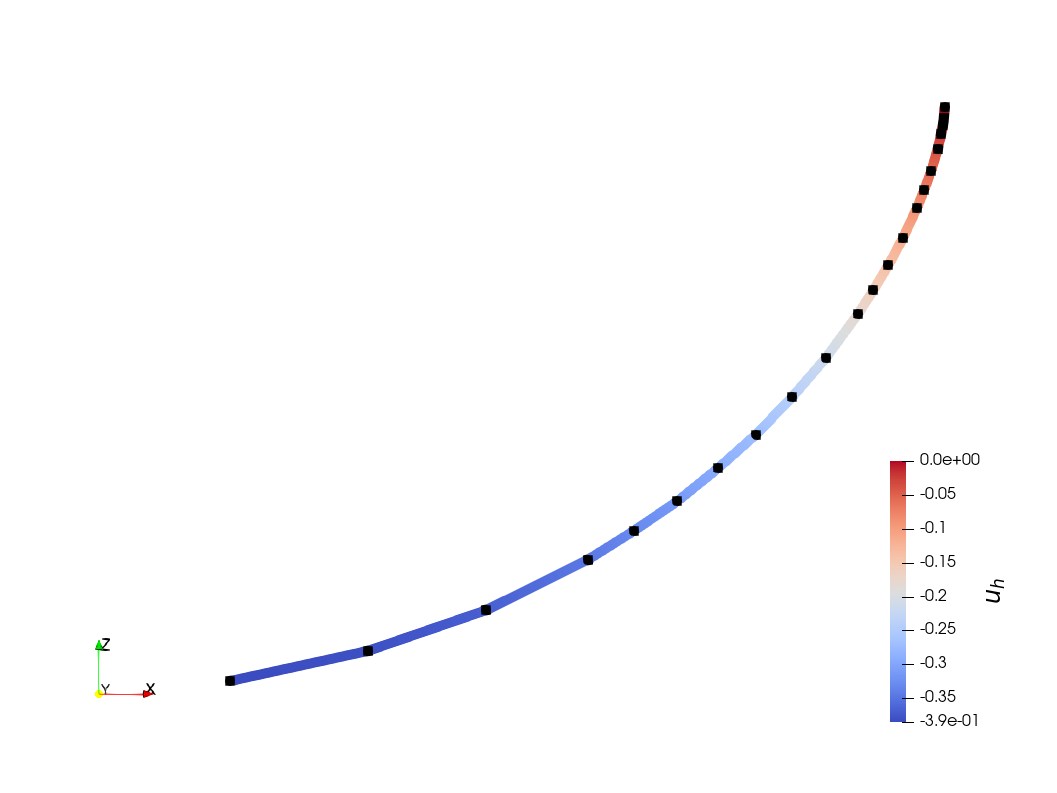} 	&	\includegraphics[width=0.42\linewidth]{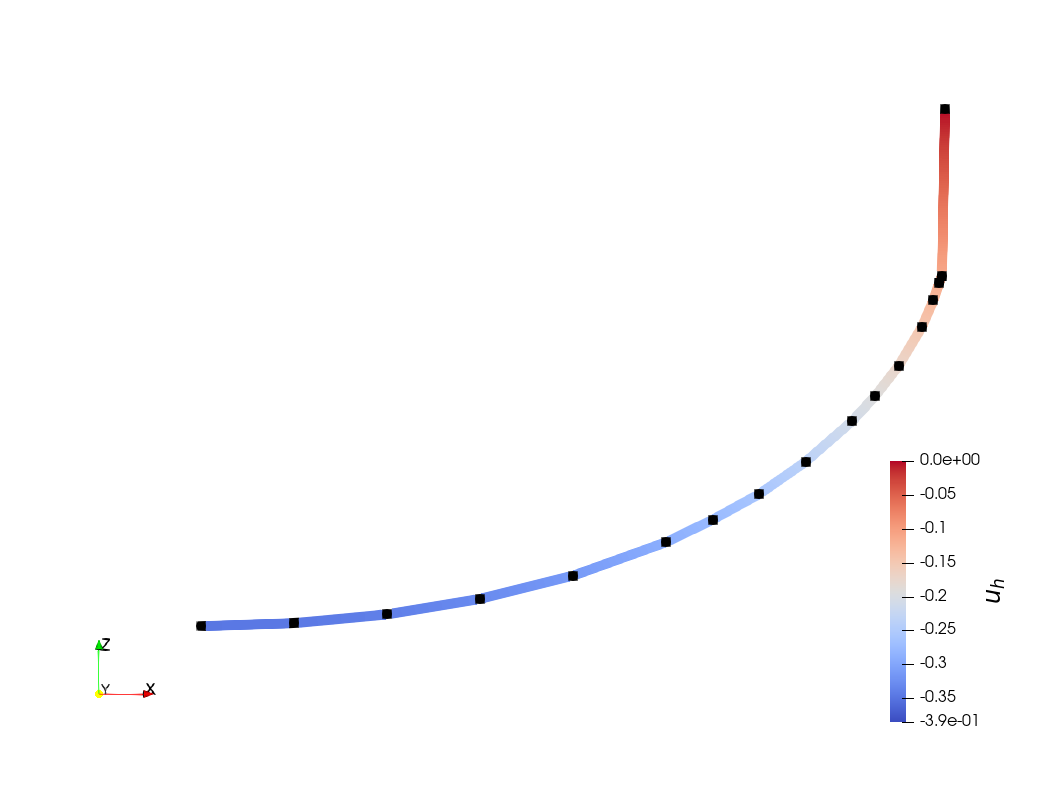}
		\end{tabular}
	\end{center}
	\caption{ Functions with prescribed fractional mean curvature $f = -1$ in $\Omega$ that vanish in $\Omega^c$. The domain $\Omega$ is either an annulus (top row) or a square (middle and bottom row). The plot in the top-right panel corresponds to a radial slice ($y = 0, \,\,  0.25 \le x \le 1$) of the annulus, while the ones in the bottom-left and bottom-right show parts of slices along the diagonal ($0 \le y = x \le 1$) and perpendicular to an edge of the square ($y = 0.5, 0 \le x \le 1$), respectively. The discrete solutions exhibit a much steeper gradient near the concave portions of the boundary than near the convex ones, and they seem to be continuous in the corners of the square.}
	\label{fig:curvature_square}
\end{figure}
\begin{figure}
	\begin{center}
		\begin{tabular}{c c} 
			\includegraphics[width=0.42\linewidth]{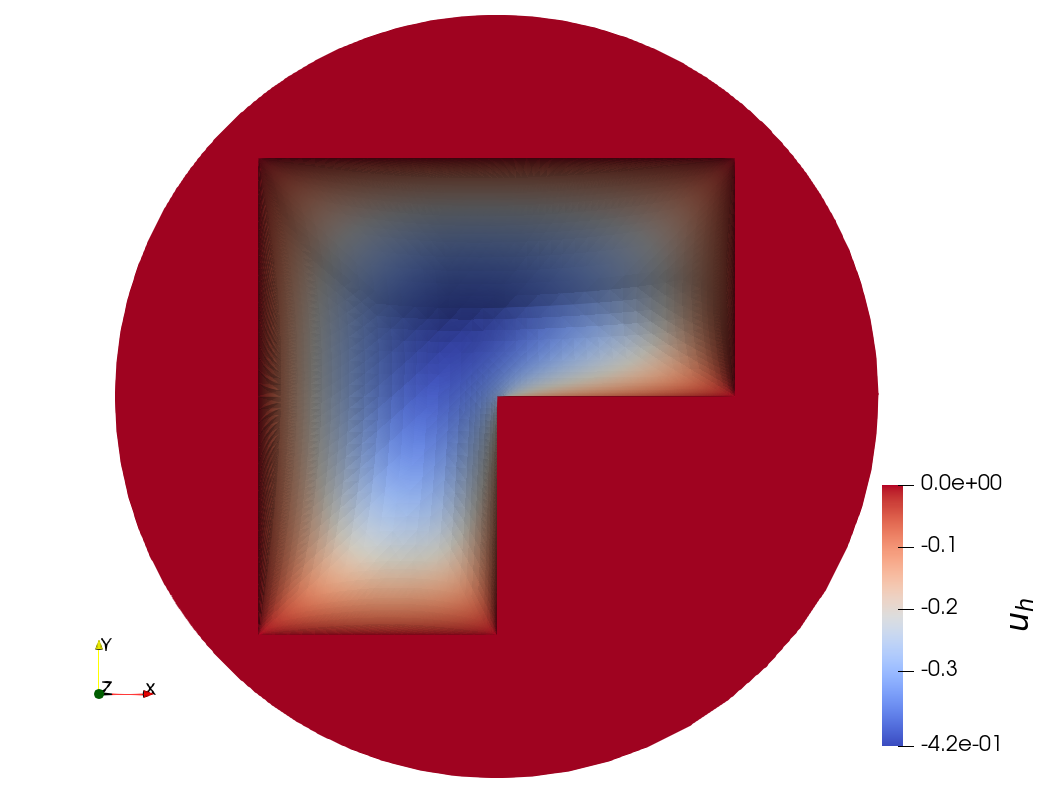}& 
			\includegraphics[width=0.42\linewidth]{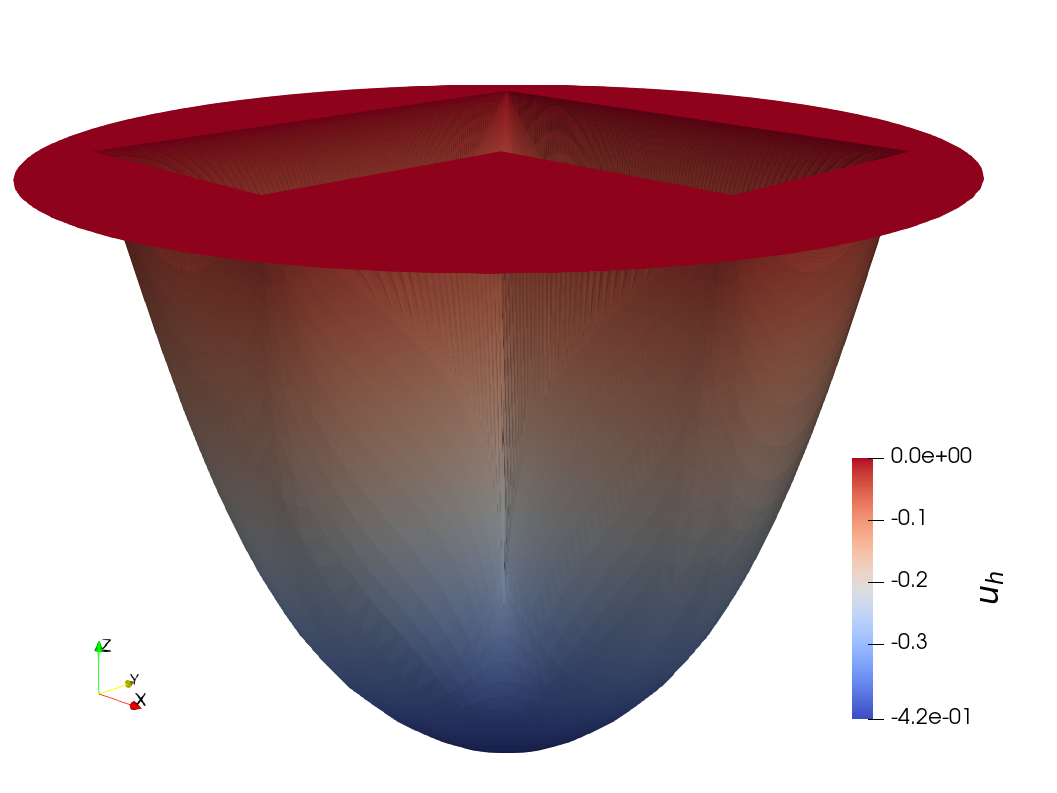}\\
			\includegraphics[width=0.42\linewidth]{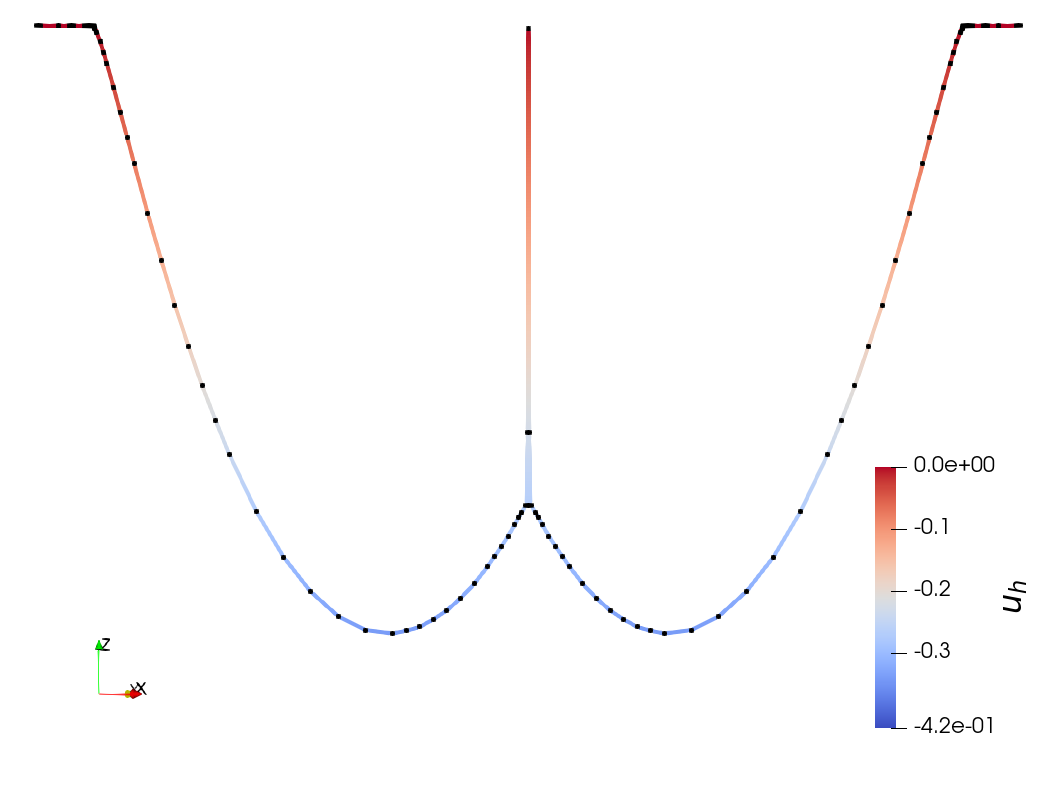} 
			&
			\includegraphics[width=0.42\linewidth]{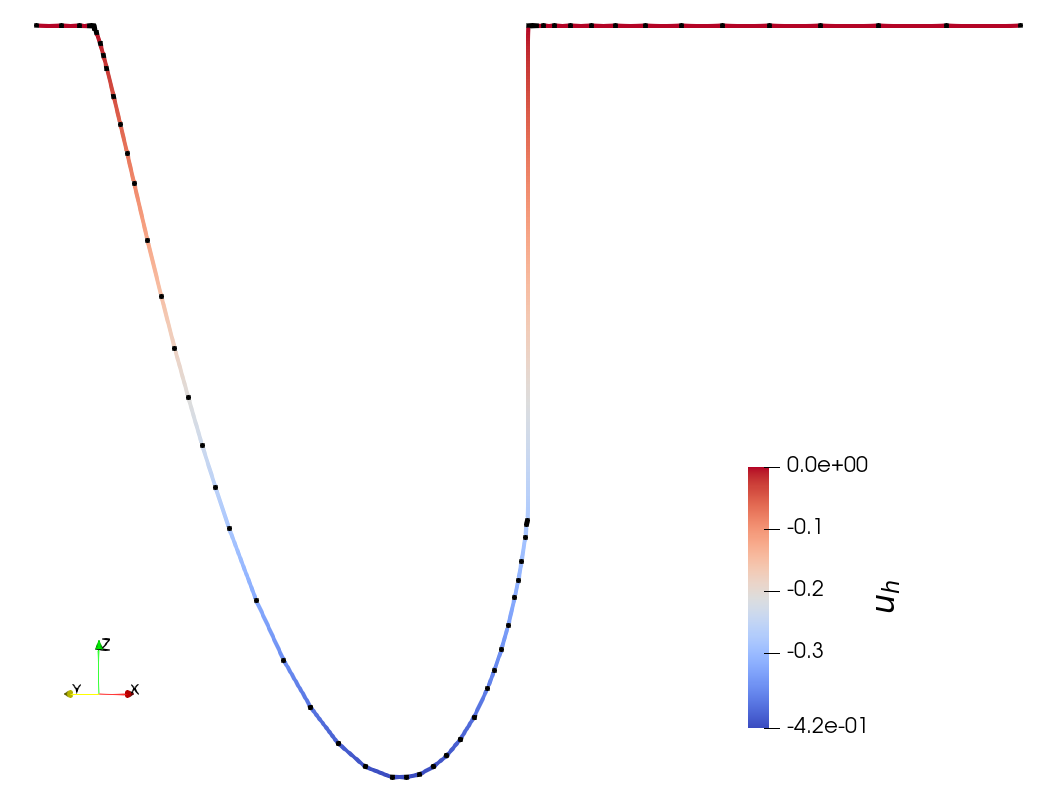} \\
			\includegraphics[width=0.42\linewidth]{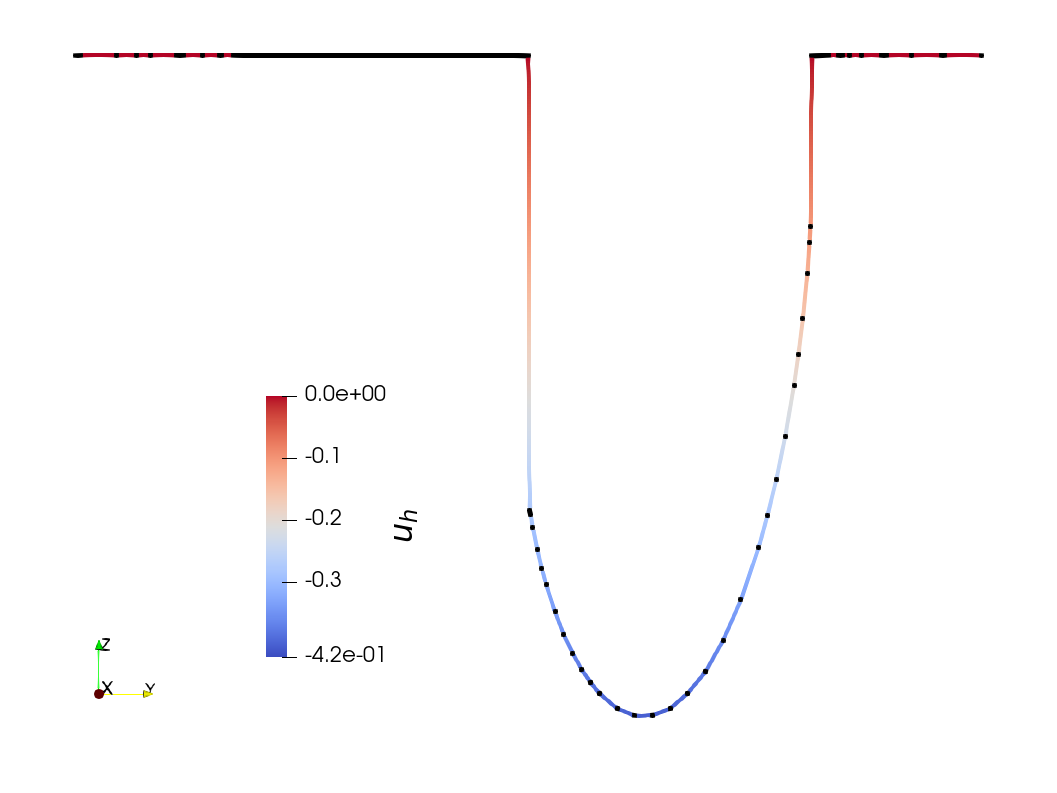} 	&	\includegraphics[width=0.42\linewidth]{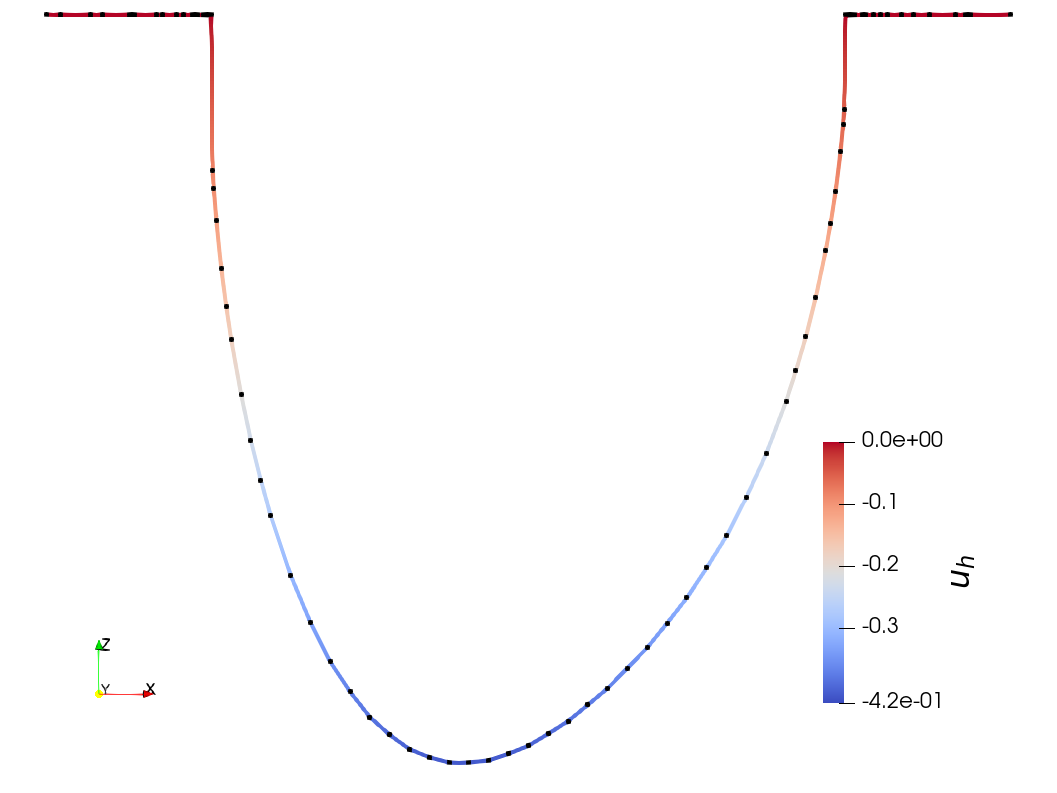}
		\end{tabular}
	\end{center}
	\caption{Stickiness on the L-shaped domain $\Omega = (-1,1)^2 \setminus (0,1)\times (-1,0)$ with prescribed fractional mean curvature $f=-1$ in $\Omega$ and zero Dirichlet condition in $\Omega^c$. The plots in the center row correspond to slices along $y = x$ and $y = -x$, respectively, while the ones in the bottom are slices along $x = 0$ or $y = 0.5$ respectively. We observe evidence of stickiness at the reentrant corner, while such a phenomenon is absent at the convex corners.
} 
	\label{fig:curvature_Lshape}
\end{figure}

\begin{example}[effect of boundary curvature on nonlocal minimal graphs]
We present examples of graphs with prescribed nonlocal mean curvature in three two-dimensional domains $\Omega$
with qualitatively different $\partial\Omega$ and examine the impact on stickiness.
We fix data $g = 0$ and $f = -1$ in \eqref{eq:quasilinear-weak}, and solve \eqref{E:def-a-NMS}--\eqref{eq:Gts}. We first consider the annulus $\Omega = B(0,1) \setminus B(0,1/4)$ and $s=0.25$. The top row in \Cref{fig:curvature_square} depicts a top view of the discrete solution $u_h$ and a radial slice of it. We observe that the discrete solution is about three times stickier in the inner boundary than in the outer one. The middle and bottom row in \Cref{fig:curvature_square} display different views of the solution in the square $\Omega = (-1,1)^2$ for $s = 0.01$. Near the boundary of the domain $\Omega$, we observe a steep slope in the middle of the edges; however, stickiness is not observed at the convex corners of $\Omega$.
We finally investigate stickiness at the boundary of the L-shaped domain $\Omega = (-1,1)^2 \setminus (0,1)\times (-1,0)$ with $s = 0.25$. We observe in \Cref{fig:curvature_Lshape} that stickiness is most pronounced at the reentrant corner but is again absent at the convex corners of $\Omega$.
\end{example}

The previous example indicates that there is a connection between the jump of solutions across the boundary $\partial \Omega$ and the curvature of $\partial \Omega$. Stickiness is stronger at the concave portions of the boundary than at the convex ones. Moreover, we find no numerical evidence of stickiness at convex corners. We refer to \cite[\S 7.6.3]{BoLiNo21} for an heuristic explanation supporting these observations.

%\medskip
Next, we perform some experiments involving the fractional $(p,s)$-Laplacian $(-\Delta)_p^s$.

\begin{example}[dependence of $(-\Delta)_p^s$ with respect to $s$ and $p$]
Let $\Omega = (-0.5, 0.5) \subset \R$, and set $f = 1$ in \eqref{eq:frac-pLap-weak}, with the form given by \eqref{E:def-a-pLap}. We use uniform grids with mesh size $h = 2^{-11}$. In \Cref{fig:frac-ps-Lap-1d-multi-s}, we exhibit numerical solutions for different choices of $s$ and $p$. We observe that the boundary behavior depends on $s$ but is independent of $p$. 
\begin{figure}[!htb]
	\begin{center}
		\includegraphics[width=0.45\linewidth]{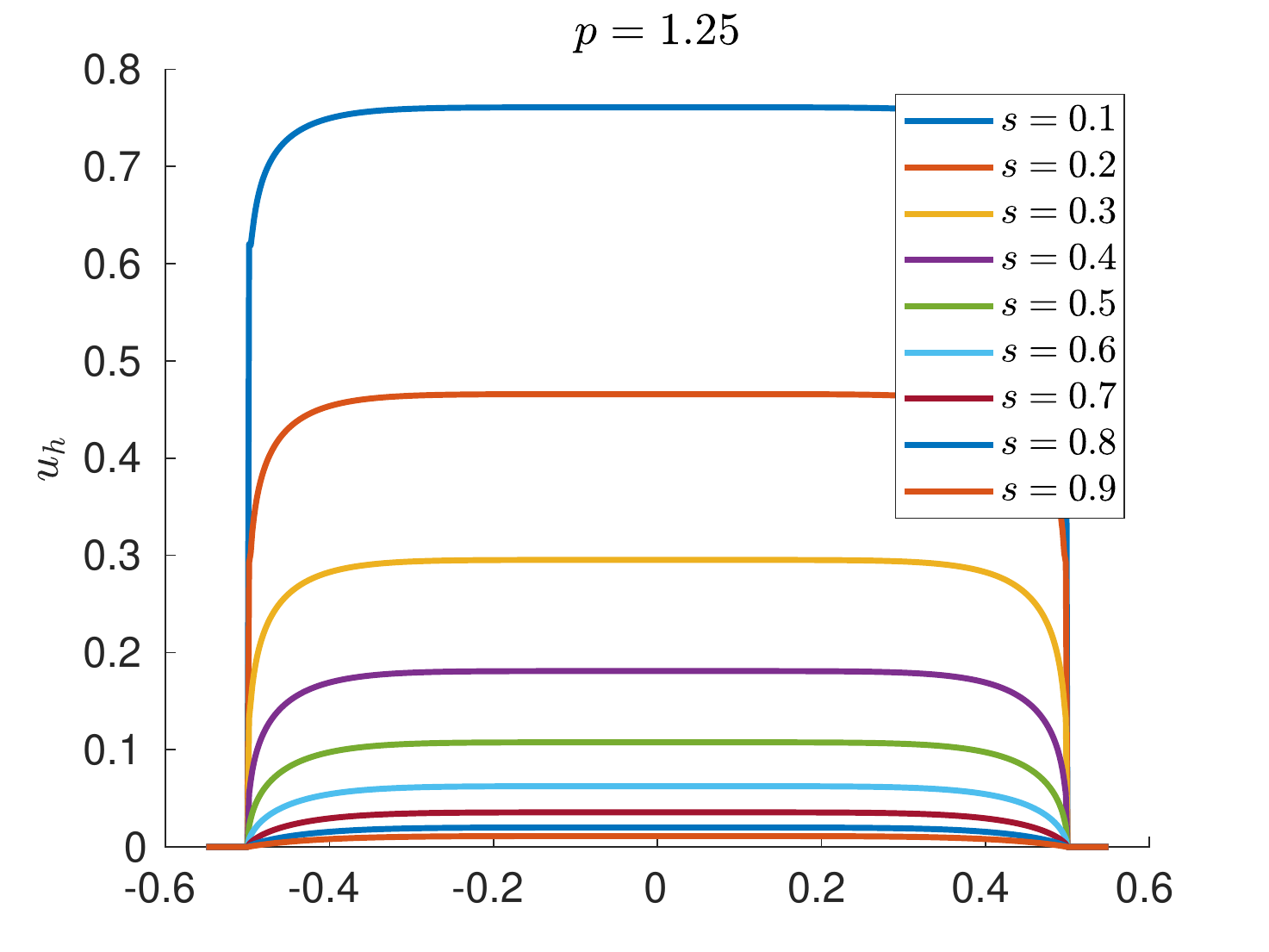}
		\includegraphics[width=0.45\linewidth]{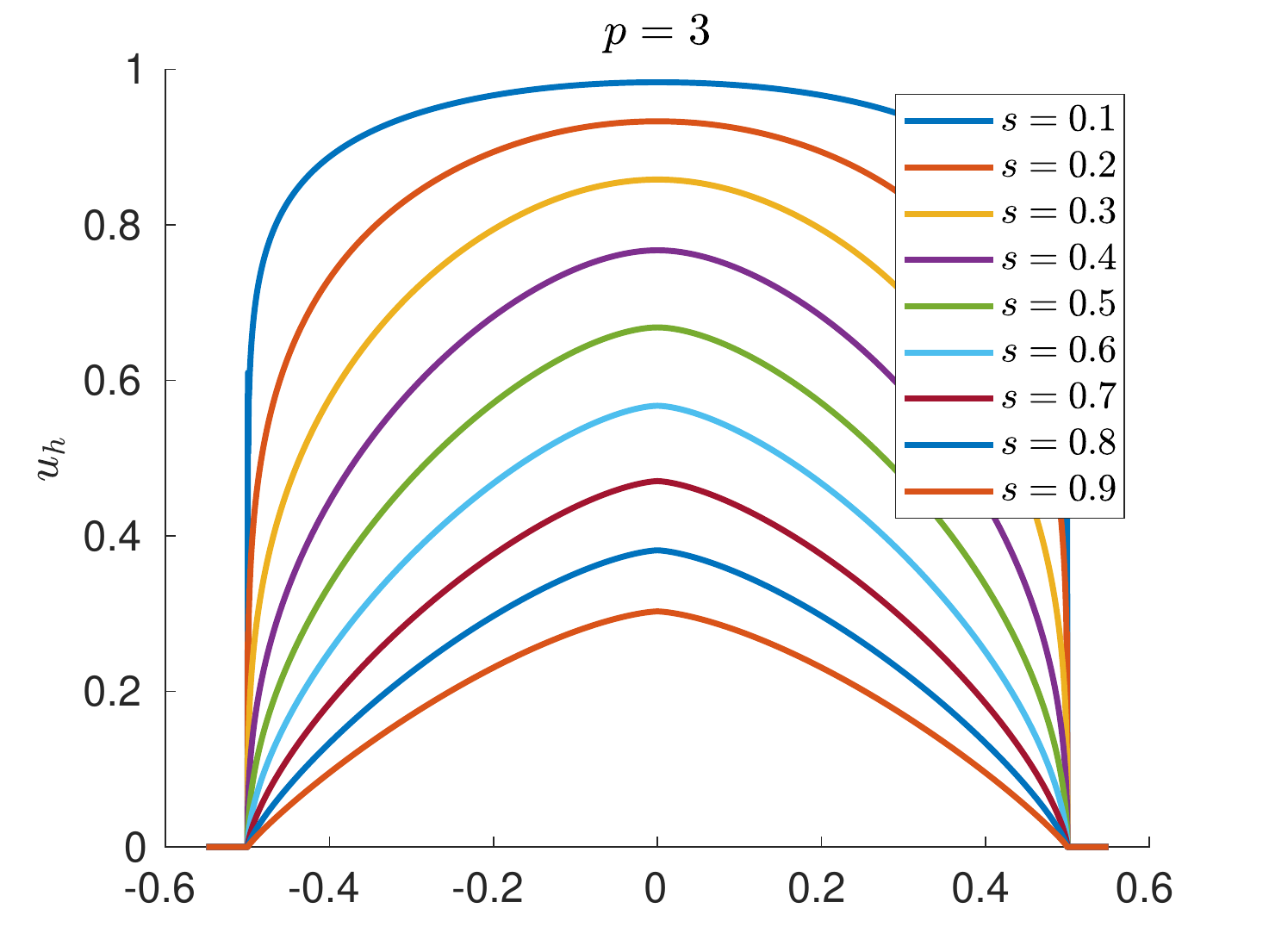} \\
		\includegraphics[width=0.45\linewidth]{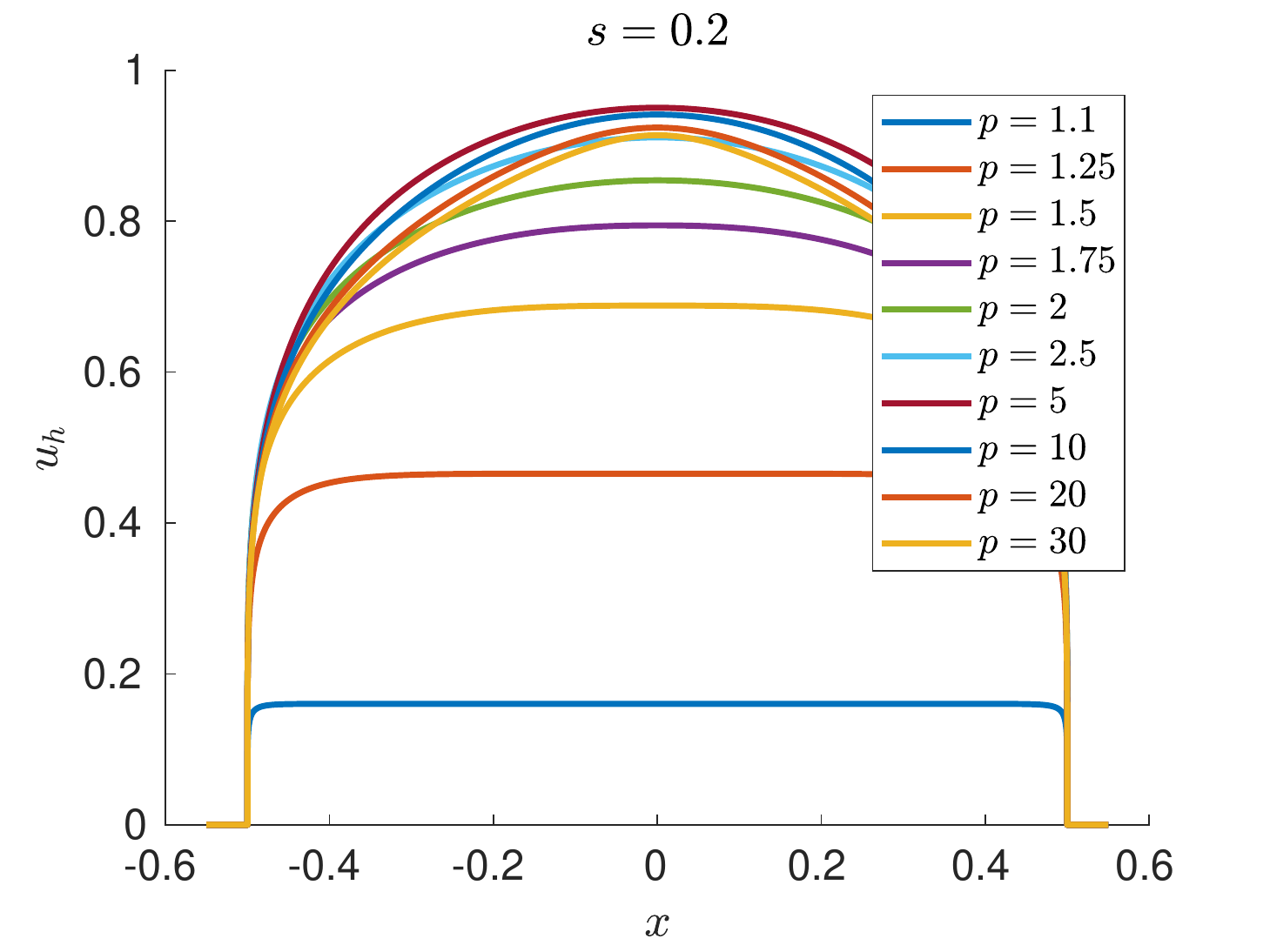}
		\includegraphics[width=0.45\linewidth]{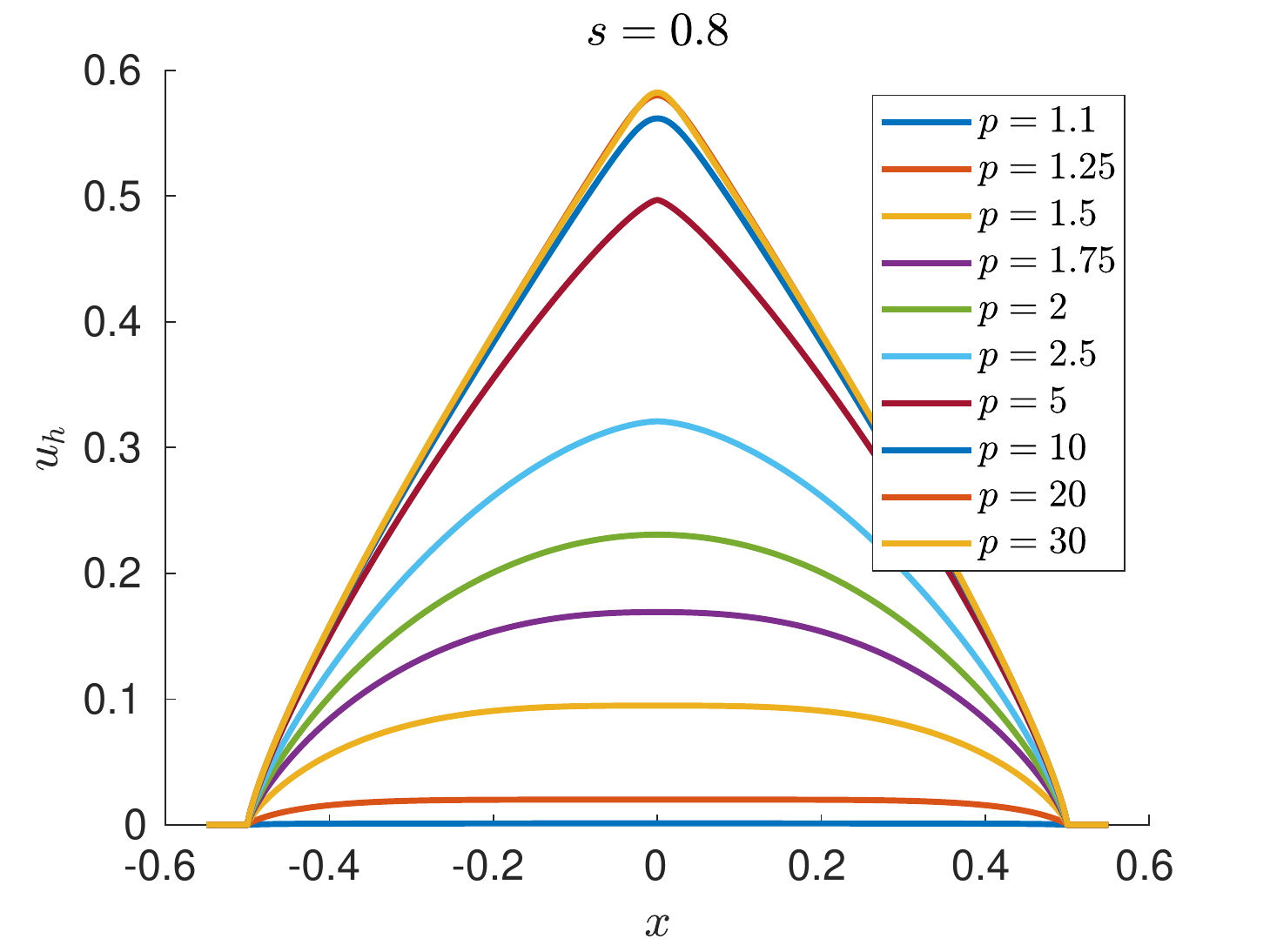}
	\end{center}
	\caption{Numerical solutions of fractional $(p,s)$-Laplacians in $\Omega = (-0.5, 0.5)$ for $f = 1$. The boundary behavior of solutions depends on $s$ but not on $p$.} 
	\label{fig:frac-ps-Lap-1d-multi-s}
\end{figure}

We also measure the convergence rates in the $\widetilde{W}^s_p(\Omega)$ norm on uniform meshes and meshes graded according to \eqref{eq:H} with $d=1$ and grading parameter $\mu = p(2-s)$. We now justify the choice of $\mu$ by the expected boundary behavior
\[
u(x) \approx d(x,\partial\Omega)^s,
\]
and corresponding second derivative of $u$, obtained heuristically and similarly to \eqref{eq:RL-derivative},
\[
|u '' (x)| \approx d(x,\partial\Omega)^{s-2}.
\]
Defining weighted Sobolev spaces $W^2_{p,\kappa}(\Omega)$ in the spirit of \eqref{eq:weighted_sobolev} (modifying the integrability index accordingly from $2$ to $p$), we would then expect to have 
\[
u \in W^2_{p,\alpha}(\Omega) \quad \mbox{if } \alpha > 2 - s - \frac1p.
\]
Revisiting \Cref{L:interpolation-error} with this regularity and constructing meshes according to \eqref{eq:H}, we immediately find that the optimal (in terms of $h$) interpolation estimate
\[
\| u - \Pi u \|_{\wt{W}^s_p(\Omega)} \le C h^{2-s} |u|_{ W^2_{p,\alpha}(\Omega)}
\]
holds provided $\mu \ge p(2-s)$. By \eqref{eq:dofs} with $d=1$, we know already that $\dim \mathbb{V}_h \approx h^{-1}$ regardless of the value of $\mu$, and we therefore expect the interpolation error to be of order $2-s$ in such a case. In contrast, if $d>1$ mesh shape-regularity limits the optimal range of $\mu$ in \eqref{eq:H} and one cannot expect to recover optimal interpolation error estimates unless some additional conditions on $p$ and $s$ are satisfied.

Table \ref{table:frac-ps-Lap-1d-rates-mu2-s} records convergence rates on uniform and graded meshes for $p=1.25$, $p=3$, and a range of values of $s$. On uniform meshes, the energy-norm error decays with order approximately $h^{1/p}$. This rate is consistent with the interpolation error corresponding to the regularity of solutions derived in \Cref{T:Besov_regularity_pLap} for $p\ge2$, but suggests that our result for $p<2$ is non-optimal in the sense that the best regularity one can expect given the boundary behavior \eqref{eq:RL-derivative} is $W^{s + 1/p - \eps}_p(\Omega)$. Moreover, the rates in \Cref{table:frac-ps-Lap-1d-rates-mu2-s} are better than the convergence rate predicted by Theorem \ref{thm:FE-error-pLap-rate} for a general $f \in B^{-s+q}_{p',1}(\Omega)$, $q = \max\{1/p', 1/2\}$.
In contrast, on graded meshes we observe convergence with order $h^{2-s}$ as expected unless for $p =3$ and $s$ close to $1$. Let us provide a possible explanation for this issue. In the local case ($s=1$), the solution corresponding to $f = 1$ is
\[
u(x) = C \left(1 - \left| 2x \right|^{\frac{p}{p-1}} \right)_+,
\]
which is locally of class $W^2_p$ only if $\frac{p}{p-1} + \frac1p \ge 2$. Thus, for the local problem, if $p \ge (3+\sqrt{5})/2 \approx 2.62$ we do not expect to recover optimal convergence rates in the interior unless some adaptive refinement near the origin is performed. In the fractional-order case, we expect a similar behavior to occur when $s$ is close to $1$, and this may justify the reduced rates for $p=3$ and $s > 0.5$ in \Cref{table:frac-ps-Lap-1d-rates-mu2-s}. 

\begin{table}[htbp]
\centering\scalebox{0.95}{
	\begin{tabular}{c |c| c| c| c| c| c| c| c| c| c|} \cline{2-11}
		&	Value of $s$ & 0.1 & 0.2 & 0.3 & 0.4 & 0.5 & 0.6 & 0.7 & 0.8 & 0.9 \\ \hline
		\multicolumn{1}{|c|}{\multirow{2}{*}{Uniform $\Th$}}	&		$p = 1.25$& 0.796 & 0.777 & 0.780 & 0.788 & 0.797 & 0.807 & 0.821 & 0.85 &  0.898 \\ \cline{2-11}
		\multicolumn{1}{|c|}{}	&	$p = 3$ & 0.337 & 0.334  & 0.332 & 0.332  &  0.333 & 0.333  & 0.334 & 0.335 & 0.339  \\  \hline  
		\multicolumn{1}{|c|}{\multirow{2}{*}{Graded $\Th$}}	&		$p = 1.25$ & 1.848 & 1.747  & 1.634 & 1.523  & 1.417 & 1.313 & 1.216 & 1.128 & 1.053  \\  \cline{2-11}
		\multicolumn{1}{|c|}{}	&	$p = 3$ & 1.880 & 1.790  & 1.673 & 1.579  & 1.475  & 1.306 & 1.141 & 1.011 & 0.906  \\  \hline  				
	\end{tabular}
}
\bigskip
\caption{Convergence rates $\alpha$ for $\Vert u - u_h \Vert_{\wt{W}^s_p(\Omega)}$ for $p = 1.25, 3$ and different $s$ on uniform mesh and graded mesh with $\mu = p(2-s)$.}
\label{table:frac-ps-Lap-1d-rates-mu2-s}
\end{table}
\end{example}

Furthermore, for this specific example, we modify the mesh grading \eqref{eq:H} as follows:
\begin{equation}\label{eq:H-special}
h_T \leq C(\sigma) \left\lbrace
\begin{array}{rl}
h^\mu & \mbox{if } T \cap \left\{\pp \Omega, 0 \right\} \neq \emptyset, \\
h \, \dist(T, \left\{\pp \Omega, 0\right\})^{(\mu-1)/\mu} & \mbox{if } T \cap\left\{\pp \Omega, 0 \right\} = \emptyset.
\end{array} \right.
\end{equation}
Above, $\mu \ge 1$ is a parameter and a simple calculation yields the mesh cardinality 
\begin{equation*}
\#\calT \approx \dim \mathbb{V}_h \approx h^{-1}
\end{equation*}
because of $d = 1$. Let us comment that the grading near the origin $x = 0$ seems to be an overkill because the expected singularity there is weaker than the one near the boundary $\pp \Omega$. In Table \ref{table:frac-ps-Lap-1d-rates-mu2-s-inner}, we report the convergence rates for $p = 1.25, 3$ on meshes satisfying \eqref{eq:H-special} with $\mu = p(2-s)$. The rates for $p = 3$ and $s$ close to $1$ are now in good agreement with the predicted rate $2 - s$, supporting our claim about the singularity of solutions near $x = 0$ in this example.

\begin{table}[htbp]
	\centering\scalebox{0.95}{
		\begin{tabular}{c |c| c| c| c| c| c| c| c| c| c|} \cline{2-11}
			&	Value of $s$ & 0.1 & 0.2 & 0.3 & 0.4 & 0.5 & 0.6 & 0.7 & 0.8 & 0.9 \\ \hline 
			\multicolumn{1}{|c|}{\multirow{2}{*}{Graded $\Th$}}	&		$p = 1.25$ & 1.846 & 1.735  & 1.623 & 1.514  & 1.411 & 1.308 & 1.213 & 1.125 & 1.052 \\  \cline{2-11}
			\multicolumn{1}{|c|}{}	&	$p = 3$ & 1.895 & 1.762  & 1.653 & 1.549  & 1.446  & 1.345 & 1.250 & 1.160 & 1.085  \\  \hline  				
		\end{tabular}
	}
	\bigskip
	\caption{Convergence rates $\alpha$ for $\Vert u - u_h \Vert_{\wt{W}^s_p(\Omega)}$ for $p = 1.25, 3$ and different $s$ on graded mesh with $\mu = p(2-s)$ and mesh grading satisfying \eqref{eq:H-special}. We observe an improvement with respect to \Cref{table:frac-ps-Lap-1d-rates-mu2-s} for $p = 3$ and $s > 0.5$.}
	\label{table:frac-ps-Lap-1d-rates-mu2-s-inner}
\end{table}

Our last experiment involves the fractional $(p,s)$-Laplacian on an $L$-shaped domain in two-dimensional space.

\begin{example}[$(p,s)$-Laplacian on $L$-shaped domain]
	Let $\Omega = (-0.5, 0.5)^2 \setminus [(0, 0.5) \times (-0.5, 0)] \subset \R^2$ and $f = -1$. \Cref{fig:frac-ps-Lap-2d-Lshape} depicts numerical solutions for $s = 0.2, 0.8$ and $p = 1.25, p = 3$. For a fixed $p$, we observe a much stronger boundary behavior for $s=0.2$ than for $s=0.8$, while for a fixed $s$ the solutions converge to the distance to boundary to the power $s$ as $p \to \infty$.	
\end{example}

\begin{figure}[!htb]
	\begin{center}
		\includegraphics[width=0.45\linewidth]{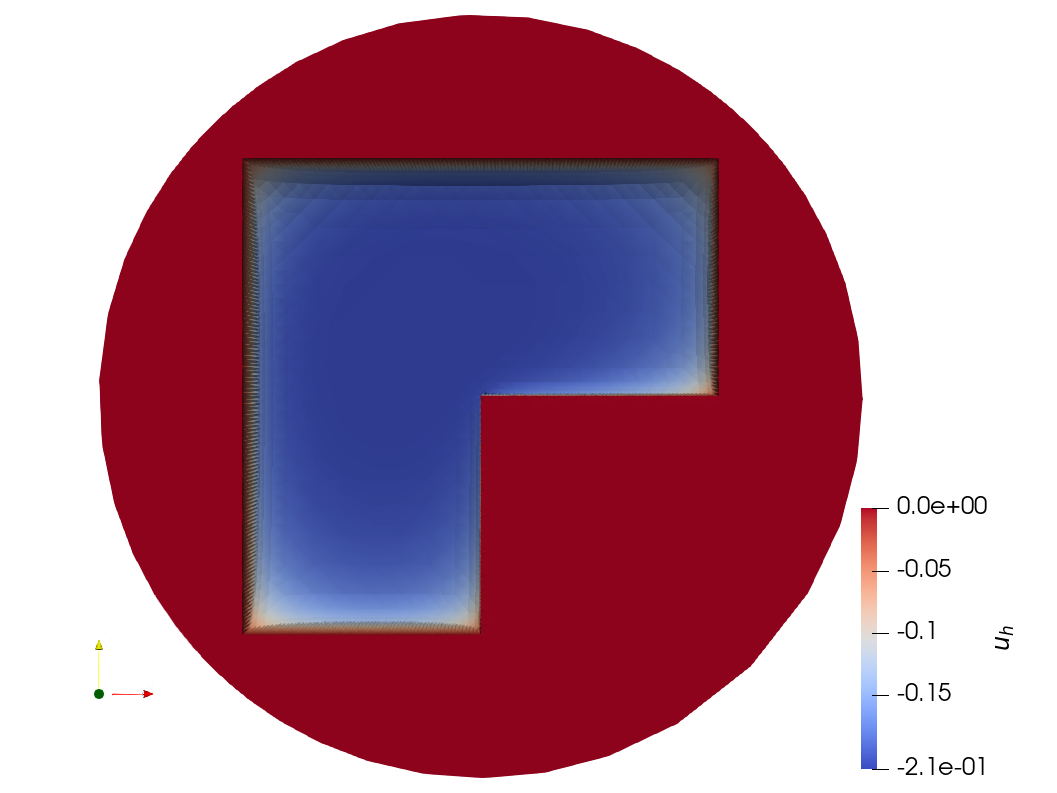}
		\includegraphics[width=0.45\linewidth]{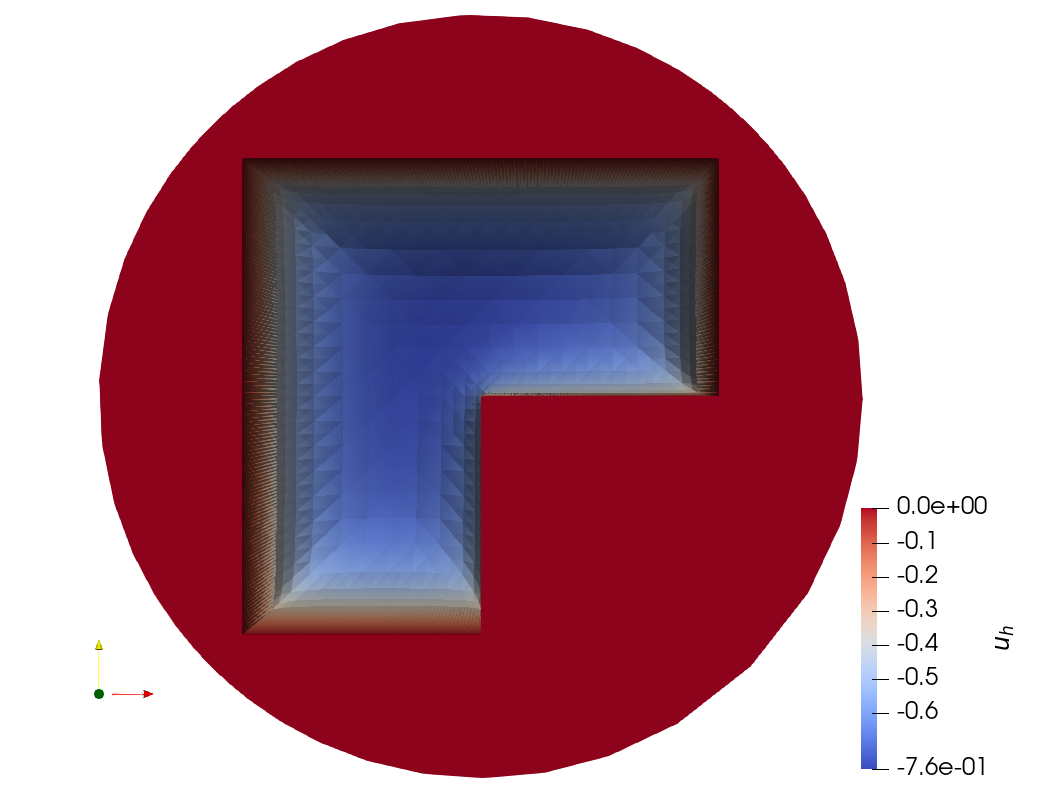} \\
		\includegraphics[width=0.45\linewidth]{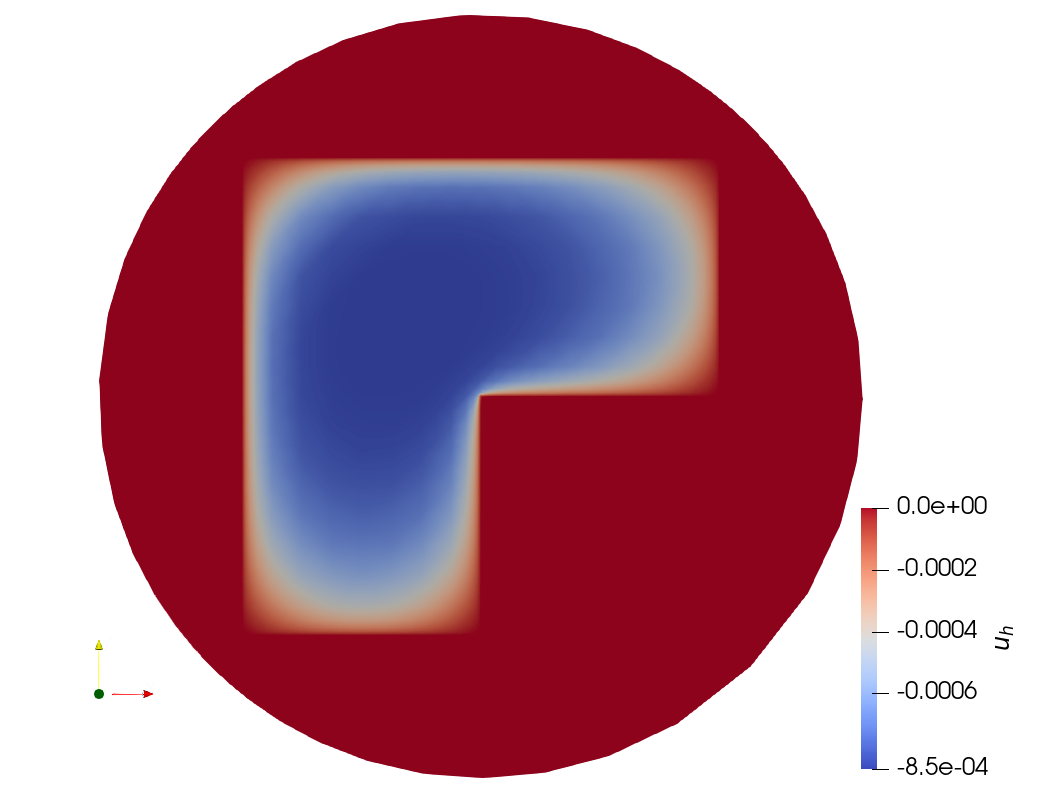}
		\includegraphics[width=0.45\linewidth]{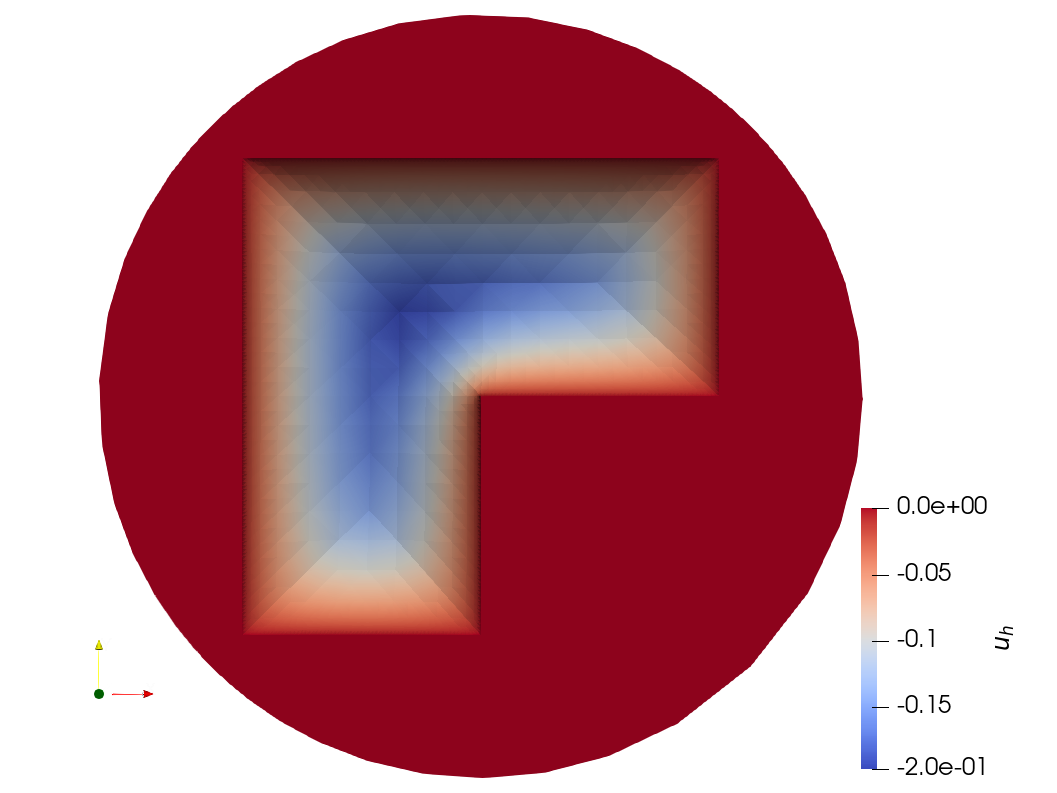}
	\end{center}
	\caption{Numerical solutions of fractional $(p,s)$-Laplacians in $\Omega = (-0.5, 0.5)^2 \setminus [(0, 0.5) \times (-0.5, 0)] \subset \R^2$ for $f = -1$. Top left: $p = 1.25, s = 0.2$; top right: $p = 3, s = 0.2$; bottom left: $p = 1.25, s = 0.8$; bottom right: $p = 3, s = 0.8$. 
	We observe a stronger boundary behavior as $s \searrow 0$, while the solution approaches $d(\cdot,\partial\Omega)^s$ as $p \nearrow \infty$.
	} 
	\label{fig:frac-ps-Lap-2d-Lshape}
\end{figure}

%---------------------------------------------
%---------------------------------------------
\bibliography{fractional}{}
\bibliographystyle{plain}
%---------------------------------------------
%---------------------------------------------

%%%%%%%%%%%%%%%%%%%%%%%%%%%%%%%%%%%%%%%%%%%%%%%%%%%%%%%%%%%%%%%%%%%%%%%%%%%%%%%%%%%
\end{document}